\documentclass[11pt]{amsart}
\usepackage{amsmath,amssymb}
\usepackage{mathabx}
\usepackage{mathdots}
\usepackage{graphicx,caption,subfig}
\usepackage{color,soul}
\usepackage{hyperref}
\usepackage{enumerate}
\usepackage{enumitem}
\usepackage[normalem]{ulem}
\usepackage[ruled,vlined,linesnumbered,algosection]{algorithm2e}

\usepackage{fullpage}

\newtheorem{theorem}{Theorem}[section]
\newtheorem{lemma}[theorem]{Lemma}
\newtheorem{definition}[theorem]{Definition}

\newtheorem{proposition}[theorem]{Proposition}
\newtheorem{remark}[theorem]{Remark}
\newtheorem{example}[theorem]{Example}

\newcommand{\tab}{\hspace*{2em}}

\numberwithin{equation}{section}

\DeclareMathOperator{\sinc}{sinc}
\DeclareMathOperator{\Arg}{Arg}

\def\P{{\cal P}}
\def\barP{\bar{{\cal P}}_d}
\def\F{{\cal F}}
\def\EO{E_{\epsilon,\Omega}(F)}
\def\EOR{E_{R,\Omega}(F)}
\def\EG{E_{\epsilon,(\lambda)}(F)} 
\def\EGR{E_{R,(\lambda)}(F)} 
\def\R{\mathbb{R}}
\def\C{\mathbb{C}}
\def\Z{\mathbb{Z}}
\def\NN{\mathbb{N}}
\def\FMG{FM_{\lambda}}

\def\gs{\lambda^{*}}
\def\EGRS{E_{R,(\lambda^*)}(F)}

\def\SRF{\textrm{SRF}}

\renewcommand{\vec}[1]{\boldsymbol{\mathrm{#1}}}
\newcommand{\xv}{\vec{x}}
\newcommand{\av}{\vec{a}}
\newcommand{\lv}{\vec{\ell}}

\newcommand{\mm}{{\cal E}} 

\newcommand{\est}{\frak{f}} 
\newcommand{\ecl}{\frak{F}} 

\makeatletter
\DeclareRobustCommand*\cal{\@fontswitch\relax\mathcal}
\makeatother

\begin{document}
\title[Super-resolution of near-colliding point
sources]{Super-resolution of near-colliding point sources}

\author[D.Batenkov]{Dmitry Batenkov} \address{Department of Applied
  Mathematics, School of Mathematical Sciences, Tel-Aviv University,
  P.O. Box 39040, Tel-Aviv 6997801, Israel}
\email{dbatenkov@tauex.tau.ac.il} \thanks{}

\author[G. Goldman]{Gil Goldman} \address{Department of Mathematics,
  The Weizmann Institute of Science, Rehovot 76100, Israel}
\email{gil.goldman@weizmann.ac.il} \thanks{}

\author[Y. Yomdin]{Yosef Yomdin} \address{Department of Mathematics,
  The Weizmann Institute of Science, Rehovot 76100, Israel}
\email{yosef.yomdin@weizmann.ac.il}

  \date{}

\begin{abstract}
  We consider the problem of stable recovery of sparse signals of the
  form
  $$F(x)=\sum_{j=1}^d a_j\delta(x-x_j),\quad x_j\in\mathbb{R},\;a_j\in\mathbb{C}, $$
  from their spectral measurements, known in a bandwidth $\Omega$ with
  absolute error not exceeding $\epsilon>0$.  We consider the case
  when at most $p\le d$ nodes $\{x_j\}$ of $F$ form a cluster whose
  extent is smaller than the Rayleigh limit ${1\over\Omega}$, while
  the rest of the nodes are well separated. Provided that
  $\epsilon \lessapprox \SRF^{-2p+1}$, where $\SRF=(\Omega\Delta)^{-1}$
  and $\Delta$ is the minimal separation between the nodes, we show
  that the minimax error rate for reconstruction of the cluster nodes
  is of order ${1\over\Omega}\SRF^{2p-1}\epsilon$, while for recovering
  the corresponding amplitudes $\{a_j\}$ the rate is of the order
  $\SRF^{2p-1}\epsilon$. Moreover, the corresponding minimax rates for
  the recovery of the non-clustered nodes and amplitudes are
  ${\epsilon\over\Omega}$ and $\epsilon$, respectively. These results
  suggest that stable super-resolution is possible in much more
  general situations than previously thought. Our numerical
  experiments show that the well-known Matrix Pencil method achieves
  the above accuracy bounds.
\end{abstract}

\maketitle

\section{Introduction}\label{sec.intro}
\subsection{Super-resolution of sparse signals}
The problem of mathematical super-resolution (SR) is to extract the
fine details of a signal from band-limited and noisy measurements of
its Fourier transform \cite{lindberg_mathematical_2012}. It is an
inverse problem of great theoretical and practical interest.

The specifics of SR highly depend on the type of prior information
assumed about the signal structure. Many theoretical and practical
studies assume \emph{signals of compact support}, in which case the SR
problem is equivalent to analytic continuation (equivalently,
extrapolation) of the Fourier transform. However, it can be shown that
the spectrum of a compactly supported function can be extrapolated
from samples of accuracy $\epsilon$ by a factor which scales at most
logarithmically with the signal-to-noise ratio ${1\over\epsilon}$, see
e.g. \cite{goodman_introduction_2005,lindberg_mathematical_2012,batenkov2019,bertero1998}
and references therein. On the other hand, in recent years
considerable progress has been made in studying SR for \emph{sparse
  signals}, which are frequently modelled as idealized {\it
  spike-trains}
\begin{equation}\label{eq.spike.train.signal}
  F(x)=\sum_{j=1}^d a_j\delta(x-x_j), \;x_j\in\R,
\end{equation}
where $\delta$ is the ubiquitous Dirac's $\delta$-distribution. This
particular type of signals is widely used in the literature, as it is
believed to capture the essential difficulty of SR with sparse priors,
see e.g. \cite{donoho1992superresolution,candes2014towards}.

Let ${\cal{F}}\left(F\right)$ denote the Fourier transform of $F$:
\begin{equation}\label{eq.fourier.tr}
{\cal F}(F)(s) = \int_{-\infty}^\infty F(x)e^{-2 \pi i s x}dx.
\end{equation}
Further suppose that the spectral data is given as a function $\Phi$
satisfying, for some $\epsilon>0$ and $\Omega>0$,
\begin{equation}\label{eq:noise-condition}
  \left| \Phi(s) - {\cal F}(F)(s) \right| \le \epsilon,\quad s\in[-\Omega,\Omega].
\end{equation}
The sparse SR problem reads as follows: given $\Phi$ as above,
estimate the unknown parameters of $F$, namely, the {\it amplitudes}
$\{a_j\}$ and the {\it nodes} $\{x_j\}$.

If $\epsilon=0$, the problem can be solved exactly by a variety of
parametric methods (Prony's method etc., see
e.g. \cite{prony1795essai,stoica_spectral_2005} and Subsection
\ref{sub:clustering-intro} below). For $\epsilon>0$, if $\est$ is any
reconstruction algorithm receiving $\Phi$ as an input, and producing
an estimate $F'=\est\left(\Phi\right)$ of the signal which satisfies
\eqref{eq:noise-condition}, then, under an appropriate definition of
the distance $\|F-F'\|$, it is of great interest to have a good
estimate of the noise amplification factor (or the \emph{problem
  condition number}) ${\cal K}$ such that
\begin{equation}\label{eq:noise.amplification}
  \|F-F'\| \approx {\cal K} \epsilon.
\end{equation}

\subsection{Rayleigh limit and minimal separation}
\label{sub:clustering-intro}

It has been well-established that the difficulty of sparse SR is
directly related to the \emph{minimal separation}
$\Delta=\min_{1\leq i<j\leq d} |x_i-x_j|$, or, more precisely, to the
relationship between $\Delta$ and $\Omega$.

Without any a-priori information, the best attainable resolution from
spectral data of bandwidth $\Omega$ is of the order
$\frac{1}{\Omega}$, which is also known as the \emph{Rayleigh
  limit}. Both classical methods of non-parametric spectral estimation
\cite{stoica_spectral_2005}, as well as modern convex optimization
based methods solve the problem under some sort of a separation
condition of the form $\Delta \geq {c\over\Omega}$
\cite{candes2014towards,candes2013super,fernandez2016super,heckel2016super,fernandez2013support,benedetto_super-resolution_2016,azais_spike_2015,bhaskar_atomic_2013,schiebinger_superresolution_2015,tang_near_2015},
and moreover these methods are generally considered to be stable.
  
\begin{figure}
  \centering
  \includegraphics[width=0.9\linewidth]{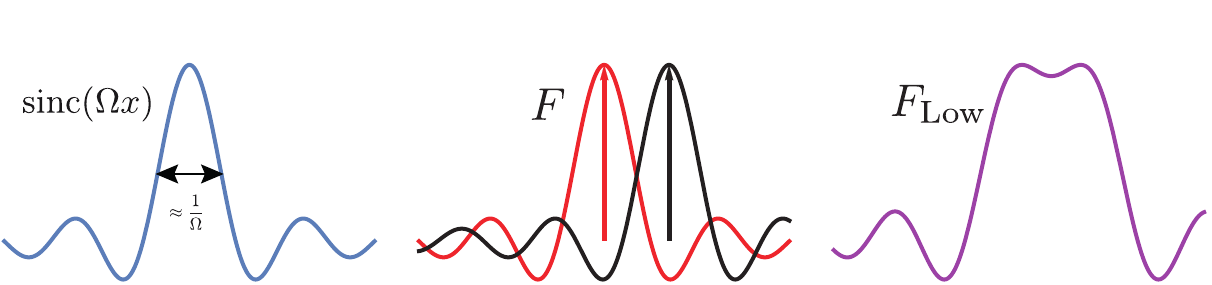}
  \captionsetup{singlelinecheck=off,font=small,width=\linewidth}
          \caption{\small The Rayleigh limit. For a signal
            $F(x) = \sum_{j} a_j \delta(x-x_j)$, its low resolution
            version is given by
            $$F_{\text{Low}}(x) = {\cal F}^{-1}\left({\cal
                  F}\left(F\right) \cdot
                \chi_{\left[-\Omega,\Omega\right]}\right) \asymp
              \sum_j a_j \sinc (\Omega (x-x_j)).$$ $F_{\text{Low}}(x)$
            will have peaks of width $\approx {1\over\Omega}$, and
            therefore it will be increasingly difficult to recover
            signals for which the minimal separation between the
            $\{x_j\}$'s is much smaller than ${1\over\Omega}$.}
  \label{fig:rayleigh}
\end{figure}
On the other hand, the case $\Delta \ll {1\over\Omega}$ (and arbitrary
signed/complex amplitudes $\{a_j\}$) is much more difficult (see
Figure \ref{fig:rayleigh}).  

The sparse SR problem has appeared
already in the work by R. Prony \cite{prony1795essai}, where he
devised an algebraic scheme to recover the parameters $\{x_j,a_j\}$
from $2d$ equispaced measurements of ${\cal F}(F)$, assuming $F$ is
given by \eqref{eq.spike.train.signal}, and for arbitrary $\Delta>0$
and $|a_j|>0$ (see Proposition \ref{prop.prony.rec.rel} below). Since
then, Prony's method and its various extensions and generalizations
have been used extensively in applied and pure mathematics and
engineering
(\cite{auton1981investigation,stoica_spectral_2005,pereyra_exponential_2010,peter2013generalized,plonka2014prony,vetterli2002sampling}
and references therein). While these methods provide exact recovery
for $\epsilon=0$, the question of their stability (the magnitude of
${\cal K}$ in \eqref{eq:noise.amplification}) becomes of essential
interest. For instance, if it so happens that an estimate
$F'=\sum_{j=1}^d a_j' \delta(x-x_j')$ satisfies
$\min_{1\leq j\leq d}|x_j'-x_j|\gtrapprox \Delta$, then such $F'$ may
be of little practical use in many applications (because the inner
structure of the sparse signal will be determined incorrectly).

\smallskip

The first work which examined the stability of SR in the sub-Rayleigh
regime was by D.Donoho \cite{donoho1992superresolution}. The signal
$F$ was assumed to have an infinite number of spikes $\{x_j\}$,
constrained to a grid of step size $\Delta$, with less than one spike
per unit interval on average, but whose local complexity was
constrained to have no more than $d$ spikes per any interval of length
$d$ (such $d$ is called the {\it Rayleigh index}). It was shown that the
worst-case $\ell_2$ error of such $F$ (i.e. the $\ell_2$ norm of the
coefficient sequence of the difference) from continuous measurements
with a band-limit $\Omega$ and perturbation of size $\epsilon$ (in
$L_2$ sense) scales like $\SRF^{\alpha}\epsilon$, where
$\SRF=\frac{1}{\Omega\Delta}>1$ is the so-called
\emph{super-resolution factor}, and $\alpha$ satisfies
$2d-1 \leq \alpha \leq 2d+1$. In \cite{demanet2015recoverability} the
authors considered the case of $d$-sparse signals supported on a grid,
and showed that the correct exponent should be $\alpha=2d-1$ in this
case. In another recent work \cite{li_stable_2017} the same scaling
was shown to hold in the case of $d$-sparse signals and discrete
Fourier measurements.

In the papers mentioned above, the error rate $\SRF^{2d-1}\epsilon$ is
\emph{minimax}, meaning that on one hand, it is attained by a certain
algorithm for all signals of interest, and on the other hand, there
exist worst-case examples for which no algorithm can achieve an
essentially smaller error. It turns out that these worst-case signals
all have the structure of a cluster, where all the $d$ nodes
$\{x_j\}$ appear consecutively, i.e.
$x_j = x_1 + (j-1)\Delta,\;j=1,\dots,d$. A natural question which
arises is: \emph{if it is a-priori known that only a subset of the $d$
  spikes can become clustered, can we have better reconstruction
  accuracy?} In this paper we shall provide a positive answer to this
question.

\subsection{Main contributions}
\label{sub:contributions}

In this paper we consider the case where the nodes $\{x_j\}$ can take
arbitrary real values (the so-called {\it off-grid} setting), while
the amplitudes $\{a_j\}$ can be arbitrary complex scalars. We further
assume that exactly $p$ nodes, $x_{\kappa},\ldots,x_{\kappa+p-1}$,
form a small cluster of extent $h\ll \frac{1}{\Omega}$ and are
approximately uniformly distributed inside the cluster, while the rest
of the nodes are well-separated from the cluster and from each other
(see Definition \ref{def.uniform.cluster} below). The approximate
uniformity is expressed by the assumption that the minimal separation
between any two cluster nodes is bounded from below by $\Delta=\tau h$
for some fixed $0<\tau\leq 1$. Under these {\it $p$-clustered}
assumptions, we show in Theorem \ref{thm.minmax.bounds}, that 
for small enough $\epsilon$ -- and, in particular, for
$\epsilon \lessapprox \left(\Omega \Delta\right)^{2p-1}$, the 
worst case error rates of a minimax reconstruction algorithm (see Definition \ref{def.minmax} below), 
receiving $\Phi$ satisfying \eqref{eq:noise-condition} as an input, and  
returning an estimate $x_j'=x_j'(\Phi)$, $a_j'=a_j'(\Phi)$, satisfy \footnote{We use the symbol $\asymp$ to
  denote order equivalence, up to constants: $A(t)\asymp B(t)$, if and
  only if there exist positive constants $c_1,c_2$ (depending on the
  specified parameters) such that $c_1B(t)\leq A(t) \leq c_2B(t)$ for
  all specified values of $t$.}
\begin{enumerate}
  \item Non-cluster nodes:
	\begin{align*}
		 \max_{j \notin \{\kappa,\dots,\kappa+p-1\}}  |x_j-x_j'| &\asymp \frac{\epsilon}{\Omega},\\
		  \max_{j \notin \{\kappa,\dots,\kappa+p-1\}}  |a_j-a_j'| &\asymp \epsilon.
	\end{align*}
  \item Cluster nodes:
  	\begin{align*}
		 \max_{j \in \{\kappa,\dots,\kappa+p-1\}}  |x_j-x_j'| &\asymp \frac{\epsilon}{\Omega} \left(\Omega \Delta\right)^{-2p+2},\\
		  \max_{j \in \{\kappa,\dots,\kappa+p-1\}}  |a_j-a_j'| &\asymp \epsilon \left(\Omega \Delta\right)^{-2p+1}.
	\end{align*}
\end{enumerate}

The constants appearing in our bounds depend on $p,d$, a-priori bounds
on the magnitudes $|a_j|$, and additional geometric parameters, but
neither on $\Delta$ nor on $\Omega$.

Our results indicate, in particular, that the non-clustered nodes
$\{x_j\}_{j\notin\left[\kappa,\dots,\kappa+p-1\right]}$ can be
recovered with much better accuracy than the cluster nodes. Let the
super-resolution factor be defined, as before, by
$\SRF=\left(\Omega \Delta\right)^{-1}$, then the condition number of the
cluster nodes scales like $\SRF^{2p-1}$ in the super-resolution regime
$\SRF\gg 1$, while the condition number of the non-cluster nodes does
not depend on the $\SRF$ at all.

Our approach is to reduce the continuous measurements problem to a
certain ``Prony-type'' system of $2d$ nonlinear equations, given by
equispaced measurements of $\Phi(s)$ with a carefully chosen spacing
$\lambda \approx \Omega$, and analyze the sensitivity of
this system to perturbations. The proofs involve techniques from
quantitative singularity theory and numerical analysis. Some of the
tools, in particular the ``decimation-and-blowup'' technique, were
previously developed in
\cite{akinshin2017error,batenkov2017accurate,paramertrization2017,batenkov2016stability,akinshin2015accuracy,batenkov2013geometry,batenkov2013accuracy,batenkov2018}. The
single-cluster case $p=d$ has been first analyzed in
\cite{batenkov2016stability}, while the lower bound (in a slightly
less general formulation) has been essentially shown in
\cite{akinshin2015accuracy}. One of the main technical results, Lemma
\ref{lemma.uniform.blowup}, has been first proven in
\cite{batenkov2018}.

Our numerical experiments in Section~\ref{sec:numerics} show that the
above bounds are attained by Matrix Pencil (MP), a well-known
high-resolution algorithm \cite{hua_svd_1991,hua_matrix_1990}.

\subsection{Related work and discussion}\label{sub:related-work}

Our main results generalize several previously available bounds for
both on-grid and off-grid SR
\cite{demanet2015recoverability,li_stable_2017,batenkov2016stability},
replacing the overall sparsity $d$ with the ``local'' sparsity
$p$\footnote{Our clustering model is distinct from Donoho's model of
  {\it sparse clumps} on a grid \cite{donoho1992superresolution}, and
  so the two results cannot be compared directly.}.  Compared with
previous works, we also have an explicit control of the perturbation
$\epsilon$ for which the stability bounds hold:
$\epsilon \leq C\cdot\left(\Omega \Delta\right)^{2p-1}$. So, given $F$
satisfying the clustering assumptions and $\Omega$, we can choose
$\epsilon=c\left(\Omega \Delta\right)^{2p-1}$ such that $F$ can be
accurately resolved, and $c$ does not depend on $\Omega,\Delta$. But
this also means that given $\epsilon>0$, we can choose $\Delta_0$ and
$\Omega_0$ such that
$\left(\Omega_0 \Delta_0\right)^{2p-1} \ge {\epsilon\over c}$, and for
any $F$ satisfying the clustering assumptions with $\Delta = \Delta_0$
and $\Omega = \Omega_0$, the SR problem can be accurately solved.
Therefore, fixing $\epsilon$, our results show that accurate recovery
is possible for all $\SRF$ values up to
$\left({1\over\epsilon}\right)^{1\over{2p-1}}$ (but possibly also for
higher values of $\SRF$). On the other hand, a similar argument using
the lower bounds for the minimax error shows that with perturbation
of magnitude $\epsilon$, no algorithm can resolve signals having a
cluster of size $p$ and separation
$\Delta_{\epsilon} \lessapprox {1\over\Omega}\epsilon^{1\over{2p-1}}$,
giving an upper bound for the attainable $\SRF$ values exactly
matching the lower bound above. To summarize, we obtain {\it the best possible 
scaling of the attainable resolution with clustered sparsity $p$ and
absolute perturbation $\epsilon$:}
\begin{equation}
  \label{eq:srf-bounds-ultimate}
  \SRF \asymp \sqrt[2p-1]{1\over\epsilon}.
\end{equation}

This H\"{o}lder-type scaling is much more favorable compared to
 SR by analytic continuation under the
prior of compact signal support, where the bandwidth extrapolation
factor scales only as a fractional power of $\log{1\over\epsilon}$,
see e.g. \cite{batenkov2019} and references therein. Also note that
the sparse SR problem enjoys linear  stability in $\epsilon$ 
\eqref{eq:noise.amplification}, whereas analytic continuation exhibits
stability of the form $Error \approx \epsilon^{\gamma}$, where
$\gamma < 1$ \cite{bertero1998,batenkov2019}.

Stable SR in the on-grid setting of
\cite{demanet2015recoverability,donoho1992superresolution,li_stable_2017}
is closely related to the smallest singular value of a certain class
of Fourier-type matrices. Using the decimation technique (see also
\cite{cuyt2018a,cuyt2018}), in a recent paper \cite{batenkov2018} we
have derived novel estimates\footnote{Estimates for the smallest
  singular value were independently obtained in \cite{li_stable_2017}
  giving same asymptotic order but better absolute constants. In
  \cite{batenkov2019spectral} we have obtained optimal scalings of all
the singular values by different techniques.} for
this quantity under the partial clustering setting (compare with
\cite{aubel_vandermonde_2017,moitra_super-resolution_2015,bazan_conditioning_2000,ferreira_superresolution_1999,kunis2018}),
and using these results, we have shown in the same paper that the
asymptotic scaling of the condition number for on-grid SR in this
regime is $\SRF^{2p-1}$, matching the off-grid setting of the present
paper.

The question of providing rigorous performance guarantees for
high-resolution algorithms such as MP, MUSIC, ESPRIT and others, in
the super-resolution regime $\SRF>1$, is of current interest.  In two
very recent works, \cite{li2019,li_stable_2017}, the authors derive
stability estimates for MUSIC and ESPRIT algorithms under similar
clustering assumptions, finite sampling and white Gaussian
perturbation model. Their results suggest that the corresponding
noise amplification factors ${\cal K}$ for the nodes are of the order
$\SRF^{2p-2}$ with high probability. During the review of the present
paper, the authors of \cite{li2019} established near-optimality
of ESPRIT in the bounded noise model. In particular, they showed that
ESPRIT is optimal up to a factor of $1/\Omega$, i.e.
$|x_j-\tilde{x}_j| \lessapprox (\Omega \Delta)^{-2p-2} \varepsilon$ with
discrete Fourier measurements, however, requiring
$\varepsilon \lessapprox (\Omega \Delta)^{4p-3}/\Omega$. We also mention
\cite{beckermann2007,golub1999}, where the connection between
perturbation of (square) matrix pencil eigenvalues and the a-priori
distribution of these eigenvalues was established via potential
theory. It will be interesting to investigate the possibility to
applying these methods to the analysis of MP in the clustered setting.

Turning to other techniques, the special case of a single cluster can
be solved with optimal accuracy by polynomial homotopy methods, as
described in \cite{batenkov2017accurate}, however in order to
generalize this algorithm to configurations with non-cluster nodes, we
need to know the optimal decimation parameter $\lambda$. Nonlinear
least-squares and related methods (e.g., Variable Projections
\cite{golub1973,oleary2013}) apparently provide an optimal recovery
rate, however they generally require very accurate initialization. We
hope that our methods may help in analyzing these techniques as well,
and plan to pursue this line of research in the future. For the case
of positive point sources, stability rate $\SRF^{2p}$ has been
established for convex optimization techniques in
\cite{morgenshtern2016super}, see also a related preprint
\cite{denoyelle_support_2015}.

\subsection{Organization of the paper}
In Section \ref{sec:main-results} we provide the necessary definitions
and formulate the main results. In Section \ref{sec:numerics} we
present several numerical experiments confirming the optimality of the
Matrix Pencil algorithm. The proof of Theorem
\ref{thm.accuracy.bounds.upper} (upper bound) is presented in Section
\ref{sec.upper.bound}.  The proof of Theorem
\ref{thm.accuracy.bounds.lower} (lower bound) is given in Section
\ref{sec:lower-bound}.

\subsection{Acknowledgements}
\label{sub:acks}
The research of GG and YY is supported in part by the Minerva
Foundation. DB is supported in part by AFOSR grant FA9550-17-1-0316,
NSF grant DMS-1255203, and a grant from the MIT-Skolkovo initiative.

\section{Minimax bounds for clustered super-resolution}
\label{sec:main-results}
\subsection{Notation and preliminaries}\label{subsection.notation}

We shall denote by $\P_d$ the parameter space of
signals $F$ with complex amplitudes and real, pairwise distinct and
ordered nodes,
$$ {\cal P}_d = \left\{(\av,\xv) : \av=(a_1,\ldots,a_d)\in {\mathbb
    C}^{d},\; \xv = (x_1,\ldots,x_d)\in \R^d , \ x_1<x_2<\ldots<x_d \right\},
$$
and identify signals $F$ with their parameters
$(\av,\xv)\in \P_d.$ In particular, this induces a structure of a
linear space on ${\cal P}_d$. Throughout this text we will always use
the maximum norm $\|\cdot\|=\|\cdot\|_{\infty}$ on $\C^d,\;\R^d$ and  ${\cal P}_d$, where for $F=(\av,\xv) \in \P_d$
$$
	\|F\| =\max \big(\|\av\|_{\infty},\|\xv\|_{\infty}\big).
$$

We shall denote the orthogonal coordinate projections of a signal $F$
to the $j$-th node and $j$-th amplitude, respectively, by
$P_{\xv,j}: \P_d \to \mathbb{R}$ and $P_{\av,j}:\P_d \to
\mathbb{C}$. We shall also denote the $j$-th component of a vector
$\vec{v}$ by $\vec{v}_j$.

Let $L_{\infty}[-\Omega,\Omega]$ denote the space of bounded
complex-valued functions defined on $[-\Omega,\Omega]$ with the norm
$\|e\|=\max_{|s|\leq\Omega}\left|e(s)\right|$.

\begin{definition}\label{def.reconstruction.algo}
  Given $\Omega>0$ and $U\subseteq \P_d$, we denote by
  $\ecl(\Omega,U)$ the class of all admissible
  reconstruction algorithms, i.e.
  $$
  \ecl(\Omega,U)=\biggl\{\est:L_{\infty}\left[-\Omega,\Omega\right] \to
    U\biggr\}.
  $$
\end{definition}

\begin{definition}\label{def.minmax}
Let $U\subset \P_d$. We consider the minimax error rate in estimating
a signal $F \in U$ \footnote{To ensure the minimax error rate is
  finite, depending on the noise level, we impose constraints on
  $U \subset \P_d$, namely lower and upper bounds on the magnitude of
  the amplitudes and the separation of the nodes. We will specify
  these constraints exactly in the statements of the accuracy bounds.}
from $\Omega$-bandlimited data as in \eqref{eq:noise-condition},
with measurement error $\epsilon>0$:
$$\mm(\epsilon, U,\Omega) = \inf_{\est\in\ecl(\Omega,U)} \ \sup_{F \in U} \ \sup_{\|e\|\leq\epsilon} \|F-\est\left({\cal F}(F)+e\right)\|.
$$

Similarly the minimax errors of estimating the individual nodes,
respectively, the amplitudes of $F\in U$ are defined by
\begin{align*}
  \mm^{\xv,j}(\epsilon, U, \Omega) &= \inf_{\est\in\ecl(\Omega,U)} \ \sup_{F\in U} \ \sup_{\|e\|\le \epsilon} \left|P_{\xv,j}(F)-P_{\xv,j}\left(\est\left({\cal F}(F)+e\right)\right)\right|,\\
  \mm^{\av,j}(\epsilon, U,\Omega) &= \inf_{\est\in\ecl(\Omega,U)} \ \sup_{F \in U} \ \sup_{\|e\|\le \epsilon} \left|P_{\av,j}(F)-P_{\av,j}\left(\est({\cal F}(F)+e)\right)\right|.
\end{align*}
\end{definition}

Let a signal $F \in \P_d$  be fixed. We define the
$\epsilon$-error set $\EO$ as the following pre-image.

\begin{definition}\label{def.error.set} The error set
  $E_{\epsilon,\Omega}(F) \subset \P_d$ is the set
  consisting of all the signals $F'\in \P_d$ with
  \begin{align*} \left|\F(F')(s)-\F(F)(s) \right|\le \epsilon, & &
                                                                   s\in[-\Omega,\Omega].
  \end{align*}
\end{definition}

We will denote by $E^{\xv,j}_{\epsilon}(F)=E^{\xv,j}_{\epsilon,\Omega}(F)$
and $E^{\av,j}_{\epsilon}(F)=E^{\av,j}_{\epsilon,\Omega}(F)$ the
projections of the error set onto the individual nodes and the amplitudes
components, respectively:
\begin{align}\label{eq:individ.err.set.def}
  \begin{split}
  E^{\xv,j}_{\epsilon,\Omega}(F)&= \left\{ \xv_j'\in{\mathbb R}: \left(\av',\xv'\right)\in E_{\epsilon,\Omega}(F) \right\} \equiv P_{\xv,j} E_{\epsilon,\Omega}(F), \\
  E^{\av,j}_{\epsilon,\Omega}(F)&= \left\{ \av_j'\in{\mathbb C}: \left(\av',\xv'\right)\in E_{\epsilon,\Omega}(F) \right\}\equiv P_{\av,j} E_{\epsilon,\Omega}(F).
\end{split}
\end{align}

\smallskip

For any subset $V$ of a normed vector space with norm $\|\cdot\|$, the diameter of $V$ is
$$
	diam(V) = \sup_{\vec{v}',\vec{v}'' \in V} \|\vec{v}'-\vec{v}''\|.
$$
The minimax errors are directly linked to the diameter of the
corresponding projections of the error set by the following easy
computation, which is standard in the theory of optimal recovery
\cite{micchelli_lectures_1985,micchelli_survey_1977,micchelli_optimal_1976}
(see also
\cite{donoho1992superresolution,demanet2015recoverability,li_stable_2017}).

\begin{proposition}\label{prop.minimax} For $U \subset \P_d$,
  $\Omega>0$, $1\leq j \leq d$ and $\epsilon > 0$ we have
\begin{align}
    \label{eq:minimax-equiv-error-set}
    \frac{1}{2} \sup_{F:\ E_{\frac{1}{2}\epsilon,\Omega}(F) \subseteq U}diam\big(E_{\frac{1}{2}\epsilon,\Omega}(F)\big) 
    &\le \mm(\epsilon, U,\Omega) \le \sup_{F\in U }
                                     diam\big(E_{2\epsilon,\Omega}(F))\\
    \frac{1}{2} \sup_{F:\ E_{\frac{1}{2}\epsilon,\Omega}(F) \subseteq U}diam\big(E^{\xv,j}_{\frac{1}{2}\epsilon,\Omega}(F)\big) 
    &\le \mm^{\xv,j}(\epsilon, U,\Omega) \le \sup_{F\in U } diam\big(E^{\xv,j}_{2\epsilon,\Omega}(F)) \\
    \frac{1}{2} \sup_{F:\ E_{\frac{1}{2}\epsilon,\Omega}(F) \subseteq U}diam\big(E^{\av,j}_{\frac{1}{2}\epsilon,\Omega}(F)\big) 
    &\le \mm^{\av,j}(\epsilon, U,\Omega) \le \sup_{F\in U } diam\big(E^{\av,j}_{2\epsilon,\Omega}(F))
  \end{align}
\end{proposition}

\begin{proof} We shall prove \eqref{eq:minimax-equiv-error-set}, the
  proof in the other cases is identical. We omit $\Omega$ from the
  following to reduce clutter.
  \begin{description}
  \item[Upper bound] Let $\epsilon > 0$. For any
    $\Phi\in L_{\infty}[-\Omega,\Omega]$, let
    \begin{align*}
      {\cal B}\left(\epsilon,\Phi\right)&=\left\{F\in U: \|{\cal
                                          F}(F)-\Phi\|\leq\epsilon\right\}.
    \end{align*}
    Consider an oracle estimator $\est_{\epsilon}\in\ecl(\Omega,U)$
    defined as
  $$
  \est_{\epsilon}(\Phi)=\begin{cases}
    \text{any element of } {\cal B}(\epsilon,\Phi) & \text{if } {\cal B}(\epsilon,\Phi)\neq\emptyset, \\
    F_0 & \text{else},
  \end{cases}
  $$
  where $F_0$ is an arbitrary element of $U$. Now let $F\in U$, and
  $\Phi={\cal F}(F)+e$ where $\|e\|\leq\epsilon$. Then by definition
  $F\in{\cal B}(\epsilon,\Phi)$.  Put $F'=\est_{\epsilon}(\Phi)$,
  thus $\left\|{\cal F}(F') - \Phi\right\| \le \epsilon$, and
  therefore
  $$
  \|{\cal F}(F')-{\cal F}(F)\|\leq \|{\cal F}(F')-\Phi\|+\|\Phi-{\cal
    F}(F)\|=2\epsilon.
  $$
  We conclude that $F' \in E_{2\epsilon}(F)$, and consequently
  $\mm(\epsilon,U,\Omega)\leq\|F-F'\| \le diam\big(E_{2\epsilon}(F)\big)$.

\item[Lower bound] For the lower bound, let $F\in U$ such that
  $E_{\frac{1}{2}\epsilon}(F) \subseteq U$. Let $\xi>0$ small enough
  be fixed. There exist $F^1,\ F^2 \in E_{\frac{1}{2}\epsilon}(F)$
  with $\|F^1 -F^2\| = diam\big(E_{\frac{1}{2}\epsilon}(F)\big)-\xi$.
  Let $\Phi = {\cal F}(F)$, and let $F'=\est(\Phi)$ be the output of a
  certain estimator $\est$ corresponding to the input $\Phi$. We have
  $\big\|\Phi - {\cal F}(F^1)\big\|,\ \big\|\Phi - {\cal F}(F^2)\big\|
  \le \epsilon$.  Consequently, there exist perturbation functions
  $e_1,e_2$ satisfying $\|e_1\|,\|e_2\|\leq\epsilon$, while also
  \begin{align*}
    {\cal F}(F')=\Phi&={\cal F}(F^1)+e_1 = {\cal
      F}(F^2)+e_2.
  \end{align*}
  By definition of the minimax error we therefore have
  \begin{align*}
    \mm\left(\epsilon,U,\Omega\right) &= \inf_{\est}\sup_{\|e\|<\epsilon,F\in U}\|F-\est({\cal F}(F)+e)\| \\
    &\geq\inf_{\est}\max\left(\|F^1-F'\|,\|F^2-F'\|\right) \\
    &\geq \inf_{\est}{1\over 2}\left\{\|F^1-F'\|+\|F^2-F'\|\right\} \\
    &\geq {1\over 2}\|F^1-F^2\| \\
    &={1\over 2}diam\left(E_{{1\over 2}\epsilon}(F)\right)-{\xi\over 2}.
  \end{align*}
  The lower bound follows by letting $\xi\to 0$.
\end{description}
\end{proof}

\subsection{Uniform estimates of minimax error for clustered configurations}
The main goal of this paper is to estimate
$\mm\left(\epsilon,U,\Omega\right)$ (in fact its component-wise
analogues $\mm^{\xv,j}\left(\epsilon,U,\Omega\right)$ and
$\mm^{\av,j}\left(\epsilon,U,\Omega\right)$) where
$U\subset {\cal P}_d$ are certain compact subsets of ${\cal P}_d$
containing signals with $p\leq d$ nodes forming a small, approximately
uniform, cluster. In order to have explicit bounds, we describe such
sets $U$ by additional parameters $T,h,\tau,\eta,m,M$ as follows.

\begin{definition}[Uniform cluster configuration, Figure \ref{fig:cluster-conf}]\label{def.uniform.cluster}
  Given $0 < \tau,\eta\leq 1$ and $0< h \le T$, a node vector
  $\xv=(x_1,\ldots,x_d) \in \R^d$ is said to form a
  $(p,h,T,\tau,\eta)$-clustered configuration, if
  there exists a subset of $p$ nodes
  $\xv^c=\{x_\kappa,\ldots, x_{\kappa+p-1}\} \subset \xv$, $p\ge 2$,
  which satisfies the following conditions:
  \begin{enumerate}
  \item for each $x_j, x_k \in \xv^{c}, j\ne k$,
    $$\tau h \le |x_j-x_k| \le h;$$
  \item for $x_{\ell} \in \xv \setminus \xv^{c}$ and $x_j\in \xv$, $\ell \ne j$,
    $$\eta T \le |x_{\ell}-x_j| \le T.$$
  \end{enumerate}
\end{definition}

\begin{figure}
  \centering
  \includegraphics[width=0.9\linewidth]{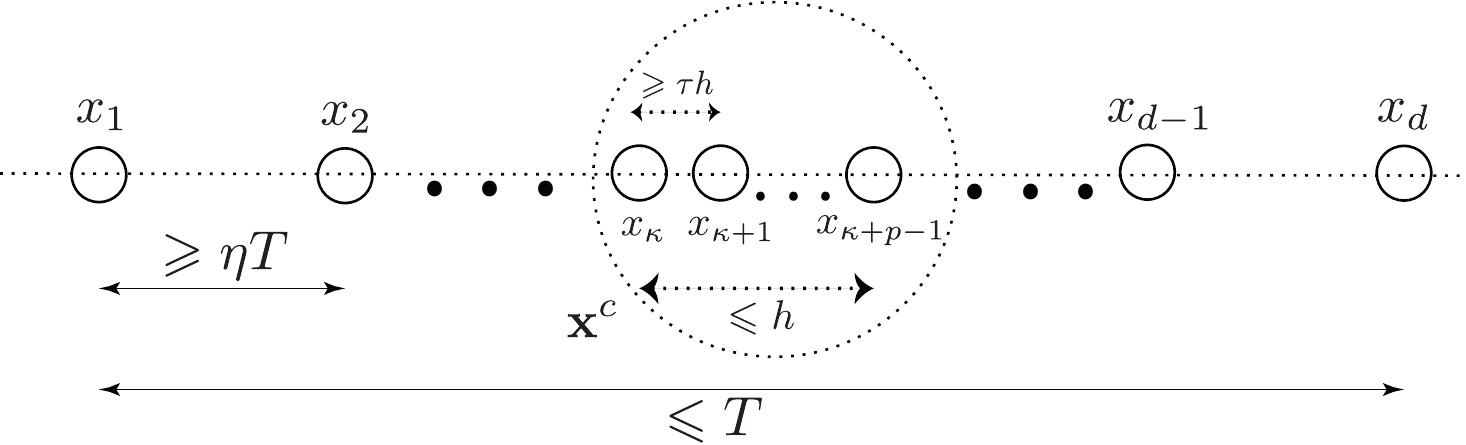}
  \captionsetup{singlelinecheck=off,font=small}
  \caption{A sketch of a uniform $(p,h,T,\tau,\eta)$-clustered
    configuration $\xv=(x_1,\dots,x_d)$ as in Definition
    \ref{def.uniform.cluster}.}
  \label{fig:cluster-conf}
\end{figure}

Our first main result provides an upper bound on
$diam\left(E_{\epsilon,\Omega}(F)\right)$, and its coordinate
projections, for any signal $F$ forming a clustered configuration as
above.

\begin{theorem}(Upper bound)\label{thm.accuracy.bounds.upper}
  Let $F=(\av,\xv) \in \P_d$, such that $\xv$ forms a
  $(p,h,T,\tau,\eta)$-clustered configuration and $0<m \leq \|\av\|$.  Then there exist positive constants
  $C_1,\ldots, C_5$, depending only on $d,p,m$, such that for
  each $\frac{C_4}{\eta T} \le \Omega \le \frac{C_5}{h} $ and
  $\epsilon \le C_3(\Omega \tau h)^{2p-1}$, it holds that:
    \begin{align*}
    	diam(E^{\xv,j}_{\epsilon,\Omega}(F)) & \le {C_1\over\Omega} \epsilon \times
    	\begin{cases}
    		(\Omega \tau h)^{-2p+2} , & x_j \in \xv^c,   \\
    		1,& x_j \in \xv \backslash \xv^c;
    	\end{cases}\\ 
    	diam(E^{\av,j}_{\epsilon,\Omega}(F)) &\le C_2 \epsilon \times
    	\begin{cases}
    		(\Omega \tau h)^{-2p+1}, & x_j \in \xv^c,\\
    		1,& x_j \in \xv \backslash \xv^c.
    	\end{cases}
    \end{align*}
\end{theorem}

\begin{remark}
  Our main focus is to investigate the error rates of the
  SR problem as the cluster size becomes small. Fixing the parameters
  $p,d,m$, the range of admissible $\Omega$ in Theorem
  \ref{thm.accuracy.bounds.upper},
  $\frac{C_4}{\eta T} \le \Omega \le \frac{C_5}{h}$, is non-empty for
  a sufficiently small cluster size $h$. Furthermore we comment here that the
  constants $C_4, C_5$ actually only depend on $d$.
\end{remark}

The above estimates are order optimal, as our next main theorem
shows. For simplicity and without loss of generality, in the results
below we assume that the index $\kappa$ is fixed.

\begin{theorem}(Lower bound)\label{thm.accuracy.bounds.lower} Let
  $m\leq M,2\leq p\leq d,\tau\leq{1\over{p-1}},\eta<{1\over{d}},T>0$
  be fixed. There exist positive constants $C_1'\dots,C_5'$, depending
  only on $d,p,m,M$, such that for every $\Omega,h$ satisfying
  $h\leq C_4' T$ and $\Omega h \leq C_5'$ there exists
  $F=(\av,\xv)\in\P_d$, with $\xv$ forming a
  $(p,h,T,\tau,\eta)$-clustered configuration, and with
  $0<m \leq \|\av\| \leq M < \infty$, such that for certain indices
  $j_1,j_2\in\left\{\kappa,\dots,\kappa+p-1\right\}$ and every
    $\epsilon \le C_3'(\Omega\tau h)^{2p-1}$, it holds that:
  \begin{align*}
    	diam(E^{\xv,j}_{\epsilon,\Omega}(F)) & \ge {C_1'\over\Omega} \epsilon \times
    	\begin{cases}
    		  (\Omega \tau h)^{-2p+2} , & \text{if } j=j_1,   \\
    		1 ,& \forall j\notin\left\{\kappa,\dots,\kappa+p-1\right\};
    	\end{cases}\\ 
    	diam(E^{\av,j}_{\epsilon,\Omega}(F)) &\ge C_2' \epsilon \times
    	\begin{cases}
    		(\Omega\tau h)^{-2p+1}, & \text{if } j=j_2,\\
    		1,& \forall j\notin\left\{\kappa,\dots,\kappa+p-1\right\}.
    	\end{cases}
    \end{align*}
\end{theorem}

\begin{remark}
  The lower bounds for the quantities
  $diam(E^{\xv,j}_{\epsilon,\Omega}(F))$ were shown in
  \cite{akinshin2015accuracy} to hold for \emph{any signal $F$ with
    real amplitudes}, however, at the expense of the implicit
  dependence of the constants on the separation parameter
  $\tau$. While bounding $diam(\EO)$ (and its projections) for all
  signals $F$ is an interesting question in its own right, in this
  paper we use these to bound the minimax error rate, and therefore it
  is sufficient to show that there exist certain signals with large
  enough $\EO$.  As it turns out, it is possible to obtain a more
  accurate geometric description of these sets, which in turn can be
  used for reducing reconstruction error if additional a-priori
  information is available. Work in this direction was started in
  \cite{akinshin2017error} and we intend to provide further details of
  these developments in a future work.
\end{remark}

Combining Theorems \ref{thm.accuracy.bounds.upper} and
\ref{thm.accuracy.bounds.lower} with Proposition \ref{prop.minimax},
we obtain optimal rates for the minimax error $\mm$ and its
projections as follows.

\begin{theorem}\label{thm.minmax.bounds} Let   $m <  M,2\leq p\leq
  d,\tau<{1\over{2(p-1)}},\eta<{1\over{2d}},T>0$ 
  be fixed. There exist constants $c_1,c_2,c_3$, depending only on
  $d,p,m,M$ such that for all
  $\frac{c_1}{\eta T} \le \Omega \le \frac{c_2}{h} $ and
  $\epsilon \le c_3(\Omega \tau h)^{2p-1}$, the minimax error rates
  for the set
   \begin{align*}
     U:&=U(p,d,h,\tau,\eta,T,m,M)\\
     &=\left\{(\av,\xv)\in\P_d:\quad 0<m
       \leq \|\av\| \leq M < \infty,\; \xv \text{ forms a }
       (p,h,T,\tau,\eta)\text{-clustered configuration} \right\},
   \end{align*}
   satisfy the following.
   \begin{enumerate}
   \item For the non-cluster nodes:
     \begin{align*}
       \forall j\notin\left\{\kappa,\dots,\kappa+p-1\right\}:\quad
       \begin{cases}
         \mm^{\xv,j}(\epsilon,U,\Omega)  \asymp {\epsilon\over\Omega},\\
         \mm^{\av,j}(\epsilon,U,\Omega)  \asymp \epsilon.
       \end{cases}
     \end{align*}
   \item For the cluster nodes:
   \end{enumerate}
   \begin{align*}
    \max_{j=\kappa,\dots,\kappa+p-1}\mm^{\xv,j}(\epsilon,U,\Omega) & \asymp {\epsilon\over\Omega} (\Omega\tau h)^{-2p+2}, \\
    \max_{j=\kappa,\dots,\kappa+p-1}\mm^{\av,j}(\epsilon,U,\Omega) &\asymp \epsilon 
                    (\Omega\tau h)^{-2p+1}.
   \end{align*}
   The proportionality constants in the above statements depend only
   on $d,p,m,M$.
 \end{theorem}

 \begin{proof}
   Let $C_3,C_3',C_4,C_4',C_5,C_5'$ be the constants from Theorems
   \ref{thm.accuracy.bounds.upper} and
   \ref{thm.accuracy.bounds.lower}. Put $c_1=C_4$ and
   $c_2=\min\left(C_5,C_5',C_4 C_4'\right)$. Let
   $\frac{c_1}{\eta T} \leq \Omega \leq \frac{c_2}{h}$, and
   $\epsilon \leq c_3(\Omega\tau h)^{2p-1}$, where
   $c_3\leq\min(C_3,C_3')$ will be determined below. It is immediately
   verified that $\Omega,h$ and $\epsilon$ as above satisfy the
   conditions of both Theorems \ref{thm.accuracy.bounds.upper} and
   \ref{thm.accuracy.bounds.lower}.
   \begin{description}
   \item[Upper bound] Directly follows from the upper bounds in
     Theorem \ref{thm.accuracy.bounds.upper} and Proposition
     \ref{prop.minimax}.
   \item[Lower bound] Denote
     $U_{\epsilon}=\left\{F\in U: E_{\frac{1}{2}\epsilon,\Omega}(F)
       \subseteq U \right\}$. To prove the lower bounds on $\mm$, it
     is sufficient to show that there exists an
     $F\in U_{\epsilon}\neq\emptyset$ such that the conclusions of
     Theorem \ref{thm.accuracy.bounds.lower} are satisfied for this
     $F$.

     It is not difficult to see that for any choice of the parameters
     as above, the set $U$ has a non-empty interior, and furthermore
     that one can choose $m',M'$ satisfying $m<m'<M'<M$, and also
     $T'=0.99T$, $\tau'=2\tau$ and $\eta'=2\eta$, such that
     $$
     U'=U(p,d,h,\tau',\eta',T',m',M')\subset U, \quad \partial U'
     \cap \partial U = \emptyset.
     $$
     By construction, there exist positive constants
     $\tilde{C_1},\tilde{C_2}$, independent of $\Omega,h$ and
     $\tau,\eta$, such that
     \begin{align}\label{eq:mm-lb-ub-inf}
       \begin{split}
       \inf_{u\in\partial U, u'\in\partial
       U'}\left|P_{\xv,j}(u)-P_{\xv,j}(u')\right| & \ge \tilde{C_1} \times
                                              \begin{cases}
                                                \tau h , & x_j \in \xv^c,   \\
                                                \eta T ,& x_j \in \xv \setminus \xv^c;
                                              \end{cases}\\ 
       \inf_{u\in\partial U, u'\in\partial
         U'}\left|P_{\av,j}(u)-P_{\av,j}(u')\right| &\ge \tilde{C_2}.
     \end{split}
     \end{align}
     Now we use the fact that $\epsilon<c_3(\Omega\tau
     h)^{2p-1}$. Applying Theorem \ref{thm.accuracy.bounds.upper} to
     an arbitrary signal $F'\in U'$, and using the conditions
     ${1\over\Omega}\leq \frac{\eta T}{c_1}$ and
     $\Omega\tau h \leq \Omega h \leq c_2$, we obtain that
     \begin{align}\label{eq:mm-lb-ub-diam}
       \begin{split}
       diam\left(E^{\xv,j}_{{1\over 2}\epsilon}(F')\right) &\leq
                                                             \begin{cases}\frac{C_1c_3}{2}\tau h, & x_j\in\xv^c,\\
                                                               \frac{C_1c_3}{2\Omega}(\Omega \tau h)^{2p-1} \leq \frac{C_1c_3}{2c_1}c_2^{2p-1}\eta T, & x_j\in\xv\setminus\xv^c;
                                                             \end{cases}\\
       diam\left(E^{\av,j}_{{1\over 2}\epsilon}(F')\right) & \leq \begin{cases}
         \frac{C_2c_3}{2}, & x_j\in \xv^c, \\
         \frac{C_2c_3}{2}c_2^{2p-1}, & x_j\in\xv\setminus\xv^c.
       \end{cases}
     \end{split}
     \end{align}
     Now we set $c_3=\min(C_3,C_3',C_3'')$ where
     $$
     C_3''=\min(1,c_1)\times\min(1,c_2^{-2p+1}) \times \min\biggl(\frac{2\tilde{C_1}}{C_1},\frac{2\tilde{C_2}}{C_2}\biggr).
     $$
     Combining \eqref{eq:mm-lb-ub-inf} and \eqref{eq:mm-lb-ub-diam} we
     obtain that $F'\in U_{\epsilon}$. Since $F'\in U'$ was arbitrary,
     we conclude that $U'\subseteq U_{\epsilon}$. Since clearly
     $U'\neq\emptyset$, applying Proposition \ref{prop.minimax} and Theorem
     \ref{thm.accuracy.bounds.lower} finishes the proof.
   \end{description}
\end{proof}

%

\section{Numerical optimality of Matrix Pencil algorithm}
\label{sec:numerics}

The main theoretical result of this paper, Theorem
\ref{thm.minmax.bounds}, establishes the best possible
scalings for the SR problem with clustered nodes.  In this section we
provide some numerical evidence that a certain SR algorithm, the
Matrix Pencil (MP) method \cite{hua_svd_1991,hua_matrix_1990}, attains
these performance bounds. 

Our choice of MP is fairly arbitrary, as we believe that many
high-resolution algorithms have similar behaviour in the regime
$\SRF\gg 1$.

Throughout this section, we replace $\Omega$ by $N$, so that the
spectral data is sampled with unit spacing.

\subsection{The Matrix Pencil method}

\begin{algorithm}[hbt]
  \SetKwInOut{Input}{Input} \SetKwInOut{Output}{Output}
  \Input{Model order $d$}
  \Input{Sequence $\{\tilde{m}_k\},\;k=0,1,\dots,N-1$ where $N >
    2d$, of the form \eqref{eq:mp.model.eq}}
  \Input{pencil parameter $d+1 \leq L \leq N-d$}
  \Output{Estimates for the nodes $\{x_j\}$ and amplitudes
    $\{a_j\}$ as in \eqref{eq:mp.model.eq}}
  Compute the matrices $A=\widetilde{H}^{\uparrow}, B=\widetilde{H}_{\downarrow}$\;
  Compute the truncated Singular Value Decomposition (SVD) of $A,B$
  of order $d$:
  $$
  A = U_1 \Sigma_1 V_1^H,\quad B = U_2 \Sigma_2 V_2^H,
  $$
  where $U_1,U_2,V_1,V_2$ are $L\times d$ and $\Sigma_1,\Sigma_2$ are
  $d\times d$  \;
  Generate the reduced pencil
  $$
  A'=U_2^H U_1 \Sigma_1 V_1^H V_2,\quad B'=\Sigma_2
  $$
  where $A',B'$ are $d\times d$\; Compute the generalized eigenvalues
  $\tilde{z}_j$ of the reduced pencil $(A',B')$ , and put
  $\{\tilde{x}_j\}=\frac{1}{2\pi}\{\angle \tilde{z}_j\},\;j=1,\dots,d$\; Compute
  $\tilde{a}_j$ by solving the linear least squares problem
  $$
  \vec{\tilde{a}}=\arg\min_{\vec{a}\in\C^d}\|\vec{\tilde{m}}-\tilde{V}\vec{a}\|_2,
  $$
  where $\tilde{V}=\tilde{V}(\tilde{\xv})$ is the Vandermonde matrix
  $\tilde{V} = \left[\exp\left(2\pi\imath\tilde{x}_j
      k\right)\right]_{k=0,\dots,N-1}^{j=1,\dots,d}$\; \Return the
  estimated $\tilde{x}_j$ and $\tilde{a}_j$.
  \caption{The Matrix Pencil algorithm}
  \label{alg:mpencil}
\end{algorithm}

Let $F=(\av,\xv)\in\P_d$ as in \eqref{eq.spike.train.signal} with
$x_j \in \left[-\frac{1}{2},\frac{1}{2}\right]$. Given the noisy
Fourier measurements
\begin{align}
  \label{eq:mp.model.eq}
  \begin{split}
    \tilde{m}_k&= \underbrace{{\cal F}(F)(-k)}_{=m_k} + n_k \\
    &= \sum_{j=1}^d a_j \exp(2\pi\imath x_j k ) + n_k,\qquad k=0,1,\dots,N-1, \quad N > 2d,
  \end{split}
\end{align}
the Matrix Pencil method estimates
$\tilde{F}=\left(\tilde{\av},\tilde{\xv}\right)$ as follows. Consider
the Hankel matrix
\begin{equation}
  \label{eq:mp.hankel.noiseless}
H =
\begin{bmatrix}
  m_0 & m_1 & \dots & m_{N-L-1} \\
  m_1 & m_2 & \dots & m_{N-L} \\
  \vdots & \iddots & \iddots & \vdots \\
  m_{L} & m_{L+1} & \dots & \dots m_{N-1}
\end{bmatrix}
\in \C^{(L+1)\times  (N-L)},
\end{equation}
and further let $H^{\uparrow}=H[0:L-1,:]$ and
$H_{\downarrow}=H[1:L,:]$ be the $L\times(N-L)$ matrix obtained from
$H$ by deleting the last (respectively, the first) row. Then it turns
out that that the numbers $z_j=\exp(2\pi\imath x_j)$ are the $d$
nonzero generalized eigenvalues (i.e. rank-reducing numbers) of the
pencil $H_{\downarrow}-zH^{\uparrow}$. If we now construct the noisy
matrices $A=\widetilde{H}^{\uparrow}, B=\widetilde{H}_{\downarrow}$
from the available data
$\{\tilde{m}_k\}_{k=0,\dots,N-1}$, we could apparently
just solve the Generalized Eigenvalue Problem with $A,B$. However, if
$L>d$ then the pencil $B-zA$ is close to being singular, and so an
additional step of low-rank approximation is required. We summarize
the MP method in Algorithm \ref{alg:mpencil}, and the interested
reader is referred to the widely available literature on the subject
(e.g. \cite{hua_svd_1991,hua_matrix_1990,moitra_super-resolution_2015,stoica_spectral_2005},
and references therein) for further details. Note that there exist
numerous variants of MP, but, again, we believe the particular details
to be immaterial for our discussion.

\subsection{Experimental setup}\label{numerics.setup}

\subsubsection{Clustered node configurations}

In our experiments presented below, we constructed
$(p,h,T,\tau,\allowbreak\eta)$-clustered configurations with
$$
\tau={1\over{p-1}}, T=\pi, \eta=\frac{\pi-h}{\pi(d-p+1)}
$$ as follows:
\begin{enumerate}
\item The cluster nodes $\xv^c=(x_1,\dots,x_p)$ where $x_j=(j-1)\cdot
  \Delta$ and $\Delta={h\over {p-1}}$ for $j=1,\dots,p$.
\item The non-cluster nodes were chosen to be $$x_{p+j}=(p-1)\Delta +
  j\cdot \frac{\pi-(p-1)\Delta}{d-p+1},\quad j=1,\dots,d-p.$$
\end{enumerate}

\subsubsection{Choice of signal and
  perturbation}
Two different schemes were tested:
\begin{enumerate}[label=\textbf{S\arabic*}]
\item \label{s1} A generic signal with complex amplitude vector
  $\av^{(1)} = \left(\imath^0,\imath^1,\imath^2,\dots\right)\in\C^d$
  and a bounded random perturbation sequence $\{n_k\}$, uniformly
  distributed in $\left[-\epsilon,\epsilon\right]$.
\item \label{s2} Worst-case scenario in accordance
  with the construction of Section \ref{sec:lower-bound} (and in
  particular of Theorem \ref{thm:moments-lb}): a real amplitude vector
  $\av^{(2)} = \left(1,-1,1,\dots,\right)\in\R^d$ and the perturbed
  Fourier coefficient sequence $\{\tilde{m}_k\}$ of the particular
  signal $F_{\epsilon}=(\av',\xv')\in\P_d$ constructed according to
  Algorithm \ref{alg.worst.case}:
$$
\tilde{m}_k = {\cal F}(F_{\epsilon})(-k)=\sum_{j=1}^d \av'_j \exp(2\pi\imath \xv'_j k),\quad k=0,\dots,N-1.
$$
\end{enumerate}


\begin{algorithm}[!htb]
  \SetKwInOut{Input}{Input}
  \SetKwInOut{Output}{Output}

  \Input{Signal $F=(\av,\xv)\in\P_d$ with $\av=\av^{(2)}$ and cluster
    nodes $\xv^{c}=(x_1,\dots,x_p)$}
  \Input{Noise level $\epsilon$}
  \Output{The perturbed signal $F_{\epsilon}$}
  Compute the cluster center $\mu=\frac{x_1+x_p}{2}$ and put
  $\tilde{\xv}^{c}=\xv^{c}-\mu$ \;
  Construct the moment vector of the centered cluster:
  $\vec{g}=\left(\sum_{j=1}^p\av_j \tilde{\xv}_j^k \right)_{k=0,1,\dots,2p-1} \in \R^{2p}$ \;
  Construct the vector $\vec{g}'$ to be equal to $\vec{g}$ except the
  last entry: $\vec{g}'_k=\vec{g}_k$ for $k=0,1,\dots,2p-2$ and
  $\vec{g}'_{2p-1} = \vec{g}_{2p-1}+\epsilon$ \;
  Solve the Prony problem of order $p$ with the data $\vec{g}'$ (for
  $\epsilon$ small enough, a unique solution always exists -- see
  Proposition \ref{prop.prony.method} and
  \cite{batenkov2013geometry}), obtaining a signal $F'=(\av',\xv') \in
  \P_{p}$ \;
  Move the cluster nodes back and put
  $$
  F_{\epsilon}(x) = \sum_{j=p+1}^d\av_j \delta(x-\xv_j) +
  \sum_{j=1}^{p}\av'_j \delta(x-(\xv'_j+\mu));
  $$
  \Return the signal $F_{\epsilon}$.
  \caption{The worst-case perturbation signal}
  \label{alg.worst.case}
\end{algorithm}

\begin{algorithm}[!htb]
  \SetKwInOut{Input}{Input}
  \SetKwInOut{Output}{Output}
  \SetKw{true}{true} \SetKw{false}{false}

  \Input{$p,d,h,N,\epsilon$}
  \Input{Testing scheme (either \ref{s1} or \ref{s2})}
  Construct the signal $F$ and the sequence  $\tilde{m}_k$,
  $k=0,\dots,N-1$   according to Subsection
  \ref{numerics.setup} \;
  Compute the actual perturbation magnitude
  $$
  \epsilon_0 = \max_{k=0,\dots,N-1}|{\cal F}(F)(-k)-\tilde{m}_k|;
  $$
  
  Execute the MP method (Algorithm \ref{alg:mpencil}) with
  $L=\left\lceil{N\over 2}\right\rceil$ and obtain
  $F_{MP}=(\av^{MP},\xv^{MP})$ \; \For{each $j$}{ compute the error
    for node $j$:
  $$
  e_j=\min_{\ell}|\xv^{MP}_{j}-\xv_{\ell}|;
  $$

  The success for node $j$ is defined as
  $$
  Succ_j=\left(e_j < {{\min_{\ell\neq j}|\xv_{\ell}-\xv_j|} \over 3}\right).
  $$

  \If{$Succ_j$ == \true}{
    let $\ell(j)=\arg\min_{\ell}|\xv^{MP}_{j}-\xv_{\ell}|$ \;
    compute normalized node error amplification factor
    $$
    {\cal K}_{\xv,j}=\frac{|\xv_j-\xv^{MP}_{\ell(j)}|\cdot N}{\epsilon_0};
    $$
    
    compute normalized amplitude error amplification factor
    $$
    {\cal K}_{\av,j}=\frac{|\av_j-\av^{MP}_{\ell(j)}|}{\epsilon_0};
    $$
  }
}

\Return $\epsilon_0$, and $({\cal K}_{\xv,j},{\cal K}_{\av,j},Succ_j)$
for each node $j=1,\dots,d$.
  \caption{A single experiment}
  \label{alg.single.experiment}
\end{algorithm}

\subsection{Results}

\subsubsection{Error amplification factors}
In the first set of experiments, we measured the actual error
amplification factors ${\cal K}_{\xv,j},{\cal K}_{\av,j}$ as in
Algorithm \ref{alg.single.experiment} (recall also
\eqref{eq:noise.amplification}), choosing $\epsilon,N,h$ randomly
from a pre-defined numerical range. The results are presented in
Figures \ref{fig:amplification.factors.s1} and
\ref{fig:amplification.factors.s2} for the testing schemes \ref{s1}
and \ref{s2}, accordingly. The scalings of Theorem
\ref{thm.minmax.bounds}, in particular the
dependence on $\SRF$, are confirmed.

\begin{figure}[!htb]
  \centering
  \includegraphics[width=\linewidth]{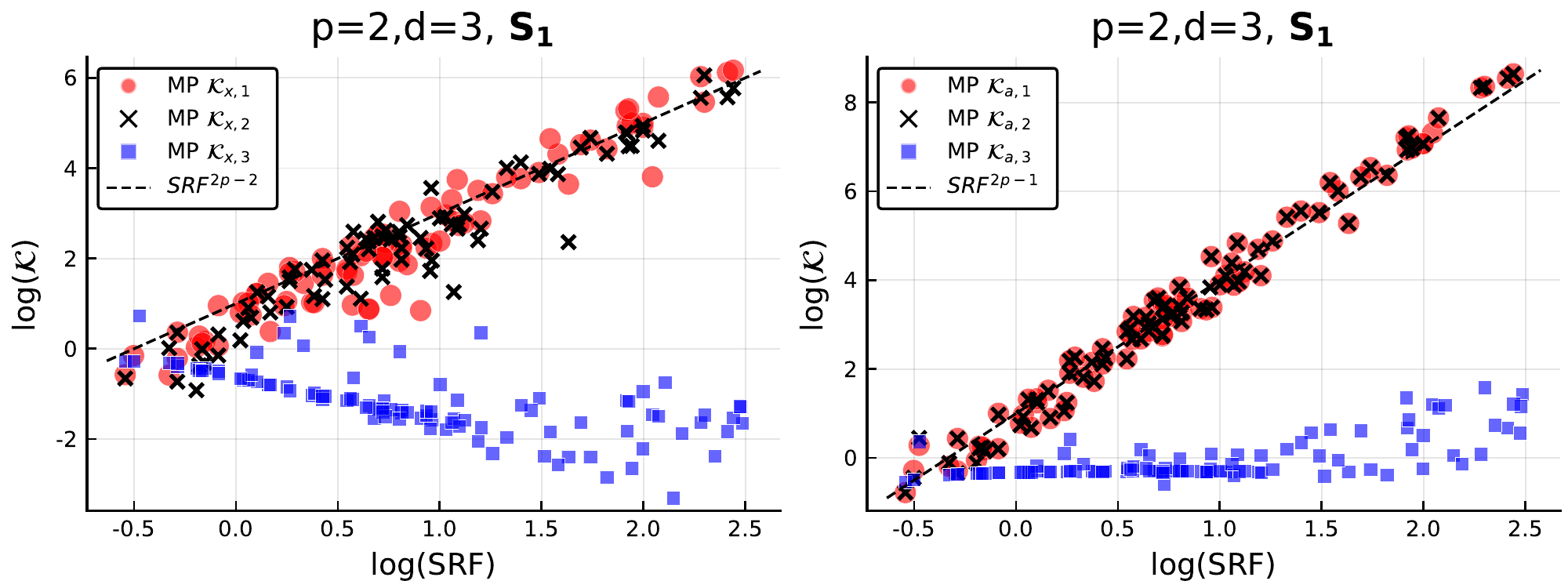}
  \captionsetup{width=\linewidth,font=small} 
  \caption{The error amplification factors. Algorithm
    \ref{alg.single.experiment} was executed 500 times with
    $p=2,d=3$, scheme \ref{s1} and varying $h,N,\epsilon$. For cluster
    nodes $j=1,2$, the node error amplification factors
    ${\cal K}_{\xv,j}$ (left panel) scale like $\SRF^{2p-2}$, while
    the amplitude error amplification factors ${\cal K}_{\av,j}$ (right panel)
    scale like $\SRF^{2p-1}$. For the non-cluster node $j=3$, both
    error amplification factors are bounded by a constant.}
  \label{fig:amplification.factors.s1}
\end{figure}

\begin{figure}[!htb]
  \centering
  \includegraphics[width=\linewidth]{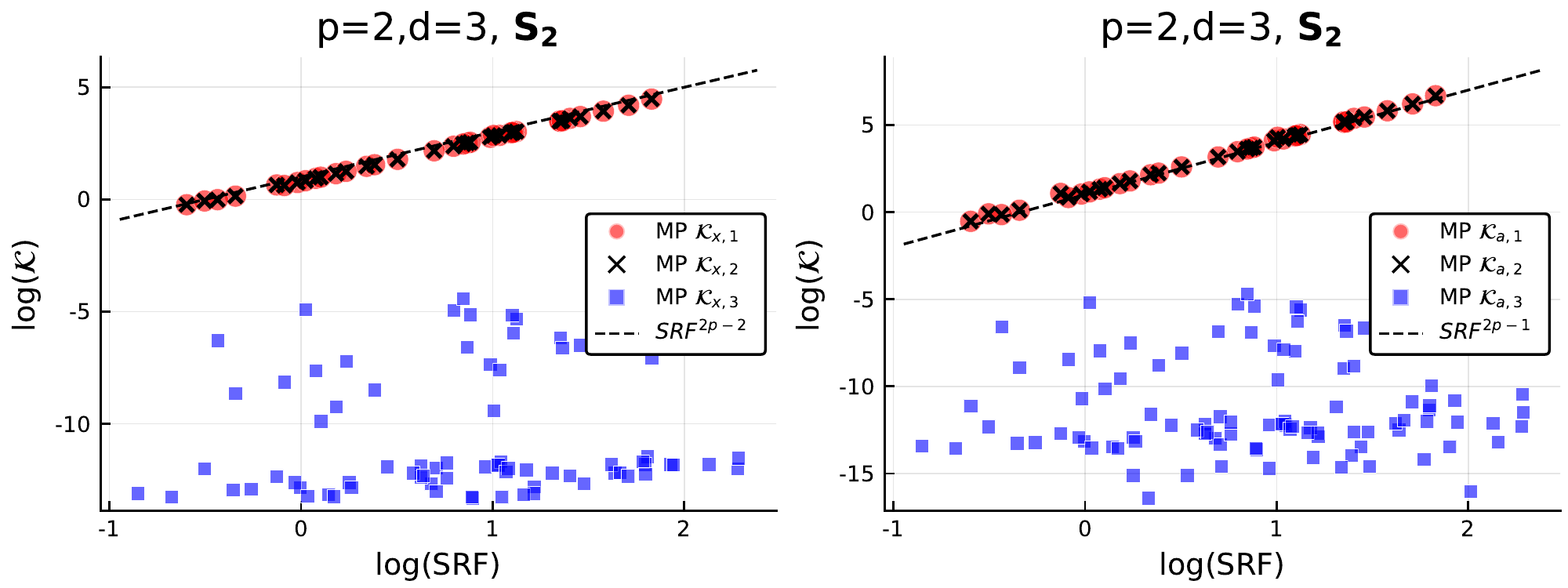}
  \captionsetup{width=\linewidth,font=small} 
  \caption{Same setup as in Figure \ref{fig:amplification.factors.s1},
    scheme \ref{s2}. Comparing with Figure
    \ref{fig:amplification.factors.s1}, the variance of the factors
    corresponding to the cluster nodes is much smaller than for the
    case of random perturbations, indicating that the construction is
    indeed worst-case.}
  \label{fig:amplification.factors.s2}
\end{figure}

\subsubsection{Noise threshold for successful recovery}
In the second set of experiments, we investigated the noise threshold
$\epsilon\lessapprox \SRF^{1-2p}$ for successful recovery, as
predicted by the theory. We have performed $15000$ random experiments
with scheme \ref{s1} (the randomness was in the choice of
$h,N,\epsilon$ and the noise sequence $\{n_k\}$) according to
Algorithm \ref{alg.single.experiment}, recording the success/failure
result of each such experiment. The results for $d=4$ and $p=2,3$ are
presented in Figure \ref{fig:phase.tran.1}, and the theoretical
scaling above is confirmed for the MP method.

\begin{figure}[hbt]
  \centering
  \subfloat{
    \includegraphics[width=0.45\linewidth]{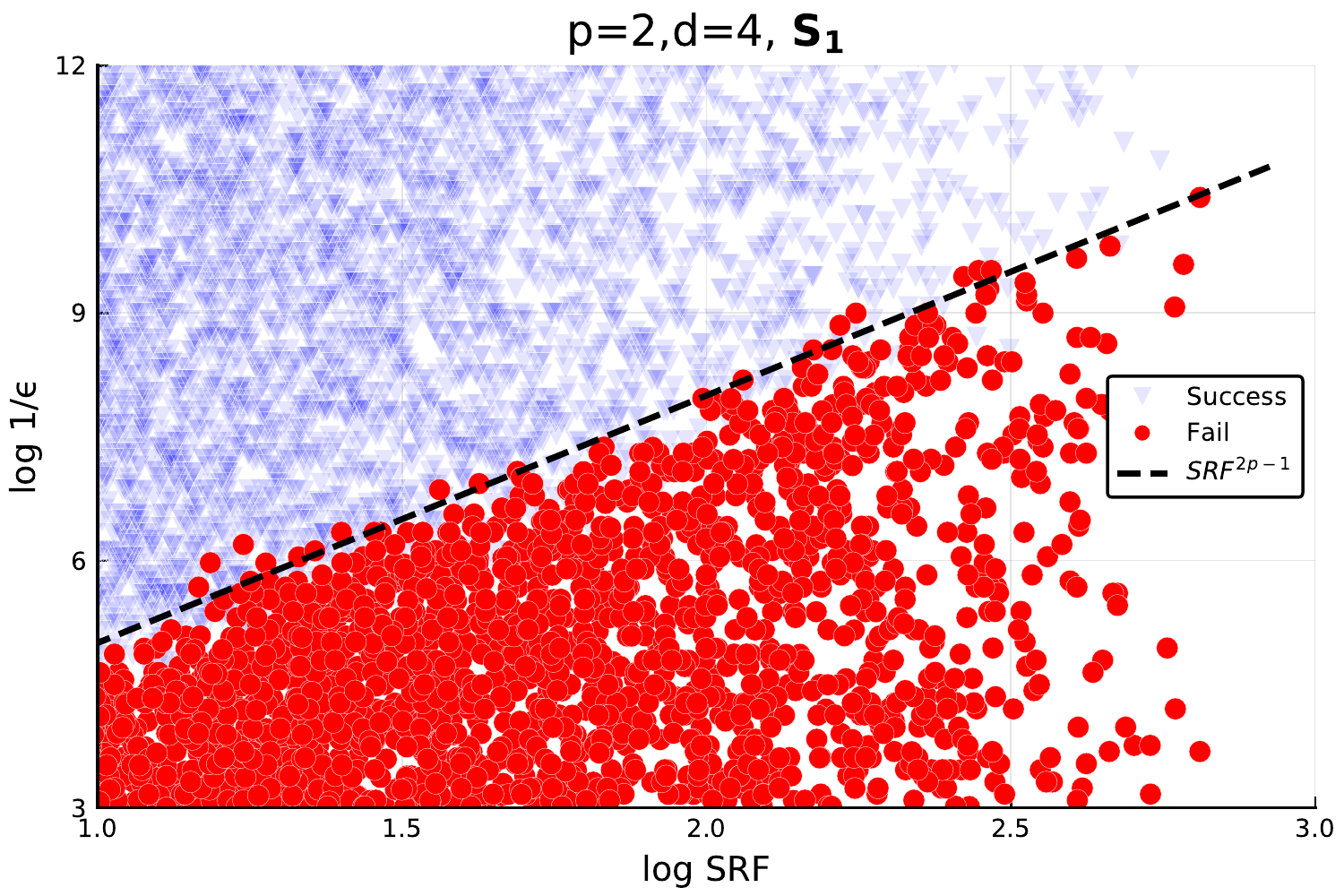}}
  \subfloat{
    \includegraphics[width=0.45\linewidth]{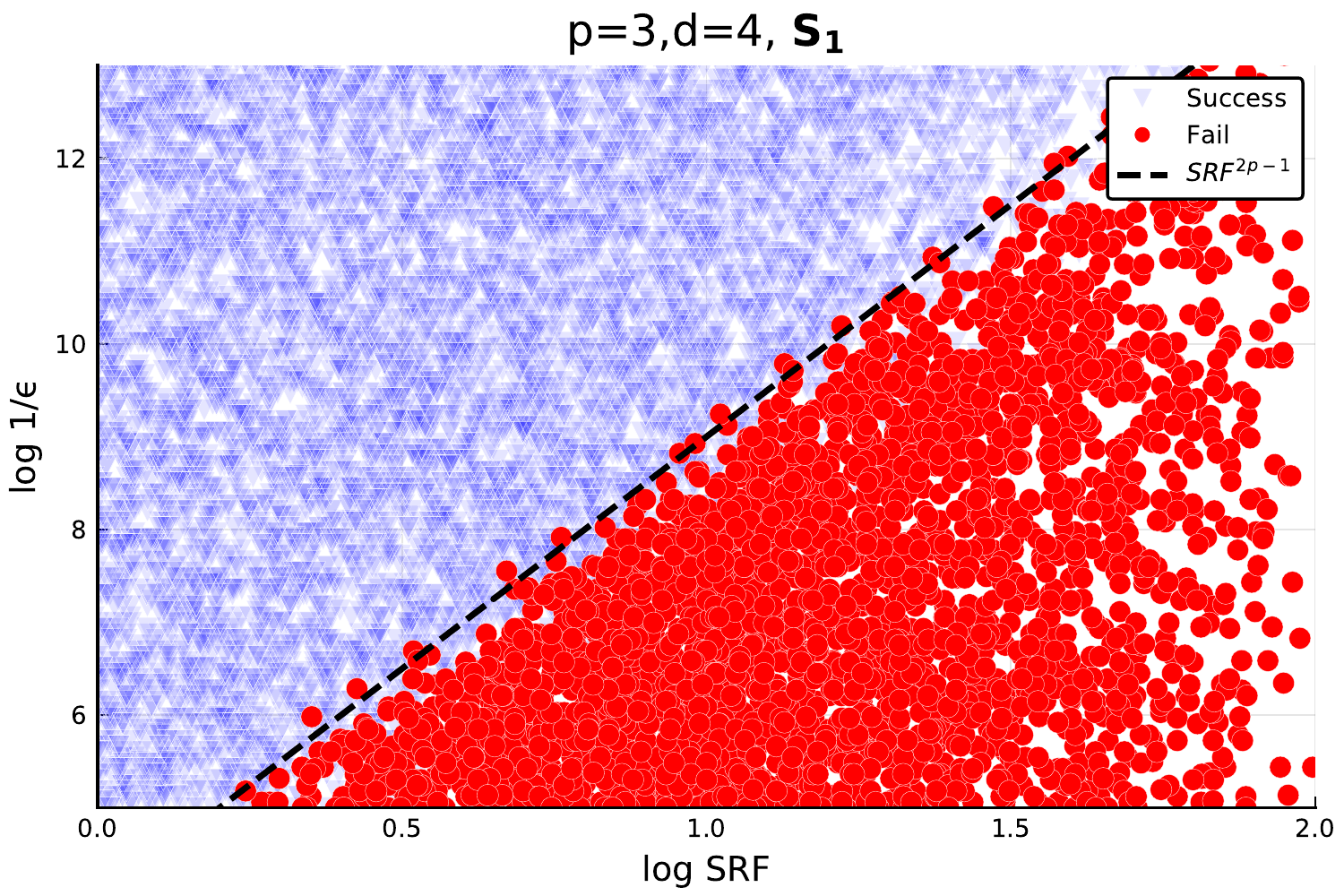}
  }
  \captionsetup{font=small,width=\linewidth}
  \caption{Phase transition for successful recovery, random bounded
    perturbations (scheme \ref{s1}) with $d=4$ and $p=2,3$. Each
    experiment is represented by either a blue triangle (if the
    recovery was successful, i.e. $Succ_j==True,\;\forall j=1,\dots,d$
    as returned by Algorithm \ref{alg.single.experiment}) or a red
    circle otherwise. The relationship
    $\epsilon_{crit} \approx \SRF^{1-2p}$ for the critical value of
    $\epsilon$ is confirmed.}
  \label{fig:phase.tran.1}
\end{figure}

Although not covered by our current theory, it is of interest to
establish the recovery threshold for every node separately. In Figure
\ref{fig:phase.tran.2} we can see that for a non-cluster node, the
threshold is approximately constant (i.e. does not depend on the
$\SRF$) -- even though Theorem \ref{thm.accuracy.bounds.upper}
requires $\epsilon\lessapprox\SRF^{1-2p}$.

\begin{figure}[hbt]
  \centering
  \includegraphics[width=0.5\linewidth]{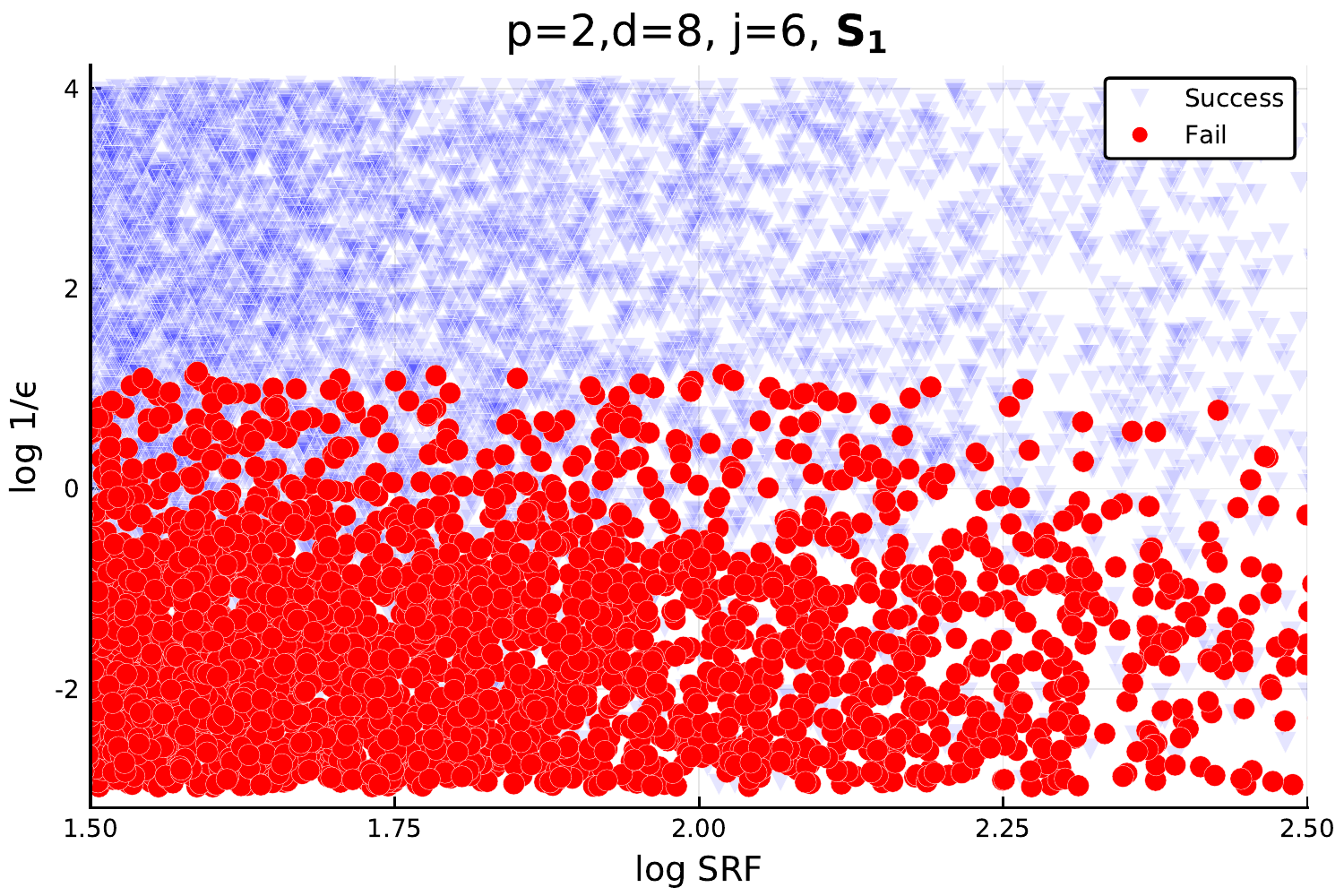}
  \captionsetup{width=\linewidth,font=small}
  \caption{Phase transition for successful recovery of a non-cluster
    node. Comparing with Figure \ref{fig:phase.tran.1}, the threshold
    is approximately constant $\epsilon_{crit}\approx const$. Here
    $p=2,d=8$, scheme \ref{s1}, plotted is the successful recovery of
    the node at index $j=6$.}
  \label{fig:phase.tran.2}
\end{figure}

\section{Normalization}
\label{sec:shift.scale}

In the intermediate claims,
instead of considering a general signal $F=(\av,\xv)\in\P_d$,
we shall usually assume that the node vector $\xv=(x_1,\dots,x_d)$ is normalized 
to the interval $\left[-{1\over 2},{1\over 2}\right]$, and
centered around the origin, i.e. $x_d=-x_1$. Let us briefly argue how
to obtain the general result from this special case.

Let us define the scale and shift transformations on $\P_d$.
\begin{definition}
  For $F=\sum_{j=1}^d
  a_j \delta(x-x_j) \in \P_d$ and $\alpha\in\R$, we define
  $SH_{\alpha}: \P_d \rightarrow \P_d$ as follows:
  $$SH_\alpha(F)(x)= \sum_{j=1}^d a_j \delta(x-(x_j-\alpha)).$$
\end{definition}

\begin{definition}
  For $F=\sum_{j=1}^d a_j
  \delta(x-x_j) \in \P_d$ and $T>0$, we define $SC_{T}: \P_d
  \rightarrow \P_d$ as follows:
  $$SC_T(F)(x) = \sum_{j=1}^d a_j \delta\left(x-\frac{x_j}{T}\right).$$
\end{definition}

  By the shift property of the Fourier transform, for any $\epsilon,\Omega >0$, we have that
  \begin{equation}\label{eq.upper.bound.shift}
    SH_{\alpha}(\EO) = E_{\epsilon,\Omega}(SH_{\alpha}(F)).
  \end{equation}
  
  \smallskip
  
  By the scale property of the Fourier transform we have that for any $\epsilon >0$,
  \begin{equation}\label{eq.upper.bound.scale}
    SC_T(E_{\epsilon,\Omega}(F))=E_{\epsilon,\Omega T}(SC_{T}(F)).
  \end{equation}

  Thus we have the following.

  \begin{proposition}\label{prop:normalization}
    Let $F=(\av,\xv)\in\P_d$, $\alpha\in \R$ and $T>0$. Then
    for any $\epsilon>0$ and $1\leq j \leq d$ we have
    \begin{align}
      \label{eq:normalized.diam.x}
      diam(E^{\xv,j}_{\epsilon,\Omega}(F))=T diam\left(E^{\xv,j}_{\epsilon,\Omega T}\left(SC_T\left(SH_{\alpha}(F)\right)\right)\right) \\
      \label{eq:normalized.diam.a}
            diam(E^{\av,j}_{\epsilon,\Omega}(F))= diam\left(E^{\av,j}_{\epsilon,\Omega T}\left(SC_T\left(SH_{\alpha}(F)\right)\right)\right)
    \end{align}
  \end{proposition}

\section{Upper bounds}\label{sec.upper.bound}

\subsection{Overview of the proof}
The proof of  Theorem \ref{thm.accuracy.bounds.upper},
presented in the next subsections and some of the appendices, is
somewhat technical. In order to help the reader, we provide an overview
of the essential ideas and steps.

The main object of the study, the error set $\EO\subset\P_d$, is the
 pre-image of an (infinite-dimensional) $\epsilon$-cube in the data
space, under the Fourier transform mapping ${\cal F}$ (recall
\eqref{eq.fourier.tr} and Definition \ref{def.error.set}). However, it
is not obvious how to obtain quantitative estimates on ${\cal F}^{-1}$
directly. Thus we replace ${\cal F}$ with certain finite-dimensional
sampled versions of it, denoted
$FM_{\lambda}:\P_d \to {\mathbb C}^{2d}$, where the sampling parameter
$\lambda$ defines the rate at which $2d$ equispaced samples of
${\cal F}(F)$ are taken. The pre-images of $\epsilon$-cubes under
$FM_{\lambda}$ define the corresponding $\lambda$-error sets
$E_{\epsilon,(\lambda)}\subset \P_d$, and in fact the original $\EO$ is
contained in the intersection of all the
$E_{\epsilon,(\lambda)}$. Thus, it is sufficient to bound the diameter
of a single such $E_{\epsilon,(\lambda^*)}$ (see remark in the next
paragraph) with a carefully chosen $\lambda^*$ so that the result will
be as small as possible. Such quantitative estimates are obtained by
careful analysis of the row-wise norms of the Jacobian matrix of
$FM_{\lambda^*}^{-1}$ and applying the so-called {\it quantitative
inverse function theorem} (Theorem \ref{thm.inverse.function}). Using
these estimates, the optimal $\lambda^*$ is shown to be on the order
of $\Omega$, from which the upper bounds of Theorem
\ref{thm.accuracy.bounds.upper} follow.

An additional technical complication arises from the fact that
$FM_{\lambda}^{-1}$ defines a multivalued mapping, and the full
pre-image $E_{\epsilon,(\lambda)}$ contains multiple copies of a
certain ``basic'' set $A=A_{\epsilon,\lambda}$. However, when
considering the intersection of all $E_{\epsilon,(\lambda)}$'s, the
non-zero shifts for certain different $\lambda$'s do not intersect,
and therefore eventually only the diameter of the basic set $A$
needs to be estimated.

Below is a brief description of the different intermediate results,
and the organization of the remainder of Section \ref{sec.upper.bound}.

\begin{enumerate}
\item In Subsection \ref{subsection.norm.estimates} we formally define
  the $\lambda$-decimated maps $FM_{\lambda}$, the corresponding error
  sets $E_{\epsilon,(\lambda)}$, and provide quantitative estimates
  on the Jacobian of $FM_{\lambda}^{-1}$ in Proposition
  \ref{prop.uniform.jacobian.bounds} (proved in Appendix
  \ref{appendix.norm.bounds}). These bounds essentially depend on the
  ``effective separation'' of each node in $\xv$ from its neighbours,
  after a blowup by a factor of $\lambda$.
\item In Subsection \ref{section.decimation} we show that for a signal
  $F=(\av,\xv)$, there exist a certain range of admissible
  $\lambda$'s, denoted by $\Lambda(\xv)$, for which the effective
  separation (see previous item) between the nodes in $\xv^c$ is of
  the order of $\Omega h$, while for the rest of the nodes, it is
  bounded from below by a constant independent of $\Omega,h$. These
  estimates are proved in Proposition \ref{prop.good.blowup}.
\item In Subsection \ref{sec.error.sets} we study in detail the
  geometry of the error sets $E_{\epsilon,(\lambda)}$ for
  $\lambda\in\Lambda(\xv)$. First, we consider (in Subsection
  \ref{sec.local}) the local inverses $FM_{\lambda}^{-1}$. For each
  $\lambda\in\Lambda(\xv)$, we show that the local inverse exists in a
  neighborhood $V$ of radius $R\approx \left(\Omega h\right)^{2p-1}$
  around $FM_{\lambda}(F)$, and provide estimates on the Lipschitz
  constants of $FM_{\lambda}^{-1}$ on $V$ and the diameter of
  $FM_{\lambda}^{-1}(V)$. The main bounds to that effect are proved in
  Proposition \ref{prop.lipshitz}, using the previously established
  general estimates from Proposition
  \ref{prop.uniform.jacobian.bounds} and the quantitative inverse
  function theorem (Theorem \ref{thm.inverse.function}).
\item Next, denoting $A=A_{R,\lambda}=FM_{\lambda}^{-1}(V)$, we show
  in Proposition \ref{prop.global.geo} that the set
  $E_{\epsilon,(\lambda)}$ is a union of certain copies of $A$, where
  each such copy is obtained by shifting the nodes in $A$ by an
  integer multiple of $\lambda^{-1}$, and/or by permuting them.
\item In Subsection \ref{sec:proof-upper-bound} we complete the
  proof. At this point we consider the entire set $\Lambda(\xv)$.  The
  main technical step, Proposition \ref{prop.elimiate.period} (proved
  in Appendix \ref{appendix.upper.bound.tech}), establishes that for a
  certain $\lambda^*\in\Lambda(\xv)$ and all possible permutations
  $\pi$ and shifts $\ell\in\mathbb{Z}\setminus\{0\}$, there exists a
  particular $\bar{\lambda} = \bar{\lambda}(\pi,\ell)\in\Lambda(\xv)$
  such that the intersection between $\pi$-permutation and
  $\ell$-shift of $A_{R,\gs}$ and the \emph{entire error set}
  $E_{R,(\bar{\lambda})}$ is empty. From this fact it immediately
  follows that the original error set $\EO$ with $\epsilon=R$ is
  contained in $A_{R,\gs}$ (Proposition \ref{prop.good.lambda}). The
  proof is finished by invoking the previously established estimates
  on the diameter of $A_{R,\gs}$ and its projections.
\end{enumerate}

\begin{remark}
  We expect that the tools developed throughout the proof will also be
  useful to calculate the minimal finite sampling rate required to
  achieve the minimax error rate stated in Theorem
  \ref{thm.accuracy.bounds.upper}.
\end{remark}

\subsection{\texorpdfstring{$\lambda$-decimation maps}{Decimation maps}}\label{subsection.norm.estimates}
For the purpose of the following analysis, we extend the space of signals $\P_d$ to include signals
with complex nodes and denote the extended space by $\bar{\P}_d$,
$$
	\bar{{\cal P}}_d=\left\{(\av,\xv) : \av=(a_1,\ldots,a_d)\in {\mathbb C}^d,\; \xv=(x_1,\ldots,x_d)\in {\mathbb C}^{d} \right\}.
$$

We will be considering specific sets of exactly
$2d$ samples of the Fourier transform, made at constant rate
$\lambda$ as follows.
\begin{definition}\label{def.partial.fourier}
  For $\lambda > 0$, we define the map $FM_\lambda:
  \barP \cong \C^{2d} \rightarrow \C^{2d}$ by
  $$FM_\lambda((\av,\xv)) = \vec{\mu} = \left(\mu_0,\dots,\mu_{2d-1}\right),\; \mu_k = \sum_{j=1}^d a_j e^{2\pi i x_j \lambda k},\
  k=0,\ldots,2d-1.$$
\end{definition} 
We call such map a $\lambda$-decimation map.

\smallskip

For $\lambda>0$ and $\epsilon > 0$, we define the
corresponding error set $E_{\epsilon,(\lambda)}$ as follows.
\begin{definition}\label{def.error.set.lambda}
  The error set $E_{\epsilon,(\lambda)}(F) \subset \P_d$ is the set
  consisting of all the signals $F'\in \P_d$ with
  \begin{align*}
    \left\|\FMG(F')- \FMG(F) \right\|\le \epsilon.
  \end{align*} 
  Similarly we denote by
  $E^{\av,j}_{\epsilon,(\lambda)}(F), \
  E^{\xv,j}_{\epsilon,(\lambda)}(F)$ the projection of the error set
  $E_{\epsilon,(\lambda)}(F)$ onto the corresponding amplitudes and
  the nodes components (compare \eqref{eq:individ.err.set.def}).
\end{definition}

\smallskip

Now consider the given spectrum
${\cal F}(F)(s),\ s \in [-\Omega,\Omega]$. Clearly for each
$\lambda \le \frac{\Omega}{2d-1}$ we have that
$E_{\epsilon,\Omega}(F)\subseteq E_{\epsilon,(\lambda)}(F)$ giving
\begin{equation}\label{eq:errorset.via.decimated.errorsets}
  E_{\epsilon,\Omega}(F) \subseteq \bigcap_{\lambda \in (0, \frac{\Omega}{2d-1}]} E_{\epsilon,(\lambda)}(F).
\end{equation}

Hence, to prove the upper bounds in Theorem
\ref{thm.accuracy.bounds.upper}, we shall show that there exists a \emph{certain} 
subset $S \subseteq \left(0, \frac{\Omega}{2d-1}\right]$
such that for each $\lambda \in S$,
$diam \big(E_{\epsilon,(\lambda)}(F)\big)$ can be effectively
controlled.

In the next proposition, we derive a uniform bound on the norms of the
inverse Jacobian of $FM_{\lambda}$ near a signal with clustered
nodes. The bounds explicitly depend on the distances between the
so-called ``mapped'' nodes $z_j(\lambda)=e^{2\pi i \lambda x_j}$.

\begin{proposition}[Uniform Jacobian bounds]\label{prop.uniform.jacobian.bounds}
  Let $F=(\av,\xv) \in \barP$, $\av=(a_1,\dots,a_d)$,
  $\xv=(x_1,\ldots,x_d)$ and for $\lambda > 0$ let
  $z_1=e^{2\pi i \lambda x_1},\ldots,z_{d}=e^{2\pi i \lambda x_d
  }$. Suppose that for each $j=1,\ldots,d$, we have
  $0< \frac{m}{2} \le |a_j|$ and $\frac{1}{2}\le |z_j| \le 2$ for some
  $m>0$.
  
  Further assume that for $\tilde{\eta}, \tilde{h}$ with
  $1 \ge \tilde{\eta} \ge \tilde{h}$,
  and $\xv^{c}=\{ x_\kappa,\ldots,x_{\kappa+p-1}\}\subset \xv$, $p\ge 2$, the nodes 
  $z_1,\ldots,z_d$ satisfy:
  \begin{enumerate}
  \item For each $x_j, x_k \in \xv^{c}, j\ne k$, we have that $|z_j-z_k| \ge \tilde{h}$.
  \item For each $x_{\ell} \in \xv \setminus \xv^{c}$ and
    $x_j\in \xv$, $\ell \ne j$, we have that
    $|z_{\ell}-z_j|\ge \tilde{\eta}$.
  \end{enumerate}
  Then the Jacobian matrix of $FM_{\lambda}$ at $F$, denoted by $J_{\lambda}(F)$, is
  non-degenerate. Furthermore, write the inverse Jacobian matrix $J^{-1}_\lambda(F)$ in the following block form 
  $J^{-1}_\lambda(F)=\begin{bmatrix} A\\ \tilde{B} \end{bmatrix}$, where $A,\tilde{B}$ are $d\times 2d$.
  Then, the $\ell_1$ norms of the rows of the blocks $A,\tilde{B}$ are bounded as follows:
  \begin{align}
    \sum_{k=1}^{2d} |A_{j,k}| &\le K_1(\tilde{\eta},d,p),
    & x_j \in \xv \setminus \xv^c,\label{eq.prop.1}\\
    \sum_{k=1}^{2d} |\tilde{B}_{j,k}| &\le K_2(m,\tilde{\eta},d,p)\frac{1}{\lambda}, 
    & x_j \in \xv \setminus \xv^c,\label{eq.prop.2}\\
    \sum_{k=1}^{2d} |A_{j,k}| &\le K_3(\tilde{\eta},d,p)\tilde{h}^{-2p+1}, & x_j \in \xv^c,\label{eq.prop.3}\\
    \sum_{k=1}^{2d} |\tilde{B}_{j,k}| 
                              & \le K_4(m,\tilde{\eta},d,p)\frac{1}{\lambda}\tilde{h}^{-2p+2}, & x_j \in
                              \xv^c,\label{eq.prop.4}
  \end{align}
  where $K_1(\cdot,\ldots,\cdot),K_2(\cdot,\ldots,\cdot),K_3(\cdot,..,,\cdot),K_4(\cdot,\ldots,\cdot)$ are constants
  depending only on the parameters inside the brackets.
\end{proposition} 
The proof of Proposition \ref{prop.uniform.jacobian.bounds} is given
in Appendix \ref{appendix.norm.bounds}.

\subsection{The existence of an admissible decimation}\label{section.decimation}

In this section we shall prove the existence of a certain blowup factors $\lambda$, 
such that the mapped nodes $\{e^{2\pi i \lambda x_j}\}$ 
(see Proposition \ref{prop.uniform.jacobian.bounds}
above) attain ``good'' separation properties. This result will later 
be used to show that for any such $\lambda$, the corresponding inverse
$\lambda$-decimation map $FM_{\lambda}^{-1}$ will have the smallest possible 
coordinatewise Lipschitz constants with respect to $\Omega,h$
(up to constants) (see Proposition \ref{prop.uniform.jacobian.bounds}). 

\begin{definition}
  
For each $x \in \R$ and $a>0$ consider the operation
$\mod{\left(-\frac{a}{2},\frac{a}{2}\right]}$ defined as
$$x \mod{\left(-\frac{a}{2},\frac{a}{2}\right]} = x-ka,$$
where $k$ is the unique integer such that
$x-ka \in \left(-\frac{a}{2},\frac{a}{2}\right]$.  Using this
notation the principal value of the complex argument function is
defined as
$$\Arg(r e^{i\theta}) = \theta \mod{(-\pi,\pi]},$$
for each $\theta \in \R$ and $r>0$.
\end{definition}

\begin{definition}
  For $\alpha,\beta\in\mathbb{C}\setminus\{0\}$, we define the angular distance
  between $\alpha,\beta$ as
  $$
  \angle(\alpha,\beta)=
  \left|\Arg\left(\frac{\alpha}{\beta}\right)\right|=
  \biggl|\left(\Arg (\alpha) - \Arg(\beta) \right) \ \mod{
    (-\pi,\pi]}\biggr|,
  $$
  where for $z \in \C\setminus\{0\}$, $\Arg(z) \in (-\pi,\pi]$ 
  is the principal value of the argument of $z$.
\end{definition}

\begin{lemma}\label{lem.ang.dist.eucl.}
For $|x|=|y|=1$, we have  
\begin{equation}\label{eq.angle.to.euclidien}
  {2\over\pi} \angle(x,y) \leq |x-y| \leq \angle(x,y).
\end{equation}
\end{lemma}
\begin{proof}
  First,
  $$
  |x-y| = \left|1-{x\over y}\right| = 2\sin\left|{1\over 2}\Arg {x\over y}\right|=2\sin\left|{{\angle(x,y)}\over 2}\right|.
  $$
  Then use the fact that for any $|\theta|\leq{\pi\over 2}$ we have
  $$
  {2\over\pi}|\theta| \leq \sin|\theta|\leq|\theta|\qedhere.
  $$
\end{proof}

Let $F=(\av,\xv) \in \P_d$ such that the node vector
$\xv=(x_1,\ldots,x_d)$ forms a $(p,h,T,\tau,\eta)$-clustered configuration, with
$\xv^c=\{x_\kappa, x_{\kappa+1},\ldots,\allowbreak x_{\kappa+p-1}\}$.
According to Proposition \ref{prop.uniform.jacobian.bounds}, the
the norms of the rows of the inverse Jacobian $J^{-1}_\lambda(F)$ 
essentially depend on the the minimal distance between the
mapped nodes $z_j(\lambda)=e^{2\pi i \lambda x_j}$. After a blowup
by a factor of $\lambda\leq {1\over {2h}}$, the pairwise angular
distances $\angle\left(\cdot,\cdot\right)$ (and hence the euclidean
distances) between the mapped cluster-nodes
$z_{\kappa},\ldots,z_{\kappa+p-1}$ are now of order $\lambda h$.

\smallskip

On the other hand, the non-cluster nodes are at
distance larger than $\eta T\gg h$.  Therefore, after the
blowup by $\lambda$, the
non-cluster nodes
$z_1,\ldots,z_{\kappa-1},z_{\kappa+p},
\ldots,\allowbreak z_d$ may in principle be located
anywhere on the unit circle.  For example, any of these
mapped non-cluster nodes might coincide with, or be very
close to, a certain
mapped cluster node, or yet another 
mapped non-cluster node.

While this situation might occur for some values of $\lambda$, we will
now show that there exist certain sets of $\lambda$'s for which this
does not happen. We shall require the following key estimate
concerning the pairwise angular distance between any two mapped nodes.

  \begin{lemma}[A uniform blowup of two
    nodes]\label{lemma.uniform.blowup}
    Let $x_j,x_k \in \R, \ x_j \ne x_k$,
    and let $\Delta = |x_j-x_k|$.
    Consider the following blowups $z_j=z_j(\lambda)=e^{2\pi i \lambda x_j},
    z_k=z_k(\lambda)= e^{2\pi i \lambda x_k}$.  Then
    for $0 \le \alpha \le \pi$ and an interval $I=[a,b]\subset \R$, the set 
    \begin{equation}\label{eq:sig.j.k.def}
    \Sigma^{\alpha}_{j,k}(I)=\left\{\lambda \in I :\angle\big(z_j(\lambda),z_k(\lambda)\big) \le
      \alpha \right\}
  \end{equation}
  is a union of $N$ intervals $I_1,\ldots,I_N$ with $\left\lfloor
      |I|\Delta\right\rfloor \le N \le \left\lfloor |I|\Delta\right\rfloor +1$,
    and 	
    $$|I_j| \le \frac{\alpha}{\pi} \frac{1}{\Delta}, \tab j=1,\ldots,N.$$ 	 

  \end{lemma}
  \smallskip
  \begin{proof}
    For each $\lambda \in I$ we have 
    \begin{equation}\label{eg.blowup.tow.nodes.angle}
      \angle(z_{j}(\lambda),z_{k}(\lambda))=
      \left|\Arg\left(\frac{z_j(\lambda)}{z_{k}(\lambda)}\right)\right|=\left|\Arg(e^{2\pi i \lambda \Delta})\right|.
    \end{equation}

    By equation \eqref{eg.blowup.tow.nodes.angle} we have 
    \begin{align*}
      \left\{\lambda \in I :\angle\big(z_j(\lambda),z_k(\lambda)\big) \le \alpha \right\}=&\\
      \left\{\lambda \in I : \left|\Arg(e^{2\pi i \lambda \Delta})\right| \le \alpha \right\}=&\\
      \left\{\lambda \in I : \left|2\pi \lambda \Delta \mod{\left(-\pi,\pi \right]}\right| \le \alpha
      \right\}=&\\
      \left\{\lambda \in I : -\alpha \le \left(2\pi \lambda \Delta \mod{\left(-\pi,\pi \right]}\right) \le \alpha
      \right\}=&\\
      \left\{\lambda \in I : -\frac{\alpha}{2\pi} \frac{1}{\Delta} \le \left(\lambda \mod{
      \left(-\frac{1}{2\Delta},\frac{1}{2\Delta}\right]} \right) \le \frac{\alpha}{2\pi}
      \frac{1}{\Delta}\right\}.
    \end{align*}
    The last set above can be written as $I\cap S^{\alpha}$ where
    \begin{equation}
      S^{\alpha}=\left\{\lambda \in \R : -\frac{\alpha}{2\pi} \frac{1}{\Delta} \le \left(\lambda \mod{
          \left(-\frac{1}{2\Delta},\frac{1}{2\Delta}\right]} \right) \le \frac{\alpha}{2\pi}
        \frac{1}{\Delta}\right\}.
    \end{equation}
    Define the interval $I^{\alpha} = \left[-\frac{\alpha}{2\pi}
      \frac{1}{\Delta},\frac{\alpha}{2\pi} \frac{1}{\Delta}\right]$.
    Then the set $S^{\alpha}$ is a union of intervals of length $\frac{\alpha}{\pi}\frac{1}{\Delta}$
    as follows
    $$
    S^{\alpha}=\bigcup_{\ell\in \mathbb{Z}} \left(I^{\alpha} +
      \frac{\ell}{\Delta}\right) = \bigcup_{\ell\in \mathbb{Z}} \left\{ \lambda
      + \frac{\ell}{\Delta} : \lambda \in I^{\alpha} \right\}.
    $$

    The intersection of $S^{\alpha}$ with any interval $I$ is then a union of 
    $\left\lfloor |I|\Delta\right\rfloor \le N \le \left\lfloor |I|\Delta\right\rfloor +1$
    intervals of length smaller or equal to $\frac{\alpha}{\pi}\frac{1}{\Delta}$.
    This concludes the proof of Lemma \ref{lemma.uniform.blowup}.
  \end{proof}

\smallskip

Now we state and prove the main result of this subsection.

\begin{proposition}\label{prop.good.blowup}
  Let
  $F=(\av,\xv) \in \P_d, \ \xv=(x_1,\ldots,x_d) \subset
  [-\frac{1}{2},\frac{1}{2}]$, such that $\xv$ forms a
  $(p,h,1,\tau,\eta)$-clustered configuration with
  $\xv^c=\{x_\kappa, x_{\kappa+1},\ldots,\allowbreak
  x_{\kappa+p-1}\}$.
  
  Let $\Omega \le \frac{2d-1}{2}\cdot{1\over h}$. For each
  $\lambda > 0 $ let
  $z_{1}(\lambda)=e^{2\pi i \lambda
    x_1},\ldots,z_{d}(\lambda)=e^{2\pi i \lambda x_{d}}$.
  
  Then each interval $I \subset\left[\frac{1}{2}\frac{\Omega}{2d-1},\frac{\Omega}{2d-1}\right]$ of length 
  $|I| = \frac{1}{\eta}$ contains a sub-interval $I' \subset I$ of length $|I'| \ge (2 d^2 \eta)^{-1}$ such
  that for each $\lambda \in I'$:
  \begin{enumerate}
  \item For all $x_{\ell}\in \xv\setminus \xv^c$ and $x_j\in \xv$, $x_j\ne x_{\ell}$,
    \begin{align}
      \angle(z_{\ell}(\lambda),z_j(\lambda))&\ge \frac{1}{d^2}.\label{eq.blowup.prop.non.cluster}
    \end{align}
  \item For all $x_j,x_k\in \xv^c, x_k \ne x_j,$
    \begin{align}
      \angle(z_j(\lambda),z_k(\lambda))&\ge 2\pi \lambda \tau h \ge \frac{\pi \tau}{2d-1}\Omega h.\label{eq.blowup.prop.cluster}
    \end{align}
  \end{enumerate} 			
\end{proposition}
\begin{proof}

  Let us first prove that assertion \eqref{eq.blowup.prop.cluster} holds
  for any
  $\frac{1}{2}\frac{\Omega}{2d-1} \le \lambda \le
  \frac{\Omega}{2d-1}$. 

  \smallskip

  Let $x_j, x_k$, $j>k$, be two cluster nodes.  The
  angular distance between the mapped cluster nodes
  $z_j = z_j(\lambda)=e^{2\pi i \lambda x_j},\ z_k = z_k(\lambda) = e^{2\pi i
    \lambda x_k}$, is
  $$\angle(z_j,z_k) =\left| \Arg(e^{2\pi i \lambda
    (x_j-x_k)})\right|.$$

  By assumption
  $\Omega h \le \frac{2d-1}{2}$, then 
  $\lambda \le \frac{1}{2h}$ and then
  $0 \le 2\pi \lambda (x_j-x_k) \le 2\pi \lambda h \le \pi$. With this
  we have
  $$\angle(z_j,z_k)=2\pi \lambda (x_j-x_k)\ge 2\pi \lambda \tau h.$$
  By assumption
  $ \lambda \ge \frac{1}{2}\frac{\Omega}{2d-1}$. Then,
  $\angle(z_j,z_k) \ge \frac{\pi \tau}{2d-1} \Omega
  h$. This concludes the proof of assertion \eqref{eq.blowup.prop.cluster}.

  Using Lemma \ref{lemma.uniform.blowup} we now
  prove that assertion
  \eqref{eq.blowup.prop.non.cluster} holds for any
  interval $I=[a,b] \subset \R$ of length
  $|I|=\frac{1}{\eta}$.  Let $I$ be such an
  interval. For each $0<\alpha \le \pi$ consider the
  set
  $$\Sigma^{\alpha}(I) = \left\{\lambda \in I : \exists x_{\ell} \in
    \xv\setminus \xv^c \text{ s.t. } \min_{1 \le j \le d,
      j\ne \ell}\angle(z_{\ell}(\lambda),z_j(\lambda))\le \alpha \right\}.$$ 

  We then have
  $$
  \Sigma^{\alpha}(I)=\bigcup_{x_{\ell} \in \xv\setminus\xv^c }\bigcup_{x_j \ne x_{\ell}}
  \Sigma^{\alpha}_{\ell,j}(I),
  $$
  where $\Sigma^{\alpha}_{\ell,j}$ are given by
  \eqref{eq:sig.j.k.def}.  By Lemma \ref{lemma.uniform.blowup} each
  $\Sigma^{\alpha}_{\ell,j}(I)$ above is a union of at most
  $\left\lfloor |I| \eta\right\rfloor +1 = 2$ intervals, the
  length of each interval is at most
  $\frac{\alpha}{\pi} \frac{1}{\eta}$.  Therefore $\Sigma^{\alpha}(I)$
  is a union of at most $K=\binom{d}{2}2=d(d-1) $ intervals.
  Moreover, let $\nu$ denote the Lebesgue measure on $\R$, then
  \begin{equation}\label{eq.number.of.bad.intervals}
    \nu(\Sigma^{\alpha}(I)) \le K \frac{\alpha}{\pi} \frac{1}{\eta} \le d(d-1)
    \frac{\alpha}{\pi} \frac{1}{\eta}\le d^2 \alpha \frac{1}{2\eta}.
  \end{equation}
  
  Put $\alpha' = \frac{1}{d^2}$ then by \eqref{eq.number.of.bad.intervals}
  \begin{equation}\label{eq.measure.bad}
    \nu(\Sigma^{\alpha'}(I)) \le \frac{1}{2\eta}.
  \end{equation}
  Now consider the complement set of $\Sigma^{\alpha'}(I)$ with respect to $I$, 
  $$
  (\Sigma^{\alpha'}(I))^c = \left\{\lambda \in I : \forall x_{\ell} \in \xv\setminus \xv^c, \min_{1 \le j \le d,
      j\ne \ell}\angle(z_{\ell}(\lambda),z_j(\lambda)) > \frac{1}{d^2} \right\}.
  $$ 
  By \eqref{eq.measure.bad}
  \begin{equation}\label{eq.measure.good}
    \nu\big((\Sigma^{\alpha'}(I))^c\big) \ge |I| - \frac{1}{2\eta} = \frac{1}{\eta} - \frac{1}{2\eta} = \frac{1}{2\eta}.					
  \end{equation}
  In addition, since $\Sigma^{\alpha'}(I)$ is a union of at most $K=d(d-1)$ intervals, 
  then $\big(\Sigma^{\alpha'}(I)\big)^c$ is a union of at most 
  \begin{equation}\label{eq.number.of.good.intervals}
    L=K+1=d(d-1)+1\le d^2
  \end{equation}
  intervals.
  Using \eqref{eq.measure.good} and \eqref{eq.number.of.good.intervals}, the average size of these intervals is bounded as follows: 
  $$\frac{\nu\big((\Sigma^{\alpha'}(I))^c\big)}{L} \ge \frac{1}{d^2} \frac{1}{2\eta}.$$
  We therefore conclude that $\big(\Sigma^{\alpha'}(I)\big)^c$ contains an interval of length greater or
  equal to
  $\frac{1}{d^2} \frac{1}{2\eta}$. This proves assertion \eqref{eq.blowup.prop.non.cluster} of Proposition
  \ref{prop.good.blowup}.
\end{proof}

\subsection{Error sets of admissible decimation maps}\label{sec.error.sets}
Throughout this section we fix a signal
$F=(\av,\xv) \in \P_d,\ \av=(a_1,\ldots,a_d), \ \xv=(x_1,\ldots,x_d)
\subset \left[-\frac{1}{2},\frac{1}{2}\right]$, such that
$\xv$ forms a $(p,h,1,\tau,\eta)$-clustered configuration, with
$\xv^c=\{x_\kappa, x_{\kappa+1},\allowbreak \ldots, x_{\kappa+p-1}\}$
and $\|\av\| \geq m > 0$. We also fix $\Omega>0$ such that
$\Omega h \le \frac{1}{20d}$.

\smallskip

Proposition \ref{prop.good.blowup} demonstrated the existence of certain 
$\lambda$-decimation maps which achieve good separation of the
non-cluster nodes. We define the set $\Lambda(\xv)$ to consist of all
such admissible $\lambda$'s, as follows. 
\begin{definition}[Admissible blowup factors]\label{def.good.blowup.set}
  For each $F=(\av,\xv) \in \P_d, \ \xv=(x_1,\allowbreak\ldots,x_d)$,
  such that $\xv$ forms a $(p,h,1,\tau,\eta)$-clustered configuration
  and $z_j = z_{j}(\lambda) = e^{2\pi i \lambda x_j},\ j=1,\ldots,d$ and
  $\Omega > 0$, we define the set of admissible blowup factors
  $\Lambda(\xv) = \Lambda_{\Omega,d}(\xv)$ as the set of all
  $\lambda \in
  \left[\frac{1}{2}\frac{\Omega}{2d-1},\frac{\Omega}{2d-1}\right]$
  satisfying:
  \begin{enumerate}
  \item For all $\ell \ne j$ such that $x_{\ell}\in \xv\setminus\xv^c$ and $x_j\in \xv$,
    \begin{align}
      \angle(z_{\ell}(\lambda),z_j(\lambda))&\ge \frac{1}{d^2}.
    \end{align}
  \item For all $j \ne k$ such that $x_j,x_k\in \xv^c$
    \begin{align}
      \angle(z_j(\lambda),z_k(\lambda))&\ge 2\pi \lambda \tau h \ge \frac{\pi }{2d-1}\Omega \tau h.
    \end{align}
  \end{enumerate} 	 
\end{definition}

\subsubsection{The local geometry of admissible decimation maps}\label{sec.local}

The next result gives an explicit description of a neighborhood around
$F$ where the map $FM_{\lambda}$ is injective (and, therefore, we can
speak about a local inverse).

\begin{definition}
For each $\alpha,\beta>0$ we denote by $H_{\alpha,\beta}(F)$ the closed polydisc
$$H_{\alpha,\beta}(F) = \left\{ (\av',\xv') \in \barP \ : \|\av'-\av\| \le \alpha,\ \|\xv'-\xv\| \le \beta \right\},$$
and by $H^{\mathrm{o}}_{\alpha,\beta}(F)$ the interior of $H_{\alpha,\beta}(F)$.
\end{definition}

The following is proved in Appendix \ref{appendix.proof.one.to.one}.

\begin{proposition}[One-to-one]\label{prop.one.to.one}
  For each $\lambda \in \Lambda(\xv)$ 
  the map $FM_{\lambda}$ is injective in the open polydisc $U = H^{\mathrm{o}}_{m,\frac{\tau h}{2 \pi}}(F) \subset
  \barP$.
\end{proposition}

Next we can estimate the Lipschitz constants of the inverse map
$FM_{\lambda}^{-1}$, using the previously established general bounds
in Proposition \ref{prop.uniform.jacobian.bounds}.
  
  \begin{proposition}\label{prop.uniform.jacobian.bounds.clustered.signal}
    Let $H=H_{\frac{m}{2},\frac{\tau h}{4 \pi }}(F) \subset U = H^{\mathrm{o}}_{m,\frac{\tau h}{2 \pi}}(F)$.
    Then, for each $F' \in H$: 
    \begin{enumerate}
      \item The Jacobian matrix of $FM_{\lambda}$ at $F'$, denoted by $J_{\lambda}(F')$, is
    	non-degenerate.
      \item Put
  		$J^{-1}_\lambda(F')=\begin{bmatrix} A\\ \tilde{B} \end{bmatrix}$, where $A,\tilde{B}$ are $d\times 2d$.
  		Then, the $\ell_1$ norms of the rows of the blocks $A,\tilde{B}$ are bounded as follows:    
	    \begin{align}
	      \sum_{k=1}^{2d} |A_{j,k}| 
	      &\le \tilde{C},&
	                       x_j \in \xv\setminus\xv^c,\label{uniform.bound.amp.non.cluster}\\
	      \sum_{k=1}^{2d} |\tilde{B}_{j,k}| 
	      & \le \tilde{C}\frac{1}{\Omega},& 
	                                        x_j \in \xv\setminus\xv^c,\label{uniform.bound.nodes.non.cluster}\\
	      \sum_{k=1}^{2d} |A_{j,k}| 
	      &\le \tilde{C}(\Omega \tau h)^{-2p+1},&
	                                         x_j \in \xv^{c},\label{uniform.bound.amp.cluster}\\
	      \sum_{k=1}^{2d} |\tilde{B}_{j,k}| 
	      & \le \tilde{C}\frac{1}{\Omega} (\Omega \tau h)^{-2p+2},& 
	                                                           x_j \in \xv^{c},\label{uniform.bound.nodes.cluster}
	    \end{align}
	    where $\tilde{C} = \tilde{C}(m,d,p)$ is a constant depending
	    only on $d,m,p$.
    \end{enumerate}
  \end{proposition}
  \begin{proof}
    Let
    $F'=(\av',\xv')\in H, \ \av'=(a'_1,\ldots,a'_d),\
    \xv'=(x_1',\ldots,x_d')$.  Let
    $z'_j=z'_j(\lambda)=e^{2\pi i \lambda x'_j}$ and let
    $z_j = z_{j}(\lambda) = e^{2\pi i \lambda x_j},\ j=1,\ldots,d$.
    
    By the integral mean value theorem, for each $j=1,\ldots,d$, 
    $$
    |z'_j-z_j| = \left|e^{2\pi i \lambda x'_j} - e^{2\pi i \lambda x_j}\right|\le \lambda\tau h.
    $$
    
    \smallskip
    
    Let $\ell \ne j$ such that $x_{\ell}\in \xv\setminus\xv^c$ and $x_j\in \xv$.
    Since $\lambda \in \Lambda(x)$,
    \begin{align*}
      \angle(z_{\ell},z_j)&\ge \frac{1}{d^2}.
    \end{align*}
    Then by \eqref{eq.angle.to.euclidien}
    $$|z_{\ell} - z_j| \ge \frac{2}{\pi d^2}.$$
    We get that
    $$|z_{\ell}' - z_j'|\ge |z_{\ell} - z_j| -|z'_{\ell} - z_{\ell}| - |z_j' - z_j| \ge |z_{\ell} - z_j| -2\lambda\tau h \ge 
    \frac{2}{\pi d^2} - 2\lambda\tau h.$$ 
    With $\Omega h \le \frac{1}{20 d}$ and $\lambda \le \frac{\Omega}{2d-1}$ by assumption,
    we have that $2\lambda\tau h \le \frac{1}{3 \pi d^2}$ then 
    $$|z_{\ell}' - z_j'| \ge \frac{2}{\pi d^2} - 2\lambda\tau h \ge \frac{2}{\pi d^2} - \frac{1}{3\pi d^2} \ge
    \frac{1}{2d^2}.$$ We conclude that for each $\ell \ne j$ such that
    $x_{\ell}\in \xv\setminus\xv^c$ and $x_j\in \xv$
    \begin{equation}\label{eq.non.cluster.z.euclidian.seperation}
      |z_{\ell}' - z_j'| \ge \frac{1}{2d^2}.
    \end{equation}
    
    Let $j \ne k$ such that $x_j,x_k\in \xv^c$. $\lambda \in \Lambda(\xv)$ then
    \begin{align*}
      \angle(z_j,z_k)&\ge 2\pi \lambda \tau h.
    \end{align*}
    Then by \eqref{eq.angle.to.euclidien}
    $$|z_j - z_k| \ge 4\lambda \tau h.$$
    With a similar argument as above, we get that
    $$|z_j' - z_k'|\ge |z_j - z_k| - 2\lambda\tau h \ge 2\lambda \tau h.$$
    Using
    $\lambda \in \Lambda(\xv) \Rightarrow \lambda \ge
    \frac{\Omega}{2(2d-1)}$, we conclude that for each $j \ne k$ such
    that $x_j,x_k\in \xv^c$
    \begin{equation}\label{eq.cluster.z.euclidian.seperation}
      |z_j' - z_k'|\ge 2\lambda \tau h \ge \frac{1}{2d-1}\Omega \tau h.		
    \end{equation}
    
    Now using \eqref{eq.non.cluster.z.euclidian.seperation} and \eqref{eq.cluster.z.euclidian.seperation} 
    we invoke Proposition \ref{prop.uniform.jacobian.bounds} with $\tilde{h} = \frac{1}{2d-1}\Omega \tau h$ and
    $\tilde{\eta} = \frac{1}{2d^2}$ and as a result prove Proposition \ref{prop.uniform.jacobian.bounds.clustered.signal} with
    \begin{align*}
      \tilde{C} = \left(2d-1\right)^{2p-1} \max
      \bigg[K_1\left(\frac{1}{2d^2},d,p\right),K_2\left(m,\frac{1}{2d^2},d,p\right),\\K_3\left(\frac{1}{2d^2},d,p\right),K_4\left(m,\frac{1}{2d^2},d,p\right) \bigg].
    \end{align*}
  \end{proof}

\begin{definition}\label{def.r.cube}
	For $\vec{v} \in \C^{d}$ and $r>0$, we denote by $Q_r(\vec{v})$ the closed cube
	of radius $r$ centered at $\vec{v}$:
	$$Q_{r}(\vec{v}) = Q_{r,d}(\vec{v}) = \left\{ \vec{u} \in \C^{d} : \|\vec{u}-\vec{v}\| \le r \right\}.$$
\end{definition}

\begin{proposition}\label{prop.lipshitz}
  Let $U = H^{\mathrm{o}}_{m,\frac{\tau h}{2 \pi}}(F)$ and $H=H_{\frac{m}{2},\frac{\tau h}{4 \pi }}(F) \subset U$. 
  Let $\lambda \in \Lambda(\xv)$ and let $\vec{\mu}_\lambda = FM_{\lambda}(F)$, then there exists a constant
  $\tilde{C}_3=\tilde{C}_3(m,d,\allowbreak p)$ such that for $R = \tilde{C}_3 (\Omega \tau h)^{2p-1}$,
  $$\FMG(H) \supseteq Q_{R}(\vec{\mu}_\lambda).$$
  Furthermore for $V_\lambda = \FMG(U)$ let 
  $$\FMG^{-1}:V_\lambda \rightarrow U$$ 
  be the local inverse of $\FMG$, i.e. for all $F' \in U$ we have
  $\FMG^{-1}(\FMG(F'))=F'$.  For each $1\le j \le d$, let
  $P_{\av,j},P_{\xv,j}: \barP \rightarrow \C$ be the projections onto the
  $j^{th}$ amplitude and the $j^{th}$ node coordinates respectively.
  Then $\FMG^{-1}$ is Lipschitz on $Q_R(\vec{\mu}_\lambda)$ with the following
  bounds:
  \begin{align*}
    	\left|P_{\xv,j} FM^{-1}_{\lambda}(\vec{\mu}') - P_{\xv,j} FM^{-1}_{\lambda}(\vec{\mu}'')\right| & \le \tilde{C}_{1}
    	\frac{1}{\Omega} \|\vec{\mu}''-\vec{\mu}'\| \times 
         \begin{cases}                                                                                    
         	1 & x_j \in \xv\setminus\xv^c \\
            (\Omega\tau h)^{-2p+2} & x_j \in \xv^{c} 
         \end{cases},\\
    \left|P_{\av,j} FM^{-1}_{\lambda}(\vec{\mu}') - P_{\av,j} FM^{-1}_{\lambda}(\vec{\mu}'')\right| & \le
    \tilde{C}_{2}\|\vec{\mu}''-\vec{\mu}'\| \times \begin{cases} 1 & x_j \in \xv\setminus\xv^c \\
      (\Omega \tau h)^{-2p+1} & x_j \in \xv^{c}
    \end{cases},
  \end{align*}
  for each $\vec{\mu}'',\vec{\mu}' \in Q_{R}(\vec{\mu}_\lambda)$, where
  $\tilde{C}_1 = \tilde{C}_{1}(m,d,p), \ \tilde{C}_2 =
  \tilde{C}_{2}(m,d,p)$ are constants depending only on
  $d,m,p$ and $\tilde{C}_1 \tilde{C}_3 \le 1$.
\end{proposition} 

\begin{proof}
 
  By Proposition \ref{prop.one.to.one} $FM_{\lambda}$ is injective in
  the open neighborhood $U$ of the polydisc
  $H=H_{\frac{m}{2},\frac{\tau h}{4 \pi }}(F)$. In addition, for each
  $F' \in H$ the inverse Jacobian norm bounds derived in Proposition
  \ref{prop.uniform.jacobian.bounds.clustered.signal} apply.  Finally
  one can verify (using a similar argument as in the proof of
  Proposition \ref{prop.uniform.jacobian.bounds.clustered.signal} )
  that $J_\lambda(F')$ is non-degenerate for each $F' \in U$. We can
  therefore invoke Theorem \ref{thm.inverse.function} with $U,H$ and
  $f=\FMG$ and the bounds \eqref{uniform.bound.amp.non.cluster},
  \eqref{uniform.bound.nodes.non.cluster},
  \eqref{uniform.bound.amp.cluster},
  \eqref{uniform.bound.nodes.cluster}, and conclude that Proposition
  \ref{prop.lipshitz} holds with
  $\tilde{C}_1 = \tilde{C}_2 = \tilde{C}$ and
  $\tilde{C}_3 = \min\left(\frac{m}{2\tilde{C}},\frac{1}{4 \pi
      \tilde{C}}\right)$.
\end{proof}

\subsubsection{The global geometry of admissible decimation maps}\label{sec.global}

In this subsection we give a global description of the geometry of the
error set $E_{\epsilon,(\lambda)}(F)$ for any $\lambda \in \Lambda(\xv)$
and for $\epsilon \le R$ where $R = \tilde{C}_3 (\Omega \tau h)^{2p-1}$, and $\tilde{C}_3$ is as
specified in Proposition \ref{prop.lipshitz}.

\smallskip

For each $\lambda\in \Lambda(\xv)$ let $\vec{\mu}_{\lambda}=\FMG(F)$, and put
\begin{equation}\label{eq.def.AG}
  A_{\epsilon,\lambda}(F) = \FMG^{-1}\left(Q_\epsilon(\vec{\mu}_{\lambda})\right) \bigcap \P_d,			
\end{equation}
where $\FMG^{-1}: V_\lambda \rightarrow U$ is the local inverse of $\FMG$ on $U$.

Observe that $A_{\epsilon,\lambda}(F) \subset \EG$. The analysis
of this subsection will reveal that globally $\EG$ is made from certain periodic repetitions 
of the set $A_{\epsilon,\lambda}(F)$ and its permutations.

Consider the following example.
\begin{example}
  Let $F(x)=\delta(x-\frac{1}{10})+\delta(x-\frac{2}{10})$ and let $\lambda = \frac{10}{3}$.
  Applying $\FMG$ on $F$ we get that
  $$
  \FMG(F)=(2,e^{\frac{2\pi}{3}i}+e^{-\frac{2\pi}{3}i},e^{-\frac{2\pi}{3}i}+
  e^{\frac{2\pi}{3}i},2) = (2,-1,-1,2).
  $$ 
  If we set $F=(\av,\xv)$ with $\av=(a_1,a_2)=(1,1)$ and
  $\xv=(x_1,x_2)=(\frac{1}{10},\frac{2}{10})$ then clearly the signal
  $F'=(\av,\xv')$, $\xv'=(x_2,x_1)=(\frac{2}{10},\frac{1}{10})$, that
  is attained by permuting the nodes of the signal $F$, satisfies that
  $\FMG(F)=\FMG(F')$. Observe that $F' \notin \P_2$ since its nodes
  are not in ascending order (a condition that was posed on $\P_d$ to
  avoid redundant solutions).  However, the signal $F''=(a,\xv'')$ with
  $\xv''=\xv'-\frac{1}{\lambda}(1,0)=\xv'-\frac{3}{10}(1,0)=(-\frac{1}{10},\frac{1}{10})$,
  is in $\P_2$ and it holds that $\FMG(F)=\FMG(F'')$.
  
  One can verify that the set of signals $G \in \P_2$, which satisfies $\FMG(G)=\FMG(F)$ 
  is given by
  \begin{align*}
    \left\{ G=(\av,\vec{y})\in \P_2 : \vec{y}=\xv+\frac{1}{\lambda}(n_1,n_2),\ \ n_1,n_2\in \Z\right\} & \bigcup \\
    \left\{ G=(\av,\vec{y})\in \P_2 : \vec{y}=\xv'+\frac{1}{\lambda}(n_1,n_2),\ \ n_1,n_2\in \Z\right\} &.
  \end{align*}
\end{example}

In order to formalize the statement regarding the global structure of $\EG$, which
is essentially a generalization of the example above, we require 
some notation regarding permutation and shift operations.

We denote the set of permutations of $d$ elements by 
$$\Pi = \Pi_d \subset \left\{ \pi : \{1,\ldots,d\} \rightarrow \{1,\ldots,d\} \right\}.$$

For a vector $\xv=(x_1,\ldots,x_d) \in \C^d$ and a permutation $\pi $, we denote by $\xv^{\pi}$ the vector attained by 
permuting the coordinates of $\xv$ according to $\pi$
$$\xv^{\pi}=(x_{\pi(1)},\ldots,x_{\pi(d)}).$$
For a set $A \subseteq \P_d$ and a permutation $\pi\in\Pi_d$, we denote by $A^{\pi}$ the set attained 
from $A$ by permuting the nodes and amplitudes of each signal in $A$ according to $\pi$
$$A^{\pi}=\left\{ (\av^{\pi},\xv^{\pi}) : (\av,\xv)\in A \right\}.$$

The following proposition gives a description of the global geometry
of $\EG$. Its proof is presented in Appendix \ref{appendix.global.geo}.
\begin{proposition}\label{prop.global.geo}
  For each $\lambda \in \Lambda(\xv)$ and $\epsilon \le R$
  $$\EG = \left(\bigcup_{\pi \in \Pi_d} \bigcup_{\lv \in \Z^d} A_{\epsilon,\lambda}^{\pi}(F)+\frac{1}{\lambda}\lv
  \right)\bigcap \P_d.$$
\end{proposition} 
	
\subsection{Proof of the upper bound}\label{sec:proof-upper-bound}

  Fix $F=(\av,\xv) \in \P_d,\ \av=(a_1,\ldots,a_d), \
  \xv=(x_1,\ldots,x_d)\subset \left[-\frac{1}{2},\frac{1}{2}\right]$,
  such that $\xv$ forms a $(p,h,1,\tau,\eta)$-clustered configuration
  with $\xv^c=\{x_\kappa, x_{\kappa+1},\allowbreak\ldots,\allowbreak x_{\kappa+p-1}\}$,
  and $\|\av\| \geq m > 0$.

    Consider the set of the admissible blowup factors $\Lambda(\xv)$ (see
  Definition \ref{def.good.blowup.set}).  By the analysis of Section
  \ref{sec.error.sets}, under the assumption that
  $\Omega h \le \frac{1}{20d}$, the following assertions hold:

  \begin{enumerate}
  \item By Proposition \ref{prop.one.to.one} there exists a neighborhood $U$ of $F$
    such that for each $\lambda \in \Lambda(\xv)$, $\FMG$ is one-to-one on $U$.
  \item By Proposition \ref{prop.lipshitz} there exists a constant
    $\tilde{C}_3=\tilde{C}_3(m,d,\allowbreak p)$ such for each
    $\lambda \in \Lambda(\xv)$, $V_{\lambda} = \FMG(U)$ contains a
    cube $Q_R(\vec{\mu}_{\lambda})$, where
    $\vec{\mu}_{\lambda} = \FMG(F)$ and $R = \tilde{C}_3 (\Omega \tau h)^{2p-1}$.
  \end{enumerate}			

  For each $\lambda \in \Lambda(\xv)$ consider the local inverse $\FMG^{-1}: V_{\lambda} \rightarrow U$
  and let (as above)
  $$ A_{R,\lambda}(F) = \FMG^{-1}(Q_{R}(\vec{\mu}_{\lambda})) \bigcap \P_d.$$
  
  The following intermediate claim is proved in Appendix
  \ref{appendix.upper.bound.tech}.

  \begin{proposition}\label{prop.elimiate.period}
    There exist positive constants $K_9$ and $K_{10}\le \frac{1}{20d}$ depending only on $d$, 
    such that for $\frac{K_9}{\eta} \le \Omega \le \frac{K_{10}}{h}$ the following holds.
    There exists $\lambda \in \Lambda(\xv)$ such that for each pair $(\pi,\lv) \in \Pi_d \times \left(\Z^d\setminus\{\vec{0}\}\right)$,
    there exists $\lambda_{\pi,\lv} \in \Lambda(\xv)$ for which 
    \begin{equation}\label{eq.what.we.intend.to.show}
      \left(A_{R,\lambda}^{\pi}(F)+\frac{1}{\lambda}\lv \right) \bigcap 
      E_{R,(\lambda_{\pi,\lv} )}(F) = \emptyset.
    \end{equation}
  \end{proposition} 

With a bit of additional work, we obtain the main geometric result
regarding the error set $\EO$.

\begin{proposition}\label{prop.good.lambda}
    Let $\Omega$ as in Proposition \ref{prop.elimiate.period}, then there exists $\lambda \in \Lambda(\xv)$ such that
  \begin{equation}\label{eq.goal}
    E_{R,\Omega}(F) \subset A_{R,\lambda}(F).
  \end{equation}
\end{proposition}

\begin{proof}
  Using Proposition \ref{prop.elimiate.period} fix
  $\lambda^* \in \Lambda(\xv)$ which satisfies
  \eqref{eq.what.we.intend.to.show}.  We will prove that $\lambda^{*}$
  satisfies \eqref{eq.goal}.

  For each $\lambda \in \Lambda(\xv)$, we have the following result due to Proposition \ref{prop.global.geo}:
  \begin{equation}\label{eq.upper.bound.repeat.res.global}
    \EGR \subset \bigcup_{\pi \in \Pi_d} \bigcup_{\lv \in \Z^d} \left(A_{R,\lambda}^{\pi}(F)+\frac{1}{\lambda}\lv\right).
  \end{equation}
  
  \smallskip
  
  Putting $\epsilon=R$ in \eqref{eq:errorset.via.decimated.errorsets}
  we obtain
  \begin{equation}\label{eq.uper.bound.intersection}
    \EOR \subseteq \bigcap_{\lambda \in (0, \frac{\Omega}{2d-1}]} 
    \EGR.
  \end{equation}
  
  We then obtain \eqref{eq.goal} from
  \eqref{eq.what.we.intend.to.show}, \eqref{eq.upper.bound.repeat.res.global} and \eqref{eq.uper.bound.intersection}
  by algebra of sets calculation as follows:
  
  \noindent First by \eqref{eq.uper.bound.intersection}
  \begin{equation}\label{eq.set.alg.1}
    \EOR \subseteq \bigcap_{\lambda \in (0, \frac{\Omega}{2d-1}]} 
    \EGR = \EGRS \cap \left( \bigcap_{\lambda \in (0, \frac{\Omega}{2d-1}]} \EGR\right).
  \end{equation} 		
  
  By \eqref{eq.upper.bound.repeat.res.global} 
  \begin{equation}\label{eq.set.alg.2}
    \EGRS \subset
    \bigcup_{\pi \in \Pi_d, \lv \in \Z^d} \left(A_{R,\gs}^{\pi}(F)+\frac{1}{\gs}\lv\right).
  \end{equation}
  
  Then by \eqref{eq.set.alg.1} and \eqref{eq.set.alg.2}
  \begin{equation}\label{eq.set.alg.3}
    \EOR \subseteq \left(\bigcup_{\pi \in \Pi_d, \lv \in \Z^d} A_{R,\gs}^{\pi}(F)+\frac{1}{\gs}\lv
    \right) \cap \left( \bigcap_{\lambda \in (0, \frac{\Omega}{2d-1}]} \EGR\right).
  \end{equation}	
  
  For each pair
  $(\pi,\lv) \in \Pi_d \times \left(\Z^d\setminus\{\vec{0}\}\right)$,
  let $\lambda_{\pi,\lv}\in\Lambda(\xv)$ be the value asserted by
  Proposition \ref{prop.elimiate.period}, i.e. satisfying
  \eqref{eq.what.we.intend.to.show} for $\lambda=\gs$. By this and by
  \eqref{eq.set.alg.3} we have
  \begin{align}\label{eq.upper.bound.local.error.set}
    \begin{split}
      \EOR &\subset \left(\bigcup_{\pi \in
          \Pi_d}A_{R,\gs}^{\pi}(F)\right)\bigcup
      \left(\bigcup_{(\pi,\lv) \in \Pi_d \times
          (\Z^d\setminus\{\vec{0}\})} \left\{ \left(
            A_{R,\gs}^{\pi}(F)+\frac{1}{\gs}\lv \right) \bigcap
          E_{R,(\lambda_{\pi,\lv})}(F)\right\}\right) \\ &
      =\bigcup_{\pi \in \Pi_d}A_{R,\gs}^{\pi}(F).
    \end{split}
  \end{align} 
  By definition $\EOR \subset \P_d$ where we assume a canonical ascending order of the nodes. Then, we 
  conclude from \eqref{eq.upper.bound.local.error.set} that $\EOR \subset A_{R,\gs}(F)$ 
  which proves \eqref{eq.goal} for $\lambda = \lambda^*$.
\end{proof}

We have everything in place to estimate the diameter of the set $\EO$ and its
projections.

\begin{proposition}\label{prop.upper.bound.normelized}
  Let
  $F=(\av,\xv) \in \P_d$, $\xv\subset \left[-\frac{1}{2},\frac{1}{2}\right]$,
  such that $\xv$ forms a $(p,h,1,\tau,\eta)$-clustered configuration and $\|\av\| \geq m > 0$. 
  Then there exist positive
  constants $C_1,\ldots,C_5$, depending only on
  $d,p,m,$ such that for each
  $\frac{C_4}{\eta} \le \Omega \le \frac{C_5}{h}$ and
  $\epsilon \le C_3(\Omega \tau h)^{2p-1}$, it holds that:
  \begin{align*}
    diam(E^{\xv,j}_{\epsilon,\Omega}(F)) &\le \begin{cases} C_1  {1\over\Omega}(\Omega \tau  h)^{-2p+2} \epsilon, & x_j
    \in \xv^c, \\
      C_1 {1\over\Omega}\epsilon, & x_j \in \xv\setminus\xv^c,
    \end{cases}
    \\
    diam(E^{\av,j}_{\epsilon,\Omega}(F)) & \le
                                     \begin{cases}
                                       C_2 (\Omega \tau h)^{-2p+1}\epsilon, & x_j \in \xv^c, \\
                                       C_2  \epsilon, & x_j \in \xv\setminus\xv^c.
                                     \end{cases}
  \end{align*}
\end{proposition}
\begin{proof}
  Let $\Omega$ be such that $\frac{K_9}{\eta} \le \Omega \le \frac{K_{10}}{h}$, 
  where $K_9=K_9(d),K_{10}=K_{10}(d)$ are the constants specified in Proposition \ref{prop.elimiate.period}.
  Let $\epsilon \le \tilde{C}_3(\Omega \tau h)^{2p-1} = R$, where $\tilde{C}_3 = \tilde{C}_3(m,d,\allowbreak p)$ is as
  specified in Proposition \ref{prop.lipshitz}. Let $F' \in \EO$ with $F'=(\av',\xv')$.  
  Using Proposition \ref{prop.good.lambda} fix
  $\lambda^* \in \Lambda(\xv)$ which satisfies \eqref{eq.goal}, and put $\vec{\mu}^{*}=FM_{\gs}(F)$. 
  Consequently
  $$F' \in A_{R,\gs}(F) = FM_{\gs}^{-1}(Q_{R}(\vec{\mu}^*)) \cap \P_d.$$
  
  Put $\vec{\mu}^{'}=FM_{\gs}(F')$. By
  Proposition \ref{prop.lipshitz} there exist constants
  $\tilde{C}_1 = \tilde{C}_{1}(m,d,p), \ \tilde{C}_2 =
  \tilde{C}_{2}(m,d,p)$ such that 
  \begin{align*}
      |\xv_j-\xv'_j|=\|P_{\xv,j} FM^{-1}_{\gs}(\vec{\mu}^*) - P_{\xv,j}
    FM^{-1}_{\gs}(\vec{\mu}')\| & \le
                                     \begin{cases}
                                     \tilde{C}_{1} \frac{1}{\Omega}
                                      (\Omega h)^{-2p+2}\epsilon,                                        & x_j \in \xv^c, \\
                                     \tilde{C}_1 \frac{1}{\Omega}\epsilon, & x_j \in\xv\setminus\xv^c.
                                     \end{cases}\\
    |\av_j-\av'_j|=\|P_{\av,j} FM^{-1}_{\gs}(\vec{\mu}^*) - P_{\av,j} FM^{-1}_{\gs}(\vec{\mu}')\| & \le
           \begin{cases}
           		\tilde{C}_{2}(\Omega h)^{-2p+1}\epsilon, & x_j \in \xv^c \\
                 \tilde{C}_{2} \epsilon, & x_j \in \xv\setminus\xv^c.                               
           \end{cases}
  \end{align*}

  Since $F'$ was an arbitrary signal in $\EO$, we repeat the above argument
  with $F''\in E_{\epsilon,\Omega}(F)$ and consequently prove
  Proposition \ref{prop.upper.bound.normelized} with
  $C_1 = 2\tilde{C}_1$, $C_2 = 2\tilde{C}_2$, $C_3 = \tilde{C}_3$, 
  $C_4=K_9$ and $C_5=K_{10}$.
\end{proof}

\bigskip

We are now in a position to prove Theorem
\ref{thm.accuracy.bounds.upper}, essentially by combining Proposition
\ref{prop.upper.bound.normelized} with Proposition
\ref{prop:normalization}.
\begin{proof}[Proof of Theorem \ref{thm.accuracy.bounds.upper}]
  Let $F=(\av,\xv) \in \P_d$ such that $\xv$ forms a
  $(p,h,T,\tau,\eta)$-clustered configuration
  and $\|\av\|\geq m > 0$.  Let
  $\frac{C_4}{ \eta T}\le \Omega \le \frac{C_5}{h}$ where
  $C_4=C_4(d,p,m),C_5=C_5(d,p,m)$ are the constants
  specified in Proposition \ref{prop.upper.bound.normelized}.

  \smallskip
  
  Put $\alpha = (x_1+x_d)/2$. The signal $SC_{T}(SH_{\alpha}(F))=(\av,\tilde{\xv})$,
  $\tilde{\xv}=(\tilde{x}_1,\allowbreak \ldots,\tilde{x}_d)$, 
  $\tilde{x}_1=\frac{x_1-\alpha}{T},\allowbreak \ldots,\tilde{x}_d=\frac{x_d-\alpha}{T}$,
  is normalized such that
  $\tilde{x}_1,\ldots,\tilde{x}_d \in [-\frac{1}{2},\frac{1}{2}]$.
  The node vector $\tilde{\xv}$ forms a $(p,\frac{h}{T},1,\tau,\eta)$-clustered configuration.
  Applying Proposition \ref{prop.upper.bound.normelized} for $\tilde{F}=SC_{T}(SH_{\alpha}(F))$, $\tilde{h}=\frac{h}{T}$,
  $\tilde{\Omega}=\Omega T \ge \frac{C_4}{\eta}$ and $\tilde{\Omega}
  \tilde{h} = \Omega h \le C_5$, we conclude that there exist
  constants $C_1,C_2,C_3$, depending only on $d,p,m$, such that for any $\epsilon \le C_3 (\Omega \tau h)^{2p-1}$
  \begin{align*}
      diam\big(E^{\xv,j}_{\epsilon,\Omega T}(SC_{T}(SH_{\alpha}(F)))\big) &\le
                                                             \begin{cases}
C_1
                                                       \frac{1}{\Omega T} (\Omega \tau h)^{-2p+2}  \epsilon, & x_j \in
                                                       \xv^c, \\
    C_1 {1\over{\Omega T}}\epsilon, & x_j \in \xv\setminus\xv^c,                                                   
                                                             \end{cases}\\
    diam\big(E^{\av,j}_{\epsilon,\Omega T}(SC_{T}(SH_{\alpha}(F)))\big) &\le
                                                             \begin{cases}
                                                               C_2 (\Omega \tau h)^{-2p+1}\epsilon, & x_j \in \xv^c,\\
                                                               C_2 \epsilon, & x_j\in\xv\setminus\xv^c.                                                              
                                                             \end{cases}
  \end{align*}
  
  Applying Proposition \ref{prop:normalization} we conclude the proof
  Theorem \ref{thm.accuracy.bounds.upper}.
\end{proof}

\section{Lower bounds}\label{sec:lower-bound}

In this section all the constants
$c_1,\dots,k_1,\dots,K_1,\dots$ are unrelated to those of the previous section.

The main technical result we need is the following.

\begin{proposition}\label{prop.perturbation.large}
Let $F=(\av,\xv)\in\P_d$, such that $\xv$ forms a
$(p,h,1,\tau,\eta)$-clustered configuration, with cluster nodes
$\xv^c=\left(x_1,\dots,x_p\right)$ (according to
Definition \ref{def.uniform.cluster}), and with $\av\in\R^d$
satisfying $m\leq \|\av\|\leq M$.

Then there exist constants $c_1,k_1,k_2$, depending only on
$\left(d,\tau,m,M\right)$, such that for all
$\epsilon<c_1 (\Omega h)^{2p-1}$ and $\Omega h \leq 2$, there exists a
signal $F_{\epsilon}\in\P_d$ satisfying, for some $j_{1},j_{2}\in\left\{1,\dots,p\right\}$,
\begin{align}
  \label{eq:feps.omega.px}
  \left| P_{\xv,j_1} \left(F_{\epsilon}\right) - P_{\xv,j_1} \left(F\right)\right| &\geq
  \frac{k_1}{\Omega} \left(\Omega h\right)^{-2p+2} \epsilon, \\
  \label{eq:feps.omega.ax}
\left| P_{\av,j_2} \left(F_{\epsilon}\right) - P_{\av,j_2} \left(F\right)\right| &\geq k_2\left(\Omega h\right)^{-2p+1} \epsilon, \\
  \label{eq:feps.fourier}
  \left|{\cal F}\left(F_{\epsilon}\right)(s) - {\cal F}(F)(s)\right| &\leq \epsilon,\qquad |s|\leq\Omega.
\end{align}

\end{proposition}

Assuming validity of Proposition ~\ref{prop.perturbation.large}, let us
 prove Theorem
~\ref{thm.accuracy.bounds.lower}.

\begin{proof}[Proof of Theorem~\ref{thm.accuracy.bounds.lower}]
  
  Let $\av\in\R^d$ be any real amplitude vector satisfying
  $m\leq\|\av\|\leq M$. Let $\Omega,h$ satisfy $\Omega h \leq 2$, and
  choose $\xv$ to be the configuration with cluster nodes
  $$\xv^c=\left(x_1=0,x_1=\tau h,\dots,x_p=(p-1)\tau h\right),$$
  with the rest of the nodes equally spaced in
  $\left((p-1)\tau h,1\right)$. Now denote $h'=(p-1)\tau h$ and
  $\tau'={1\over{p-1}}$. Clearly, $\xv$ is a
  $(p,h',1,\tau',\eta)$-clustered configuration for all sufficiently
  small $h$ (for instance, $h<{1\over d}< 1-\eta(d-p+1)$). Now we
  apply Proposition \ref{prop.perturbation.large} with the signal
  $F=(\av,\xv)$. Since $\tau'$ does not depend on $\tau$, and
  therefore the constants $c_1,k_1,k_2$ depend only on $d,p,m,M$, we
  conclude that for $\epsilon<c_1(p-1)^{2p-1}(\Omega\tau h)^{2p-1}$
  and $\Omega h < \frac{2}{(p-1)\tau}$, there exist
  $j_1,j_2\in\left\{1,\dots,p\right\}$ such that
    \begin{align*}
    	diam(E^{\xv,j_1}_{\epsilon,\Omega}(F)) & \ge {k_1\over\Omega} (p-1)^{-2p+2} \epsilon 
    		  (\Omega \tau h)^{-2p+2} ,\\
    	diam(E^{\av,j_2}_{\epsilon,\Omega}(F)) &\ge k_2 \epsilon (p-1)^{-2p+1}
    		(\Omega\tau h)^{-2p+1}.
    \end{align*}

    Now we consider the case of a non-cluster node,
    $x_j\in\xv\setminus\xv^c$. Let $F=(\av,\xv)$ be the signal
    above. Decompose $F$ as follows:
    $$
    F(x)=a_j\delta(x-x_j) + \underbrace{\sum_{\ell\neq
        j}a_{\ell}\delta(x-x_{\ell})}_{F^{o}}.
    $$
    Now let $\epsilon$ be fixed. Define $a_j'=a_j+{\epsilon\over 2}$
    and $x_j'=x_j + \frac{\epsilon}{4\pi\Omega M}$. Put
    $F'_j(x)=a_j'\delta(x-x_j')+F^{o}(x)$. For $|s|\leq\Omega$, the
    difference between the Fourier transforms of $F$ and $F_j'$
    satisfies
    \begin{align*}
      \left|{\cal F}(F)(s)-{\cal F}(F_j')(s)\right| &= \left|
                                                    a_je^{2\pi i x_j s} - a_j' e^{2\pi i x_j' s}\right| \\ &\leq
                                                                                                             \left| a_j e^{2\pi i x_j s}\left(1-e^{2\pi i {\epsilon\over
                                                                                                             {4\pi\Omega M}} s}\right)\right| + \left|a_j'-a_j\right| \\
                                                  &\leq {\epsilon\over 2} + {\epsilon\over 2} = \epsilon.
    \end{align*}
    Since the constants do not depend on $\tau$ at all, and the above
    construction of $F_j'$ can be repeated for each
    $j\notin\left\{\kappa,\dots,\kappa+p-1\right\}$, the
    proof of the non-cluster node case is finished.

    Again, the case of general $T$ follows by rescaling and applying
    Proposition \ref{prop:normalization} (as was done in the proof of
    Theorem \ref{thm.accuracy.bounds.upper}).

    This finishes the proof of Theorem \ref{thm.accuracy.bounds.lower}
    with $C_1'=\max\left(\frac{k_1}{(p-1)^{2p-2}},\frac{1}{4\pi
        M}\right)$, $C_2'=\max\left({1\over 2},\frac{k_2}{(p-1)^{2p-1}}\right)$,
    $C_3'=c_1(p-1)^{2p-1}$, $C_4'={1\over d}$ and $C_5'=2$.
\end{proof}

In the rest of this section we prove Proposition
~\ref{prop.perturbation.large}.

We start by stating the following result which has been shown in
\cite[Theorems 4.1 and 4.2]{akinshin2017error}.

\begin{theorem}\label{thm:moments-lb}
  Given the parameters $0 < h \le 2$, $0<\tau\leq 1$,
  $0<m\leq M<\infty$, let the signal $F=(\av,\xv)\in\P_d$ with
  $\av\in\R^d$ form a single uniform cluster as follows:
\begin{itemize}
\item (centered) $x_d=-x_1$;
\item (uniform) for $1\leq j<k\leq d$ we have
  $$
  \tau h \leq \left| x_j-x_k\right| \leq h;
  $$
\item $m\leq \|a_j\|\leq M$.
\end{itemize}
Then there exist constants $K_1,\dots,K_5$ depending only on
$\left(d,\tau,m,M\right)$ such that for every $\epsilon<K_5 h^{2d-1}$,
there exists a signal $F_{\epsilon}=(\vec{b},\vec{y})\in\P_d$ such that
\begin{enumerate}
\item $m_k\left(F\right)=m_k\left(F_{\epsilon}\right)$ for
  $k=0,1,\dots,2d-2$, where $m_k$ are given by
  \eqref{eq:prony-moments};
\item
  $m_{2d-1}\left(F_{\epsilon}\right)=m_{2d-1}\left(F\right)+\epsilon$;
\item $K_1 h^{-2d+2}\epsilon \leq \|\xv-\vec{y}\| \leq K_2 h^{-2d+2}\epsilon$;
\item $K_3 h^{-2d+1}\epsilon \leq \|\vec{b}-\av\| \leq K_4 h^{-2d+1}\epsilon$.
\end{enumerate}
\end{theorem}

\begin{proof}[Proof of Proposition~\ref{prop.perturbation.large}]
Define $F^c$ and $F^{nc}$ to be the
cluster and the non-cluster part of $F$ correspondingly, i.e.
\begin{align*}
F^c &= \sum_{x_j\in \xv^c} a_j \delta(x-x_j), \\
  F^{nc} &= \sum_{x_j \in \xv\setminus\xv^c} a_j \delta(x-x_j).
\end{align*}
Without loss of generality, suppose that $F^c$ is centered,
i.e. $x_1+x_{p} = 0$. Next, define a blowup of $F^c$ by $\Omega$ as follows:
\begin{equation}
  \label{eq:fc.omega.def}
  F^c_{\left(\Omega\right)}=SC_{1\over\Omega}\left(F^c\right) = \sum_{x_j\in \xv^c} a_j \delta(x-\Omega x_j).
\end{equation}
Put $\tilde{d} = p, \tilde{h}=\Omega h$, and let
$c_1=K_5\left(\tilde{d},\tau,m,M\right)$ as in Theorem
\ref{thm:moments-lb}.  Let
$\epsilon \leq c_1 \left(\Omega h\right)^{2p-1}$. Now we apply Theorem
\ref{thm:moments-lb} with parameters
$\tilde{d},\tilde{h}, \tau, m, M, \tilde{\epsilon}=c_2\epsilon$ and
the signal $F^c_{\left(\Omega\right)}$, where $c_2\leq 1$ will be
determined below. We obtain a signal
$G^c_{\left(\Omega\right),\epsilon}$ such that the following hold for
the difference
$H=G^c_{\left(\Omega\right),\epsilon}-F^c_{\left(\Omega\right)}$:

\begin{align}
  \label{eq:gc.omega.mk} m_k\left(H\right) &= 0,\quad k=0,1,\dots,2p-2, \\
  \label{eq:gc.omega.m2p} m_{2p-1}\left(H\right) &= \tilde{\epsilon};
\end{align}
while also, for some $j_1,j_2\in\left\{1,\dots,p\right\}$
\begin{align}
  \label{eq:gc.omega.px.lb}
  \left| P_{\xv,j_1} \left(G^c_{\left(\Omega\right),\epsilon}\right) - P_{\xv,j_1} \left(F^c_{\left(\Omega\right)}\right)\right|
  &\geq K_1 \left(\Omega h\right)^{-2p+2} \tilde{\epsilon},\\
  \label{eq:gc.omega.px.ub}
  \left| P_{\xv,j} \left(G^c_{\left(\Omega\right),\epsilon}\right) - P_{\xv,j} \left(F^c_{\left(\Omega\right)}\right)\right|
  &\leq K_2 \left(\Omega h\right)^{-2p+2} \tilde{\epsilon},\quad j=1,\dots,p, \\
  \label{eq:gc.omega.ax}
\left| P_{\av,j_2} \left(G^c_{\left(\Omega\right),\epsilon}\right) - P_{\av,j_2} \left(F^c_{\left(\Omega\right)}\right)\right|
&\geq K_3 \left(\Omega h\right)^{-2p+1} \tilde{\epsilon}.
\end{align}

Now put
$$
F^c_{\left(\Omega\right),\epsilon}=SC_{\Omega} \left( G^c_{\left(\Omega\right),\epsilon}\right).
$$

Applying the inverse blowup to the above inequalities, we obtain in
fact that
\begin{align}
  \label{eq:fc.omega.px}
  \left| P_{\xv,j_1} \left(F^c_{\left(\Omega\right),\epsilon}\right) - P_{\xv,j_1} \left(F^c\right)\right| &\geq
  \frac{K_1}{\Omega} \left(\Omega h\right)^{-2p+2} \tilde{\epsilon},\\
  \label{eq:fc.omega.ax}
\left| P_{\av,j_2} \left(F^c_{\left(\Omega\right),\epsilon}\right) - P_{\av,j_2} \left(F^c\right)\right| &\geq K_3 \left(\Omega
h\right)^{-2p+1} \tilde{\epsilon}.
\end{align}

From the above definitions we have
$H_{\Omega}=SC_{\Omega}(H) =
F^c_{\left(\Omega\right),\epsilon}-F^c$. Let us now show that there is
a choice of $c_3$ such that
\begin{equation}
  \label{eq:lower.bound.fourier.tr}
  \left|{\cal{F}}\left(H_{\Omega}\right)\left(s\right)\right| \leq \epsilon,\quad |s|\leq \Omega.
\end{equation}

Put $\omega=s/\Omega$, then
$$
{\cal{F}}\left(H_{\Omega}\right)(s) = {\cal{F}}\left(H\right)\left(\omega\right).
$$

Now we employ the fact that the Fourier transform of a spike train has
Taylor series coefficients precisely equal to its algebraic moments
(see \cite[Proposition 3.1]{akinshin2015accuracy}):
\begin{equation}\label{eq:fourier.taylor}
{\cal{F}}(H)(\omega) = \sum_{k=0}^{\infty} {1\over k!} m_k\left(H\right) \left(-2\pi\imath \omega\right)^k.
\end{equation}

Next we apply the following easy corollary of the Tur\'{a}n's First
Theorem \cite[Theorem 6.1]{turan1984}, appearing in \cite[Theorem
3.1]{mordukhovich_taylor_2016}, using the recurrence relation
satisfied by the moments of $H$ according to Proposition
\ref{prop.prony.rec.rel}.

\begin{theorem}\label{thm:turan.taylor.dom}
  Let $H=\sum_{j=1}^{2p} \beta_j \delta(x-t_j)$, and put
  $R=\min_{j=1,\dots,2p} \left|t_j\right|^{-1} > 0$. Then, for all
  $k\geq 2p$ we have the so-called ``Taylor domination'' property
  \begin{equation}
    \label{eq:taylor.domination}
    \left|m_k(H)\right| R^k \leq \left(\frac{2ek}{2p}\right)^{2p} \max_{\ell=0,1,\dots,2p-1} \left|m_{\ell}\left(H\right)\right| R^{\ell}.
  \end{equation}
\end{theorem}

\begin{proposition}
  \label{prop.radius.conv.ub}
  The constant $R$ in Theorem \ref{thm:turan.taylor.dom} satisfies
  $R\geq C_4$, where $C_4$ does not depend on $\Omega,h$.
\end{proposition}
\begin{proof}
  Recall that
  $H=G^c_{\left(\Omega\right),\epsilon}-F^c_{\left(\Omega\right)}$. The
  nodes of $F^c_{\left(\Omega\right)}$ are, by construction, inside
  the interval $\left[-{\Omega h\over 2},{\Omega h\over
      2}\right]$. The nodes of $G^c_{\left(\Omega\right),\epsilon}$,
  by \eqref{eq:gc.omega.px.ub}, satisfy
  \begin{align*}
      \left| P_{\xv,j} \left(G^c_{\left(\Omega\right),\epsilon}\right)\right| &\leq {\Omega h \over 2} + K_2 \left(\Omega
      h\right)^{-2p+2} \tilde{\epsilon} \\
                                                                            & \leq {\Omega h \over 2} + K_2 \left(\Omega
                                                                            h\right)^{-2p+2} c_1 \left(\Omega
                                                                            h\right)^{2p-1} \\
                                                                            & = \left(\Omega h\right) \left(c_1 K_2 + \frac{1}{2}\right).
  \end{align*}
  Since $\Omega h \leq 2$ by assumption, this concludes the proof with $C_4=\frac{1}{2\left(c_1 K_2 + \frac{1}{2}\right)}$.
\end{proof}

Therefore, by \eqref{eq:taylor.domination}, \eqref{eq:gc.omega.mk} and
\eqref{eq:gc.omega.m2p} we have for $k\geq 2p$
\begin{align*}
  \left|m_k\left(H\right)\right| &\leq \left(\frac{e}{p}\right)^{2p} k^{2p} R^{2p-1-k}\tilde{\epsilon} \\
  & \leq C_5 C_4^{2p-1-k} k^{2p} \tilde{\epsilon}.
\end{align*}

Now plugging this into \eqref{eq:fourier.taylor} we obtain
\begin{align*}
  \left| {\cal{F}}\left(H\right)\left(\omega\right)\right| & \leq \frac{\tilde{\epsilon} \left|2\pi\omega\right|^{2p-1}}{(2p-1)!} +  C_5 C_4^{2p-1} \tilde{\epsilon} \sum_{k\geq 2p} \left(\frac{2\pi \left| \omega \right|}{C_4}\right)^k \frac{k^{2p}}{k!}.
\end{align*}
Put $\zeta={2\pi|\omega| \over C_4}$, then, since $|\omega|\leq 1$,
\begin{align*}
  \left| {\cal{F}}\left(H\right)\left(\omega\right)\right| & \leq C_6 \tilde{\epsilon}  \sum_{k\geq 2p-1}\zeta^k \frac{k^{2p}}{k!} \\
  &\leq C_7 \tilde{\epsilon}.
\end{align*}
We can therefore choose $c_2=\min\left(1,\frac{1}{C_7}\right)$ to ensure that
$$
\left| {\cal{F}}\left(H\right)\left(\omega\right)\right| \leq \epsilon, \quad |\omega|\leq 1,
$$
which shows \eqref{eq:lower.bound.fourier.tr}.

Finally, construct the signal
$F_{\epsilon}=F^{nc}+F^c_{\left(\Omega\right),\epsilon}$. Combining
\eqref{eq:lower.bound.fourier.tr}, together with
\eqref{eq:fc.omega.px} and \eqref{eq:fc.omega.ax} finishes the proof
of Proposition \ref{prop.perturbation.large} with $k_1=K_1$ and $k_2=K_3$.
\end{proof}

\bibliographystyle{plain}
\bibliography{bib}{}

\begin{thebibliography}{10}

\bibitem{akinshin2015accuracy}
Andrey Akinshin, Dmitry Batenkov, and Yosef Yomdin.
\newblock Accuracy of spike-train {F}ourier reconstruction for colliding nodes.
\newblock In {\em 2015 International Conference on Sampling Theory and
  Applications (SampTA)}, pages 617--621. IEEE, 2015.

\bibitem{akinshin2017error}
Andrey Akinshin, Gil Goldman, and Yosef Yomdin.
\newblock Geometry of error amplification in solving {{Prony}} system with
  near-colliding nodes.
\newblock {\em arXiv:1701.04058 [math]}, January 2017.

\bibitem{aubel_vandermonde_2017}
C{\'e}line Aubel and Helmut B{\"o}lcskei.
\newblock Vandermonde matrices with nodes in the unit disk and the large sieve.
\newblock {\em Applied and Computational Harmonic Analysis}, August 2017.

\bibitem{auton1981investigation}
Jon~R Auton and Michael~L Van~Blaricum.
\newblock Investigation of procedures for automatic resonance extraction from
  noisy transient electromagnetics data.
\newblock {\em Math. Notes}, 1:79, 1981.

\bibitem{azais_spike_2015}
Jean-Marc Aza{\"\i}s, Yohann {de Castro}, and Fabrice Gamboa.
\newblock Spike detection from inaccurate samplings.
\newblock {\em Applied and Computational Harmonic Analysis}, 38(2):177--195,
  March 2015.

\bibitem{batenkov2017accurate}
Dmitry Batenkov.
\newblock Accurate solution of near-colliding {{Prony}} systems via decimation
  and homotopy continuation.
\newblock {\em Theoretical Computer Science}, 681:27--40, June 2017.

\bibitem{batenkov2016stability}
Dmitry Batenkov.
\newblock Stability and super-resolution of generalized spike recovery.
\newblock {\em Applied and Computational Harmonic Analysis}, 45(2):299--323,
  September 2018.

\bibitem{batenkov2018}
Dmitry Batenkov, Laurent Demanet, Gil Goldman, and Yosef Yomdin.
\newblock Conditioning of partial nonuniform {{Fourier}} matrices with
  clustered nodes.
\newblock {\em To appear in SIAM J.Matrix Anal.Appl., arXiv:1809.00658 [cs,
  math]}, 2019.

\bibitem{batenkov2019}
Dmitry Batenkov, Laurent Demanet, and Hrushikesh~N Mhaskar.
\newblock Stable soft extrapolation of entire functions.
\newblock {\em Inverse Problems}, 35(1):015011, January 2019.

\bibitem{batenkov2019spectral}
Dmitry Batenkov, Benedikt Diederichs, Gil Goldman, and Yosef Yomdin.
\newblock The spectral properties of {{Vandermonde}} matrices with clustered
  nodes.
\newblock {\em arXiv:1909.01927 [cs, math]}, September 2019.

\bibitem{paramertrization2017}
Dmitry Batenkov, Gil Goldman, Yehonatan Salman, and Yosef Yomdin.
\newblock Algebraic geometry of error amplification: the {P}rony leaves.
\newblock {\em arXiv preprint arXiv:1702.05338}, 2017.

\bibitem{batenkov2013accuracy}
Dmitry Batenkov and Yosef Yomdin.
\newblock On the accuracy of solving confluent {P}rony systems.
\newblock {\em SIAM Journal on Applied Mathematics}, 73(1):134--154, 2013.

\bibitem{batenkov2013geometry}
Dmitry Batenkov and Yosef Yomdin.
\newblock Geometry and singularities of the {P}rony mapping.
\newblock In {\em Proceedings of 12th International Workshop on Real and
  Complex Singularities}, volume~10, pages 1--25, 2014.

\bibitem{mordukhovich_taylor_2016}
Dmitry Batenkov and Yosef Yomdin.
\newblock Taylor domination, {{Tur\'an}} lemma, and {{Poincar\'e}}-{{Perron}}
  sequences.
\newblock In Boris Mordukhovich, Simeon Reich, and Alexander Zaslavski,
  editors, {\em Contemporary {{Mathematics}}}, volume 659, pages 1--15.
  {American Mathematical Society}, Providence, Rhode Island, 2016.

\bibitem{bazan_conditioning_2000}
F.S.V. Baz{\'a}n.
\newblock Conditioning of rectangular {{Vandermonde}} matrices with nodes in
  the unit disk.
\newblock {\em SIAM Journal on Matrix Analysis and Applications}, 21:679, 2000.

\bibitem{beckermann2007}
B.~Beckermann, G.~H. Golub, and G.~Labahn.
\newblock On the numerical condition of a generalized {{Hankel}} eigenvalue
  problem.
\newblock {\em Numerische Mathematik}, 106(1):41--68, March 2007.

\bibitem{benedetto_super-resolution_2016}
John~J. Benedetto and Weilin Li.
\newblock Super-resolution by means of {{Beurling}} minimal extrapolation.
\newblock {\em Applied and Computational Harmonic Analysis}, May 2018.

\bibitem{bertero1998}
M.~Bertero and P.~Boccacci.
\newblock {\em Introduction to Inverse Problems in Imaging}.
\newblock {Taylor \& Francis}, 1998.

\bibitem{bhaskar_atomic_2013}
B.~N. Bhaskar, G.~Tang, and B.~Recht.
\newblock Atomic {{Norm Denoising With Applications}} to {{Line Spectral
  Estimation}}.
\newblock {\em IEEE Transactions on Signal Processing}, 61(23):5987--5999,
  December 2013.

\bibitem{candes2013super}
Emmanuel~J Cand{\`e}s and Carlos Fernandez-Granda.
\newblock Super-resolution from noisy data.
\newblock {\em Journal of {F}ourier Analysis and Applications},
  19(6):1229--1254, 2013.

\bibitem{candes2014towards}
Emmanuel~J Cand{\`e}s and Carlos Fernandez-Granda.
\newblock Towards a mathematical theory of super-resolution.
\newblock {\em Communications on Pure and Applied Mathematics}, 67(6):906--956,
  2014.

\bibitem{cuyt2018}
Annie Cuyt and Wen-shin Lee.
\newblock How to get high resolution results from sparse and coarsely sampled
  data.
\newblock {\em Applied and Computational Harmonic Analysis}, 2018.

\bibitem{cuyt2018a}
Annie Cuyt, Min-nan Tsai, Marleen Verhoye, and Wen-shin Lee.
\newblock Faint and clustered components in exponential analysis.
\newblock {\em Applied Mathematics and Computation}, 327:93--103, June 2018.

\bibitem{demanet2015recoverability}
Laurent Demanet and Nam Nguyen.
\newblock The recoverability limit for superresolution via sparsity.
\newblock {\em arXiv preprint arXiv:1502.01385}, 2015.

\bibitem{denoyelle_support_2015}
Quentin Denoyelle, Vincent Duval, and Gabriel Peyr\'e.
\newblock Support {{Recovery}} for {{Sparse Deconvolution}} of {{Positive
  Measures}}.
\newblock {\em arXiv:1506.08264 [cs, math]}, June 2015.

\bibitem{donoho1992superresolution}
David~L Donoho.
\newblock Superresolution via sparsity constraints.
\newblock {\em SIAM journal on mathematical analysis}, 23(5):1309--1331, 1992.

\bibitem{fernandez2013support}
Carlos Fernandez-Granda.
\newblock Support detection in super-resolution.
\newblock In {\em Proceedings of the 10th International Conference on Sampling
  Theory and Applications (SampTA 2013)}, pages 145--148, 2013.

\bibitem{fernandez2016super}
Carlos Fernandez-Granda.
\newblock Super-resolution of point sources via convex programming.
\newblock {\em Information and Inference}, page iaw005, 2016.

\bibitem{ferreira_superresolution_1999}
Paulo Ferreira.
\newblock {\em Superresolution, the {{Recovery}} of {{Missing Samples}}, and
  {{Vandermonde Matrices}} on the {{Unit Circle}}}.
\newblock 1999.

\bibitem{gautschi1962inverses}
Walter Gautschi.
\newblock On inverses of {V}andermonde and confluent {V}andermonde matrices.
\newblock {\em Numerische Mathematik}, 4(1):117--123, 1962.

\bibitem{gautschi1963inverses}
Walter Gautschi.
\newblock On inverses of {V}andermonde and confluent {V}andermonde matrices.
  ii.
\newblock {\em Numerische Mathematik}, 5(1):425--430, 1963.

\bibitem{golub1973}
G.~H. Golub and V.~Pereyra.
\newblock The {{Differentiation}} of {{Pseudo}}-{{Inverses}} and {{Nonlinear
  Least Squares Problems Whose Variables Separate}}.
\newblock {\em SIAM Journal on Numerical Analysis}, 10(2):413--432, April 1973.

\bibitem{golub1999}
Gene~H. Golub, Peyman Milanfar, and James Varah.
\newblock A stable numerical method for inverting shape from moments.
\newblock {\em SIAM Journal on Scientific Computing}, 21(4):1222--1243, 1999.

\bibitem{goodman_introduction_2005}
Joseph~W. Goodman.
\newblock {\em Introduction to {{Fourier Optics}}}.
\newblock {Roberts and Company Publishers}, 2005.

\bibitem{heckel2016super}
Reinhard Heckel, Veniamin~I Morgenshtern, and Mahdi Soltanolkotabi.
\newblock Super-resolution radar.
\newblock {\em Information and Inference: A Journal of the IMA}, 5(1):22--75,
  2016.

\bibitem{hua_matrix_1990}
Y.~Hua and T.~K. Sarkar.
\newblock Matrix pencil method for estimating parameters of exponentially
  damped/undamped sinusoids in noise.
\newblock {\em IEEE Transactions on Acoustics, Speech, and Signal Processing},
  38(5):814--824, May 1990.

\bibitem{hua_svd_1991}
Y.~Hua and T.~K. Sarkar.
\newblock On {{SVD}} for estimating generalized eigenvalues of singular matrix
  pencil in noise.
\newblock {\em IEEE Transactions on Signal Processing}, 39(4):892--900, April
  1991.

\bibitem{kunis2018}
Stefan Kunis and Dominik Nagel.
\newblock On the condition number of {{Vandermonde}} matrices with pairs of
  nearly-colliding nodes.
\newblock {\em arXiv:1812.08645 [math]}, December 2018.

\bibitem{li_stable_2017}
Weilin Li and Wenjing Liao.
\newblock Stable super-resolution limit and smallest singular value of
  restricted {{Fourier}} matrices.
\newblock {\em arXiv:1709.03146v2 [cs, math]}, October 2018.

\bibitem{li2019}
Weilin Li, Wenjing Liao, and Albert Fannjiang.
\newblock Super-resolution limit of the {{ESPRIT}} algorithm.
\newblock {\em arXiv:1905.03782v3 [cs, math]}, October 2019.

\bibitem{lindberg_mathematical_2012}
Jari Lindberg.
\newblock Mathematical concepts of optical superresolution.
\newblock {\em Journal of Optics}, 14(8):083001, 2012.

\bibitem{micchelli_survey_1977}
C.~A. Micchelli and T.~J. Rivlin.
\newblock A {{Survey}} of {{Optimal Recovery}}.
\newblock In {\em Optimal {{Estimation}} in {{Approximation Theory}}}, The IBM
  Research Symposia Series, pages 1--54. {Springer, Boston, MA}, 1977.

\bibitem{micchelli_lectures_1985}
C.~A. Micchelli and T.~J. Rivlin.
\newblock Lectures on optimal recovery.
\newblock In {\em Numerical {{Analysis Lancaster}} 1984}, pages 21--93.
  {Springer}, 1985.

\bibitem{micchelli_optimal_1976}
C.~A. Micchelli, T.~J. Rivlin, and S.~Winograd.
\newblock The optimal recovery of smooth functions.
\newblock {\em Numerische Mathematik}, 26(2):191--200, June 1976.

\bibitem{moitra_super-resolution_2015}
Ankur Moitra.
\newblock Super-resolution, {{Extremal Functions}} and the {{Condition Number}}
  of {{Vandermonde Matrices}}.
\newblock In {\em Proceedings of the {{Forty}}-{{Seventh Annual ACM}} on
  {{Symposium}} on {{Theory}} of {{Computing}}}, STOC '15, pages 821--830, New
  York, NY, USA, 2015. {ACM}.

\bibitem{morgenshtern2016super}
Veniamin~I Morgenshtern and Emmanuel~J Candes.
\newblock Super-resolution of positive sources: the discrete setup.
\newblock {\em SIAM Journal on Imaging Sciences}, 9(1):412--444, 2016.

\bibitem{oleary2013}
Dianne~P. O'Leary and Bert~W. Rust.
\newblock Variable projection for nonlinear least squares problems.
\newblock {\em Computational Optimization and Applications}, 54(3):579--593,
  2013.

\bibitem{pereyra_exponential_2010}
Victor Pereyra and Godela Scherer.
\newblock {\em Exponential {{Data Fitting}} and {{Its Applications}}}.
\newblock {Bentham Science Publishers}, January 2010.

\bibitem{peter2013generalized}
Thomas Peter and Gerlind Plonka.
\newblock A generalized {P}rony method for reconstruction of sparse sums of
  eigenfunctions of linear operators.
\newblock {\em Inverse Problems}, 29(2):025001, 2013.

\bibitem{plonka2014prony}
Gerlind Plonka and Manfred Tasche.
\newblock {P}rony methods for recovery of structured functions.
\newblock {\em GAMM-Mitteilungen}, 37(2):239--258, 2014.

\bibitem{prony1795essai}
R.~Prony.
\newblock Essai experimental et analytique.
\newblock {\em J. Ec. Polytech.(Paris)}, 2:24--76, 1795.

\bibitem{range2013holomorphic}
R~Michael Range.
\newblock {\em Holomorphic functions and integral representations in several
  complex variables}, volume 108.
\newblock Springer Science \& Business Media, 2013.

\bibitem{schiebinger_superresolution_2015}
Geoffrey Schiebinger, Elina Robeva, and Benjamin Recht.
\newblock Superresolution without separation.
\newblock {\em Information and Inference: A Journal of the IMA}, 7(1):1--30,
  March 2018.

\bibitem{stoica_spectral_2005}
P.~Stoica and R.L. Moses.
\newblock {\em Spectral Analysis of Signals}.
\newblock {Pearson/Prentice Hall}, 2005.

\bibitem{tang_near_2015}
G.~Tang, B.~N. Bhaskar, and B.~Recht.
\newblock Near {{Minimax Line Spectral Estimation}}.
\newblock {\em IEEE Transactions on Information Theory}, 61(1):499--512,
  January 2015.

\bibitem{turan1984}
P.~Tur\'an, G.~Hal\'asz, and J.~Pintz.
\newblock {\em On a New Method of Analysis and Its Applications}.
\newblock {Wiley-Interscience}, 1984.

\bibitem{vetterli2002sampling}
Martin Vetterli, Pina Marziliano, and Thierry Blu.
\newblock Sampling signals with finite rate of innovation.
\newblock {\em IEEE transactions on Signal Processing}, 50(6):1417--1428, 2002.

\end{thebibliography}

\newpage
\appendix

\section{Algebraic Prony system}
\label{appendix.prony}

The so-called Prony system of equations relates the parameters
of the signal $F$ as in \eqref{eq.spike.train.signal} and its
\emph{algebraic moments}

\begin{equation}
  \label{eq:prony-moments}
  m_k\left(F\right)=\int F(x) x^k dx = \sum_{j=1}^d a_j x_j^k,\quad k=0,1,\dots,.
\end{equation}

Extending the above to arbitrary complex nodes and amplitudes,
we define the {\it Prony map} $PM:\C^{2d} \rightarrow \C^{2d}$
as follows:
\begin{equation}\label{eq.def.PM}
  PM_k(a_1,\ldots,a_d,w_1,\ldots,w_d)= \sum_{j=1}^{d}a_jw_j^k, \tab k=0,1,\dots,2d-1.
\end{equation}

Now consider the system of equations defined by $PM$,
i.e. with unknowns $\left\{a_j,z_j\right\}_{j=1}^d\in\C^{2d}$ and a
given right hand side
$\mu=(\mu_0,\ldots,\mu_{2d-1}) \in \C^{2d}$,
\begin{align}\label{eq.prony}
  PM_k\left(a_1,\dots,a_d,z_1,\dots,z_d\right)=\mu_k, & & k=0,1,\dots,2d-1.
\end{align} 
The following fact can be found in the literature about Prony
systems and Pad\'e approximation (see
e.g. \cite{batenkov2013geometry} Propositions 3.2 and 3.3).
\begin{proposition}\label{prop.unique.solution.prony}
  If a solution $(a_1,\ldots,a_d,z_1,\ldots,z_d)$ to System
  \eqref{eq.prony} exists with $a_j \ne 0,\ j=1,\ldots,d$ and
  for $1 \le j < k \le d$, $z_j\ne z_k$, it is unique up to a
  permutation of the nodes $\{z_j\}$ and corresponding
  amplitudes $\{a_j\}$.
\end{proposition}

Clearly, the definition of $PM_k$ is valid for arbitrary
integer $k\in\NN$. The next fact is very well-known, and it is
the basis of Prony's method of solving \eqref{eq.prony}.

\begin{proposition}
  \label{prop.prony.rec.rel} 
  Let the sequence $\nu=\left\{\nu_k\right\}_{k\in\NN}$ be
  given by
  $$\nu_k=PM_k\left(a_1,\dots,a_d,z_1,
    \dots,z_d\right).$$
  Then each consecutive $d+1$ elements of $\nu$ satisfy the
  following linear recurrence relation:
  \begin{equation}
    \label{eq:moments.rec.rel}
    \sum_{\ell=0}^{d} \nu_{k+\ell} c_{\ell} = 0,
  \end{equation}
  where the constants $\left\{c_{\ell}\right\}_{\ell=0}^d$ are the
  coefficients of the (monic) polynomial with roots
  $\left\{z_1,\dots,z_d\right\}$ (the ``Prony polynomial''),
  i.e.
  \begin{equation}
    \label{eq:rec.rel.cj.def}
    Q(z)=\prod_{j=1}^d \left(z-z_j\right) \equiv \sum_{\ell=0}^d c_{\ell} z^{\ell}.
  \end{equation}
\end{proposition}
\begin{proof}
  Let $k\in\NN$, then
  \begin{align*}
    \sum_{\ell=0}^d \nu_{k+\ell} c_{\ell} &= \sum_{
                                            \ell=0}^d c_{\ell} \sum_{j=1}^d a_j z_j^{k+\ell} \\
                                          &= \sum_{j=1}^d a_j z_j^k Q(z_j) = 0.\qedhere
  \end{align*}
\end{proof}

\begin{proposition}[Prony's method]\label{prop.prony.method}
  Let there be given the algebraic moments $\{m_k(F)\}_{k=0}^{2d-1}$
  of the signal $F=(\av,\xv)$ where the nodes of $\xv$ are pairwise
  distinct and $\|a\|>0$.  Then the parameters $(\av,\xv)$ can be
  recovered exactly by the following procedure:
  \begin{enumerate}
  \item Construct the  $d\times (d+1)$ Hankel matrix
    $H=\left[ m_{i+j}\right]_{0\leq i \leq d-1}^{0\leq j \leq d}$;
  \item Find a nonzero vector $\vec{c}$ in the null-space of $H$;
  \item Find $\xv_j$ to be the roots of the Prony polynomial
    \eqref{eq:rec.rel.cj.def}, whose coefficient vector is $\vec{c}$;
  \item Find the amplitudes $\av$ by solving the linear system
    $V\av=\vec{m}$, where $V$ is the Vandermonde matrix $V=\left[\xv_j^k\right]^{j=1,\dots,d}_{k=0,\dots,d-1}$.
  \end{enumerate}
\end{proposition}
\begin{proof}
  See e.g. \cite{batenkov2013geometry}.
\end{proof}

\section{Quantitative Inverse Function Theorem}\label{sec.inverse.fucntion}
Here we prove a certain quantitative version of the
inverse function theorem, which applies to holomorphic mappings
$\C^{d} \rightarrow \C^{d}$ (here $d$ is a generic parameter).

\smallskip	

For $\av \in \C^d$ and $r_1,\ldots,r_d>0$, let $H_{r_1,\ldots,r_d}(\av) \subset \C^d$
be the closed polydisc centered at $\av$, 
$$ 
H_{r_1,\ldots,r_d}(\av) = \{ \xv \in \C^d : | \xv_j - \av_j |
\le r_j, \mbox{ for all } j = 1,\dots,d \}.
$$

\smallskip

For $j=1,\ldots,d,$ we denote by $P_j:\C^d \rightarrow \C$ the
orthogonal projection onto the $j^{th}$ coordinate.  With some abuse
of notation we will also treat $P_j$ as the $d\times d$ matrix
representing this projection.

\smallskip

Finally recall Definition \ref{def.r.cube} of the hypercube $Q_r$.

\begin{theorem}\label{thm.inverse.function}
  Let $U \subseteq {\mathbb C}^d$ be open. Let $f: U \rightarrow {\mathbb C}^d$ be a holomorphic injection 
  with an invertible Jacobian $J(\xv)$, for all $\xv \in U$. For $\av \in U$ and $r_1,\ldots,r_d>0$, let $H(\av)=H_{r_1,\ldots,r_d}(\av) \subset U$
  be such that for all $\xv \in H(\av)$,
  $$\sum_{k=1}^d| J^{-1}_{j,k}(\xv)| \le \alpha_j, \tab j=1,\ldots,d.$$ 
  Put $\vec{b}=f(\av)$ and $f(U)=V$. Then:
  \begin{enumerate} 
  \item For $R = \min(\frac{r_1}{\alpha_1},\ldots,\frac{r_d}{\alpha_d})$, $Q_R(\vec{b}) \subseteq f(H(\av))$ 
    and $f^{-1}: V \rightarrow U$ is holomorphic in an open neighborhood of $Q_R(\vec{b})$.
  \item For each $j=1,\ldots,d$, $f_j^{-1}=P_jf^{-1}:Q_{R}(\vec{b})\rightarrow \C^d$ is Lipschitz on $Q_{R}(\vec{b})$ with
    $$
    |f^{-1}_j(\vec{y}'')-f^{-1}_j(\vec{y}')| \le \alpha_j \|\vec{y}''-\vec{y}'\|,
    $$
    for each $\vec{y}',\vec{y}'' \in Q_R(\vec{b})$.
  \end{enumerate} 	
\end{theorem}
\begin{proof}
  First we show that $f(U)=V$ is open and $f^{-1}$ is holomorphic and provides a homeomorphism 
  between $U$ and $V$.
  
  By assumption $f: U \rightarrow V$ is an injection, then $f^{-1}: V \rightarrow U$ is well defined. 
  By assumption $f$ is continuously
  differentiable with non-degenerate Jacobians $J(x)$ for all $x\in U$. Then by the Inverse Function Theorem 
  $V$ is open and $f^{-1}$ is continuously differentiable on $V$.
  We conclude that $f$ is a biholomorphism between $U$ and $V$. 
  \footnote{It is an interesting fact that the condition that $f$ has non-degenerate Jacobians on $U$ can be dropped. 
    Contrary to a real version of Theorem \ref{thm.inverse.function} where this condition is necessary,
    it is true that if $f$ is holomorphic and an injection on the open set $U$ then $f$ is biholomorphism between $U$ and $f(U)$ 
    (see e.g. \cite{range2013holomorphic}, discussion at page 23).}
  
  \smallskip
  
  We now show that for
  $R = \min(\frac{r_1}{\alpha_1},\ldots,\frac{r_d}{\alpha_d})$,
  $Q_R(\vec{b}) \subseteq f(H(\av))$.  $f$ is a homeomorphism between
  $U$ and $V$, hence $S=f(H(\av))$ is a compact subset of $V$.  We
  take $Q_{R'}(\vec{b}) \subseteq S$ as the maximal cube centered at
  $\vec{b}$ that is contained in $S$.
  
  Then, there exists a point $\vec{p}$ such that
  $\vec{p} \in \partial S \cap \partial Q_{R'}(\vec{b})$. Put
  $\vec{h}=\vec{p}-\vec{b}$. $f^{-1}$ is continuously differentiable
  on $V \supset Q_{R'}(\vec{b})$, we can therefore apply the Mean
  Value Theorem in integral form and obtain (here the integral is
  applied to each component of the inverse Jacobian matrix)
  $$
  f^{-1}(\vec{b}+\vec{h})-f^{-1}(\vec{b}) = \left(\int_{0}^1
    J^{-1}(\vec{b}+t\vec{h})dt\right)\vec{h}.
  $$ 
  Then for each coordinate $j=1,\ldots,d,$ 
  \begin{equation}\label{eq.coordinate.mean.value}
    f_j^{-1}(\vec{b}+\vec{h})-f_j^{-1}(\vec{b}) = \left(\int_{0}^1 P_jJ^{-1}(\vec{b}+t\vec{h})dt\right)\vec{h}. 
  \end{equation} 
  $f$ is a homeomorphism between $U$ and $V$ hence $f^{-1}$ maps the boundary
  of $S$ into boundary of $f^{-1}(S)=Q_r(\vec{a})$. Therefore there exists a coordinate $\hat{j} \in \{1,\ldots,d\}$
  such that 
  $$\left|f_{\hat{j}}^{-1}(\vec{b}+\vec{h})-f_{\hat{j}}^{-1}(\vec{b})\right| = r_{\hat{j}}.$$ 
  Then by equation \eqref{eq.coordinate.mean.value}
  $$
  r_{\hat{j}}=\left|f_{\hat{j}}^{-1}(\vec{b}+\vec{h})-f_{\hat{j}}^{-1}(\vec{b})\right| = 
  \left|\left(\int_{0}^1 P_{\hat{j}}J^{-1}(\vec{b}+t\vec{h})dt\right)\vec{h}\right|
  \le \alpha_{\hat{j}} \|\vec{h}\|= \alpha_{\hat{j}} R'.
  $$ 
  Hence $R' \ge \frac{r_{\hat{j}}}{\alpha_{\hat{j}}} \ge \min(\frac{r_1}{\alpha_1},\ldots,\frac{r_d}{\alpha_d}) = R$.
  We get that 
  $$Q_{R}(\vec{b}) \subseteq Q_{R'}(\vec{b}) \subseteq S = f(H(\av)).$$ 
  
  Since we already argued that $V \supset f(H(\vec{a})) \supseteq Q_{R}(\vec{b})$ is open then 
  clearly $f^{-1}$ is holomorphic in an open neighborhood of $Q_{R}(\vec{b})$.
  This proves item (1) of Theorem \ref{thm.inverse.function}.
  
  \smallskip
  
  The second item of the Theorem is proved with a similar argument: 
  let $\vec{y}'',\vec{y}' \in Q_{R}(\vec{b})$ and put $\vec{h}'=\vec{y}''-\vec{y}'$. Applying again the Mean Value Theorem
  $$
  \left|f_j^{-1}(\vec{y}'+\vec{h}')-f_j^{-1}(\vec{y}')\right| = 
  \left|\left(\int_{0}^1 P_jJ^{-1}(\vec{y}'+t\vec{h}')dt\right)\vec{h}'\right|
  \le \alpha_j \|\vec{h}'\|.
  $$
  This proves item (2) of the Theorem.
\end{proof}

\section{Norm bounds on the inverse Jacobian
  matrix}\label{appendix.norm.bounds}

Let $F=(\av,\xv) \in \barP$,
$\av=(a_1,\ldots,x_d),\ \xv=(x_1,\ldots,x_d)$.  Put
$z_j = z_{j}(\lambda) = e^{2\pi i \lambda x_j},\ j=1,\ldots,d$. By
direct computation, the Jacobian matrix
$J=J_\lambda(F)=J_\lambda(\av,\xv)$, of $\FMG$ at $F$ is given by
\begin{equation}\label{eq.jacobian.factorization}
  {\small
  J_\lambda(\av,\xv)=
  \begin{bmatrix}
    1 & .. & 1 & 0 & .. & 0\\
    z_1 & .. & z_d & 1 & .. & 1\\
    z_1^2 & .. & z_d^2 & 2z_1 & .. & 2z_d\\
    \vdots & \ddots & \vdots & \vdots & \ddots & \vdots\\
    z_1^{2d-1} & .. & z_d^{2d-1} &
    (2d-1)z_1^{2d-2} & .. & (2d-1)z_d^{2d-2}
  \end{bmatrix}
  \begin{bmatrix}
    I_d & 0\\
    0 & D
  \end{bmatrix},}
\end{equation}
where $D$ is a $d \times d$ diagonal matrix, $D_{j,j}=a_j 2\pi i\lambda z_j,\ j=1,\ldots,d$,
and $I_d$ is the $d\times d$ identity matrix.

Denote the left hand matrix in the factorization
\eqref{eq.jacobian.factorization} by
$U_{2d}=U_{2d}(z_1,\allowbreak\ldots,z_d)$.  The matrix $U_{2d}$ is an
instance of a confluent Vandermonde matrix, whose inverses have been
extensively studied in
\cite{gautschi1962inverses,gautschi1963inverses,batenkov2016stability}. In
particular, the elements of $U_{2d}^{-1}$ can be constructed using the
coefficients of polynomials from an appropriate Hermite interpolation
scheme.  Consequently, we have the following result due to
\cite{gautschi1963inverses}.

\begin{theorem}[Gautschi, \cite{gautschi1963inverses}, eqs. (3.10), (3.12)]\label{thm.Gautschi}
  For $z_1,\ldots,z_d \in \C$ pairwise distinct, put
  $$U^{-1}_{2d}(z_1,\ldots,z_d)=\begin{bmatrix} A\\B\end{bmatrix},$$
  where $A,B$ are $d\times 2d$. Then we have the following upper
  bounds on the 1-norm of the rows of the blocks $A,B$
  \begin{align}
    \sum_{k=1}^{2d} |A_{j,k}| &\le (1+2(1+|z_j|)|\Delta_j|)\Gamma_j,\tab j=1,..,d, \label{eq.row.bound.amp} \\
    \sum_{k=1}^{2d} |B_{j,k}| &\le (1+|z_j|)\Gamma_j,\tab j=1,..,d, \label{eq.row.bound.nodes}
  \end{align}
  where $$\Delta_j = \sum_{\ell=1,\ell\ne j}^d \frac{1}{|z_j - z_{\ell}|},\tab 
  \Gamma_j = \left(\prod_{\ell=1,\ell\ne j}^d \frac{1+|z_{\ell}|}{|z_j-z_{\ell}|}\right)^2 .$$			 
\end{theorem}

\begin{proof}[Proof of Proposition \ref{prop.uniform.jacobian.bounds}]
  
  By the factorization \eqref{eq.jacobian.factorization}
  $$J_\lambda(F)=U_{2d}(z_1,\ldots,z_d)
  \begin{bmatrix}
    I_d & 0\\
    0 & D
  \end{bmatrix},
  $$
  where $z_1=e^{2\pi i \lambda x_1},\ldots,\allowbreak z_d=e^{2\pi i \lambda x_d}$ and 
  $D=D(z_1,\ldots,z_d)$ is the $d \times d$ diagonal matrix,
  $D_{j,j}=a_j 2\pi i\lambda z_j,\ j=1,\ldots,d$.

  By assumption, the mapped nodes $\{z_j\}$ are pairwise distinct, and
  so it immediately follows that $J_{\lambda}(F)$ is non-degenerate.

  \smallskip 

  Put $U^{-1}_{2d}=U^{-1}_{2d}(z_1,\ldots,z_d)=\begin{bmatrix} A\\B\end{bmatrix}$, where $A,B$ are $d\times 2d$.
  Put $\tilde{B}=D^{-1}B$. Then
  \begin{equation}\label{eq.inv.jacobian}
    J_\lambda^{-1}(F)= \begin{bmatrix} A\\ \tilde{B}\end{bmatrix}.  		
  \end{equation} 
  
  \smallskip
  
  By Theorem \ref{thm.Gautschi}
  \begin{align}
    \sum_{k=1}^{2d} |A_{j,k}| &\le (1+2(1+|z_j|)|\Delta_j|)\Gamma_j,& j=1,..,d, \label{eq.row.bound.amp.2} \\
    \sum_{k=1}^{2d} |B_{j,k}| &\le (1+|z_j|)\Gamma_j,& j=1,..,d, \label{eq.row.bound.nodes.2}
  \end{align}
  where
  $$\Delta_j = \sum_{\ell=1,\ell\ne j}^d \frac{1}{|z_j -
    z_{\ell}|},\tab \Gamma_j = \left(\prod_{\ell=1,\ell\ne j}^d
    \frac{1+|z_{\ell}|}{|z_j-z_{\ell}|}\right)^2 .$$
  
  \smallskip
  
  \begin{itemize}
  \item \underline{Non-cluster node} Let $\ell$ be such that
    $x_{\ell} \in \xv\setminus\xv^c$.
  
  \smallskip 
  
  By assumptions we have
  \begin{align*}
    |z_{\ell}-z_j| & \ge \tilde{\eta},& \forall x_{\ell} \in \xv\setminus\xv^c, x_j\in \xv, \ell\ne j.
  \end{align*}
  Then we  obtain 
  \begin{equation}\label{eq.delta.non.cluster.Delta}
    \Delta_{\ell} = \sum_{j=1,j\ne \ell}^d \frac{1}{|z_{\ell} - z_j|} 
    \le \frac{d-1}{
      \tilde{\eta}}=K_5(\tilde{\eta},d),
  \end{equation}
  while
  \begin{align}\label{eq.gamma.non.cluster.Gamma}
    \begin{split}
      \Gamma_{\ell} = \left(\prod_{j=1,j\ne \ell}^d \frac{1+|z_j|}{|z_{\ell}-z_j|}\right)^2 \le & 
      \left(3^{d-1} \prod_{j=1,j\ne \ell}^d
        \frac{1}{|z_{\ell}-z_j|}\right)^2\\ 
      \le & 
      \left(3^{d-1}\frac{\tilde{\eta}^{-d+1}}{\left(\lfloor\frac{d-p}{2}\rfloor!\right)^2}\right)^2
      \\
      = & \left(\left(\frac{3}{\tilde{\eta}}\right)^{d-1}\frac{1}{\left(\lfloor\frac{d-p}{2}\rfloor!\right)^2}\right)^2\\
      = & K_{6}(\tilde{\eta},d,p).
    \end{split}
  \end{align}
  
  \smallskip
  
  Inserting equations \eqref{eq.delta.non.cluster.Delta} and
  \eqref{eq.gamma.non.cluster.Gamma} into \eqref{eq.row.bound.amp.2} and \eqref{eq.row.bound.nodes.2}, we get
  \begin{align}\label{eq.bound.amp.non.cluster}
    \sum_{k=1}^{2d} |A_{\ell,k}| \le
    (1+2(1+|z_{\ell}|)|\Delta_{\ell}|)\Gamma_{\ell} \le (1+6K_5)K_6=K_1(\tilde{\eta},d,p), 
  \end{align}
  and
  \begin{align}\label{eq.bound.nodes.non.cluster}
    \sum_{k=1}^{2d} |B_{\ell,k}| \le
    (1+|z_{\ell}|)\Gamma_{\ell} \le 3K_6=K_7(\tilde{\eta},d,p), 
  \end{align}
  for each $\ell$ such that $x_{\ell} \in \xv\setminus \xv^c$.

  \smallskip
  
  Now we are ready to bound the norms of rows of the blocks $A,\tilde{B}$
  for each non-cluster node index. 
  
  For the block $A$, such bound is given in equation \eqref{eq.bound.amp.non.cluster}.
  
  For the block $\tilde{B}$, we have, using equation \ref{eq.bound.nodes.non.cluster},
  \begin{align}
    \sum_{k=1}^{2d} |\tilde{B}_{\ell,k}| = \sum_{k=1}^{2d} |(a_{\ell} 2\pi i\lambda z_l)^{-1}||B_{\ell,k}| \le
    \frac{2 K_7}{\pi m} \frac{1}{\lambda}=K_{2}(m,\tilde{\eta},d,p)\frac{1}{\lambda}, 
  \end{align}						 	 
  for each $\ell$ such that $x_{\ell} \in \xv\setminus \xv^c$.

  This completes the proof of equations \eqref{eq.prop.1} and \eqref{eq.prop.2} of Proposition
  \ref{prop.uniform.jacobian.bounds}.
  
  \smallskip
    \item \underline{Cluster node}

  We now bound the norm of each row of $J_\lambda^{-1}(F)$ at an index corresponding to a cluster node.
  
  \smallskip 
  
  By assumptions 
  \begin{align*}
    |z_{j}-z_{k}| & \ge \tilde{h} ,& \forall x_j,x_k, \in \xv^{c}, j\ne k,\\
    |z_j-z_{\ell}| &\ge \tilde{\eta},
                                   & \forall x_j \in \xv^{c}, x_{\ell} \in \xv\setminus \xv^c. 
  \end{align*}
  Then for each $j$ such that $x_j \in \xv^c$
  \begin{equation}\label{eq.delta.cluster.1}
    \Delta_j = \sum_{\ell=1,\ell\ne j}^d \frac{1}{|z_j - z_{\ell}|} 
    \le \frac{d-1}{\tilde{h}},
  \end{equation}
  while
  \begin{align}\label{eq.gamma.cluster.1}
    \begin{split}
      \Gamma_j = \left(\prod_{\ell=1,\ell\ne j}^d \frac{1+|z_{\ell}|}{|z_j-z_{\ell}|}\right)^2 \le & \left(
        3^{d-1}\prod_{\ell=1,\ell\ne j}^d
        \frac{1}{|z_j-z_{\ell}|}\right)^2\\ 
      \le &
      \left(3^{d-1} \frac{\tilde{\eta}^{-d+p}
          \tilde{h}^{-p+1}}
        {\left(\lfloor\frac{d-p}{2}\rfloor!\right)^2}\right)^2
      \\
      = & K_{8}(\tilde{\eta},d,p)\tilde{h}^{-2p+2},
    \end{split}
  \end{align}
  where $K_{8}(\tilde{\eta},d,p)= \left(3^{d-1} \frac{\tilde{\eta}^{-d+p}}
    {\left(\lfloor\frac{d-p}{2}\rfloor!\right)^2}\right)^2$.
  
  \smallskip
  
  Inserting equations \eqref{eq.delta.cluster.1} and
  \eqref{eq.gamma.cluster.1} into \eqref{eq.row.bound.amp.2} and \eqref{eq.row.bound.nodes.2}, we get
  \begin{align}\label{eq.bound.amp.cluster}
    \begin{split}
      \sum_{k=1}^{2d} |A_{j,k}| \le
      (1+2(1+|z_j|)|\Delta_j|)\Gamma_j & \le 7(d-1)K_8\tilde{h}^{-2p+1}\\
      &=K_{3}(\tilde{\eta},d,p)\tilde{h}^{-2p+1},
    \end{split}
  \end{align}
  \begin{align}\label{eq.bound.nodes.cluster}
    \begin{split}
      \sum_{k=1}^{2d} |B_{j,k}| \le
      (1+|z_j|)\Gamma_j & \le 3K_8\tilde{h}^{-2p+2}\\
      & =K_{9}(\tilde{\eta},d,p)\tilde{h}^{-2p+2},
    \end{split}
  \end{align}
  for each $j$ such that $x_j \in \xv^c$.
  
  \smallskip
  
  We now bound the norms of rows of the blocks $A,\tilde{B}$
  for each cluster node index. 
  
  For the block $A$, the bound was given in equation \eqref{eq.bound.amp.cluster}.
  
  For the block $\tilde{B}$, we have, using equation \ref{eq.bound.nodes.cluster},
  \begin{align}
    \begin{split}
      \sum_{k=1}^{2d} |\tilde{B}_{j,k}| = \sum_{k=1}^{2d} |(a_j 2\pi i\lambda z_j)^{-1}||B_{j,k}| & \le
      \frac{2K_{9}}{\pi m}\frac{1}{\lambda}\tilde{h}^{-2p+2}\\
      & =K_{4}(\tilde{\eta},d,p,m)\frac{1}{\lambda} \tilde{h}^{-2p+2},
    \end{split}
  \end{align}						 	 
  for each $j$ such that $x_j \in \xv^c$.
  
  \smallskip
  
  This completes the proof of equations \eqref{eq.prop.3} and \eqref{eq.prop.4} of Proposition
  \ref{prop.uniform.jacobian.bounds}.
    \end{itemize}

\end{proof}

\section{Proof of Proposition \ref{prop.one.to.one}}\label{appendix.proof.one.to.one}

\begin{proof}
  
  Let the map $g=g_{\lambda}: \barP \simeq \C^{2d} \rightarrow \C^{2d}$ be defined as 
  \begin{align}\label{eq.def.Gg}
    g_k(a_1,\ldots,a_d,x_1,\ldots,x_d) &= a_k,& k=1,\ldots,d,\\
    g_{d+k}(a_1,\ldots,a_d,x_1,\ldots,x_d) &= e^{2\pi i \lambda x_{k}},& k=1,\ldots,d.\nonumber					
  \end{align}
  
  Consider the definition of the Prony map $PM$ from
  \eqref{eq.def.PM}. We thus have
  \begin{equation}\label{eq.FM.Lambda.0.composition}
    FM_{\lambda} = PM \circ g_{\lambda}.				
  \end{equation}

  Put
  $$W=g_{\lambda}(H^{\mathrm{o}}_{m,\frac{\tau h}{2 \pi}}(F)) = g_{\lambda}(U).$$ 			
  We will show that $g_\lambda$ is injective on $U$ and that $PM$ is injective on $W$. 
  \smallskip
  
  First we show that $PM$ is injective on $W$.

  Proposition \ref{prop.unique.solution.prony} gives sufficient conditions 
  for $PM$ to be one to one on a subset of $\C^{2d}$,
  the next Proposition asserts that these conditions hold for $W$.
  
  \begin{proposition}\label{prop.geo.local.W}
    Let $\lambda \in \Lambda(\xv)$. Then for each
    $\vec{v}',\vec{v}'' \in W =
    g_{\lambda}(H^{\mathrm{o}}_{m,\frac{\tau h}{2 \pi}}(F))\allowbreak
    =g_{\lambda}(U)$, with $\vec{v}'=(\av',\vec{z}')$,
    $\av'=(a'_1,\allowbreak \ldots,a'_d)$, $\vec{z}'=(z'_1,\ldots,z'_d)$,
    $\vec{v}''=(\av'',\vec{z}'')$, $\av''=(a''_1,\ldots,a''_d)$,
    $\vec{z}''=(z''_1,\allowbreak\ldots,z''_d)$, and $\vec{v}'\ne \vec{v}''$, it holds
    that:
    \begin{enumerate} 
    \item\label{it.geo.1} $a'_j \ne 0$ for $j= 1,\ldots,d.$ 
    \item\label{it.geo.2} $z'_j \ne z'_k$ for each $1 \le j<k \le d$. 
    \item\label{it.geo.3} $z'_j \ne z''_k$ for all $1 \le j<k \le d$.
    \end{enumerate} 
  \end{proposition}
  \begin{proof}
    Let $\lambda \in \Lambda(\xv)$ and let $\vec{v}',\vec{v}'' \in g_{\lambda}(H^{\mathrm{o}}_{m,\frac{\tau h}{2 \pi}}(F))$
    as specified in Proposition \ref{prop.geo.local.W}.
    
    The first assertion is apparent from the fact that
    $\|\av'-\av\|<m$ and the assumption that $|a_j|\ge m $ for
    $j=1,\ldots,d$.
    
    \smallskip
    
    We now prove assertions \ref{it.geo.2} and \ref{it.geo.3}.
    
    \smallskip
    
    Let $\vec{z}=(z_1,\ldots,\allowbreak z_d)$, with $z_1=e^{2\pi i \lambda x_1},\ldots,\allowbreak 
    z_{d}=e^{2\pi i \lambda x_d }$. 
    
    \smallskip
    
    As a first step we argue that for each pair of mapped nodes $z_j, z_k, 1 \le j < k \le d$,
    \begin{align}\label{eq.geo.local.equclidian.dist}
      |z_{j}-z_{k}| & \ge 4\lambda \tau h,
      & 1 \le j < k \le d.
    \end{align}
    
    Indeed with the assumption that $\Omega h \le \frac{1}{20d}$ we have that
    \begin{equation}\label{eq.geo.local.angles}
      \frac{\pi}{2} > \frac{1}{d^2} > 2\pi \lambda \tau h.
    \end{equation}
    By \eqref{eq.geo.local.angles} and since $\lambda \in \Lambda(\xv)$
    \begin{equation}\label{eq.geo.local.angles.2}
      \angle(z_j,z_k) \ge 2\pi \lambda \tau h.					
    \end{equation}
    Then by \eqref{eq.geo.local.angles}, \eqref{eq.geo.local.angles.2} and \eqref{eq.angle.to.euclidien}
    \begin{align*}
      |z_{j}-z_{k}| & \ge 4\lambda \tau h.
    \end{align*} 
    
    Next we claim that 
    \begin{equation}\label{eq.local.geo.W}
      W \subset H^{\mathrm{o}}_{m,2\lambda\tau h}(\av,\vec{z}) = 
      \left\{ (\av',\vec{z}') \in \C^{2d} \ :
        \|\av'-\av\| < m,\ \|\vec{z}'-\vec{z}\| < 2\lambda\tau h \right\}.
    \end{equation}
    Let $(\av''',\xv''') \in H^{\mathrm{o}}_{m,\frac{\tau h}{2 \pi}}(F)$. To show \eqref{eq.local.geo.W}, 
    we need to verify that $g_\lambda(\av''',\xv''') \in H^{\mathrm{o}}_{m,2\lambda\tau h}(\av,\vec{z})$.	 
    For this purpose put $g_{\lambda}(\av''',\xv''')=(\av''',\vec{z}''')$, 
    $\vec{z}'''=(e^{2\pi i \lambda x'''_1},\ldots, e^{2\pi i \lambda x'''_d})$. Then using 
    the integral mean value bound, for any $j=1,\ldots,d$,	
    \begin{align*}
      \left|e^{2\pi i \lambda x'''_j} - e^{2\pi i \lambda x_j}\right| & \le \max_{c \in \{ x_j + t(x_j'''-x_j) : t \in [0,1]\}}
                                                                        \left|\frac{d}{dx}e^{2\pi i \lambda x}\Big|_c\right| \frac{\tau h}{2 \pi}\\
                                                                      & \le \lambda\tau h e^{\lambda h }\\
                                                                      & < 2\lambda\tau h,
    \end{align*}
    where in the last step we used the assumption $\Omega h \le \frac{1}{20d}$ and the fact that $\lambda \le \frac{\Omega}{2d-1}$,
    which then implies that $e^{\lambda h } < 2$. This in turn proves \eqref{eq.local.geo.W}.
    
    \smallskip 
    
    We now prove assertion \ref{it.geo.2}. 
    
    Let $1 \le j<k \le d$ and 
    assume by contradiction that $z'_j=z'_k$. 
    By \eqref{eq.local.geo.W}, $(\av',\vec{z}') \in H^{\mathrm{o}}_{m,2\lambda\tau h }(\av,\vec{z})$ then
    $|z_j-z'_j|<2\lambda\tau h $ and $|z_k-z'_j|=|z_k-z'_k|<2\lambda\tau h $.
    Then
    $$|z_j-z_k| \le |z_j-z'_j| + |z_k-z'_j| < 4\lambda\tau h ,$$
    which is a contradiction to \eqref{eq.geo.local.equclidian.dist}. 
    
    \smallskip 
    
    Finally we prove assertion \ref{it.geo.3}.
    
    Assume by contradiction that for $1 \le j<k \le d$, $z'_j = z''_k$. 
    By \eqref{eq.local.geo.W} $|z_j-z'_j| < 2\lambda\tau h$. By assumption $|z_k-z_j'| = |z_k-z_k''|$ then
    by \eqref{eq.local.geo.W} $|z_k-z_j'| < 2\lambda\tau h$. Using these 
    $$|z_j-z_k| \le |z_j-z'_j| + |z_k-z_j'| < 4\lambda\tau h,$$
    which is a contradiction to \eqref{eq.geo.local.equclidian.dist}.
    
    This completes the proof of Proposition \ref{prop.geo.local.W}.
  \end{proof}
  Now by Propositions \ref{prop.geo.local.W} and \ref{prop.unique.solution.prony} we have that $PM$ is 
  injective on $W$.
  
  \smallskip
  
  We now show that $g_\lambda$ is injective on $U$.
  
  \begin{proposition}\label{prop.g.injective}
    For each $\lambda >0 $, the map $g_{\lambda}$ is injective in the polydisc
    $H^{\mathrm{o}}_{m,\frac{1}{2\lambda}}(F)$.
  \end{proposition}
  \begin{proof}
    Let $(\av',\xv'),(\av'',\xv'') \in H^{\mathrm{o}}_{m,\frac{1}{2\lambda}}(F)$ such that 
    $g(\av'',\xv'')\allowbreak=g(\av',\xv')$. We will show that $(\av',\xv')=(\av'',\xv'')$. 
    
    For the amplitudes coordinates $k=1,\ldots,d$,
    $g_k(a_1,\ldots,a_d,x_1,\ldots,x_d)=a_k$ therefore $\av''=\av'$.
    
    For coordinates $d+1,\ldots,2d$, 
    $$g_{d+j}(a_1,\ldots,a_d,x_1,\ldots,x_d)= g_{d+j}(x_{j}) = e^{2\pi i \lambda x_j}, \tab j=1,\ldots,d.$$
    Fix a certain $1 \le j \le d$ and set $x'_j = \alpha'_j + \beta'_j i$, $\alpha'_j,\beta'_j \in \R$. 
    The set of complex numbers $w=\alpha + \beta i$ 
    such that $g_{d+j}(w)= g_{d+j}(x'_{j}) = e^{2\pi i \lambda x'_{j}}$ is equal to 
    $$S_j = \left\{\alpha + \beta i : \beta = \beta'_j,\ \alpha = \alpha'_j + \frac{\ell}{\lambda}, \ \forall \ell\in{\mathbb Z} \right\}.$$
    Since $(\av',\xv'),(\av'',\xv'') \in H^{\mathrm{o}}_{m,\frac{1}{2\lambda}}(F)$ implies that $|x'_j-x_j''| <
    \frac{1}{\lambda}$ then $x_j''=x_j'$ and because $j$ was chosen arbitrarily we have $\xv''=\xv'$.
  \end{proof} 
  
  By assumption $\lambda \le \frac{\Omega}{2d-1}$ and $\Omega h \le \frac{1}{20 d}$ then $\frac{1}{\lambda} > h$.
  Using the former, $U = H^{\mathrm{o}}_{m,\frac{\tau h}{2 \pi}}(F) \subset H^{\mathrm{o}}_{m,\frac{1}{2\lambda}}(F) $
  then by Proposition \ref{prop.g.injective} $g_\lambda$ is injective on $U$.
  
  We have shown that $g_\lambda$ is injective on $U$ and that $PM$ is injective on $W = g_\lambda(U)$ then by 
  \eqref{eq.FM.Lambda.0.composition} $\FMG$ is injective on $U$. 
  
  This completes the proof of Proposition \ref{prop.one.to.one}. 
\end{proof}

\section{Proof of Proposition \ref{prop.global.geo}}\label{appendix.global.geo}
\begin{proof}
  First observe that if $F' \in \P_d$ is of the form $F'=(\av'^{\pi},\xv'^{\pi}) + \frac{1}{\lambda}\lv$, with 
  $\pi \in \Pi_d$ and $\lv \in \Z^{d}$, and $(\av',\xv')\in A_{\epsilon,\lambda}(F)$ then
  \begin{align*}
    \FMG(F')&=\FMG\left(\left(\av'^{\pi},\xv'^{\pi}+\frac{1}{\lambda}\lv\right)\right)\\
            &=\sum_{j=1}^d \av'_{\pi(j)}e^{2\pi i \lambda (\xv_{\pi(j)}'
              +\frac{\lv_j}{\lambda})}\\
            &=\sum_{j=1}^d \av'_{\pi(j)}e^{2\pi i \lambda \xv_{\pi(j)}'}\\
            &=\sum_{j=1}^d \av'_{j}e^{2\pi i \lambda \xv_{j}'}\\
            &=\FMG\left((\av',\xv')\right).
  \end{align*}	
  Since by definition of $A_{\epsilon,\lambda}(F)$ (see equation
  \eqref{eq.def.AG} ), $(\av',\xv') \in A_{\epsilon,\lambda}(F)$ implies
  that $(\av',\xv') \in \EG$, then the above shows that
  $$\EG \supseteq \left(\bigcup_{\pi \in \Pi_d} \bigcup_{\lv \in \Z^d} A_{\epsilon,\lambda}^{\pi}(F)+\frac{1}{\lambda}\lv
  \right)\bigcap \P_d.$$
  
  \smallskip
  
  For the other direction, let $F'=(\av',\vec{y}') \in \EG$ with $\av'=(a'_1,\ldots,a'_d)$
  and $\vec{y}'=(y'_1,\ldots,y'_d)$. Put $\vec{\mu}' = \FMG(F')$, then
  $\vec{\mu}' \in Q_{\epsilon}(\vec{\mu}_{\lambda})$ (with $\vec{\mu}_{\lambda} = \FMG(F)$ as above). 
  
  By definition of the set $A_{\epsilon,\lambda}(F)$, there exists a signal $F'' \in A_{\epsilon,\lambda}(F)$ such that 
  $\FMG(F'') = \vec{\mu}'$, and put $F''=(\av'',\xv'')$ with $\av''=(a''_1,\ldots,a''_d)$
  and $\xv''=(x''_1,\allowbreak \ldots,x''_d)$.
  
  Recall that by \eqref{eq.FM.Lambda.0.composition} (see
  \eqref{eq.def.PM} and \eqref{eq.def.Gg})
  $$\FMG = PM \circ g_\lambda.$$
  Put $g_\lambda(F'')=(\av'',\vec{z}'')$ with
  $\vec{z}''=(z''_1,\ldots,z''_d)$, $z_j'' = e^{2\pi i \lambda x''_j}$
  for $j=1,\ldots,d$. By Proposition \ref{prop.geo.local.W} each point
  in $W = g_\lambda(U)$ has non-vanishing amplitudes and pairwise
  distinct nodes.  We have that
  $F'' \in A_{\epsilon,\lambda}(F) \subseteq U$ and hence
  $(\av'',\vec{z}'')$ satisfies the above properties.  Then by
  Proposition \ref{prop.unique.solution.prony} the set of all
  solutions to the equation $PM\left((\av,\vec{z})\right)=\vec{\mu}'$
  is given by
  \begin{equation}\label{eq.permutations.solutions}
    \left\{ (\av''^{\pi},\vec{z}''^{\pi}) : \pi \in \Pi_d \right\}.		 	
  \end{equation}
  By \eqref{eq.permutations.solutions} there exists $\pi \in \Pi_d$ 
  such that 
  $$g_\lambda(F')=g_\lambda\left((\av',\vec{y}')\right)=(\av''^{\pi},\vec{z}''^{\pi}).$$ 
  
  Finally since $x''_1,\ldots,x''_d$ are real, the set of all
  solutions to the equation
  $g_\lambda\left((\av,\xv)\right)=(\av''^{\pi},\vec{z}''^{\pi})$ is
  given by
  $$\left\{(\av''^{\pi},\xv''^{\pi}+\frac{1}{\lambda}\lv) : \lv\in \Z^d\right\}.$$
  By the above, $F'$ is of the form
  $\left(\av''^{\pi},\xv''^{\pi}+\frac{1}{\lambda}\lv\right)$ for some
  $\pi \in \Pi_d$ and $\lv \in \Z^d$.
  
  This concludes the proof of Proposition \ref{prop.global.geo}.		 	
\end{proof} 

\section{Proof of Proposition \ref{prop.elimiate.period}}\label{appendix.upper.bound.tech}
  Within the course of the proof we will make appropriate assumptions of the form
  $\frac{C'}{\eta} \le \Omega \le \frac{C''}{h}$, with $C',C''$ being constants depending only on $d$, 
  for which some arguments of the proof hold. It is to be understood that 
  $K_9$ is the maximum of the constants $C'$ and $K_{10}$ is the minimum of
  the constants $C''$.  

Assume that $\Omega \ge \frac{2(2d-1)}{\eta}$. Then 
the length of the interval 
$\left[\frac{1}{2}\frac{\Omega}{2d-1},\frac{\Omega}{2d-1}\right]$ 
is larger than $\frac{1}{\eta}$ and by Proposition \ref{prop.good.blowup} there exists 
an interval $I \subseteq
\left[\frac{1}{2}\frac{\Omega}{2d-1},\frac{1}{2}\frac{\Omega}{2d-1}+\frac{1}{\eta}\right]$
such that
\begin{equation}\label{eq.def.lambda.1}
  I\subset \Lambda(\xv), \tab |I| = (2d^2 \eta)^{-1}.
\end{equation} 
Fix 
$$I_1 = [\lambda_1,\lambda_1+(2d^2\eta)^{-1}] \subseteq \Lambda(\xv) \cap
\left[\frac{1}{2}\frac{\Omega}{2d-1},\frac{1}{2}\frac{\Omega}{2d-1}+\frac{1}{\eta}\right]
$$
to be the sub-interval of $\Lambda(\xv) \cap
\left[\frac{1}{2}\frac{\Omega}{2d-1},\frac{1}{2}\frac{\Omega}{2d-1}+\frac{1}{\eta}\right]$ 
with the minimal starting point $\lambda_1$ which satisfies \eqref{eq.def.lambda.1}.
We will show that there exists $\lambda \in I_1$ that satisfies \eqref{eq.what.we.intend.to.show}.	

\smallskip 

We require the following intermediate results.

\smallskip 

As in Section \ref{section.decimation} we denote by $\nu$ the Lebesgue measure on $\R$.
\begin{lemma}\label{lem.real.intervals.1}
  Let $\frac{1}{2} \le a<1$ and $I=[a,1]$. 
  Then for each $\epsilon,\alpha,c \in \R$ such that $0 < \alpha \le 1$,
  $0 < \epsilon \le \frac{1}{100} \alpha$ and $|c| \ge 8 \frac{\epsilon}{\alpha |I|}$, it holds that
  $$ \nu \big( \left\{x\in I : 
    \exists k\in \Z \mbox{ such that } \left| kx - c\right|\le \epsilon \right\}\big) < \alpha |I|.$$
\end{lemma}	 	
\begin{lemma}\label{lem.real.intervals.2}
  Consider the interval $[a,b] \subset (0,\infty)$ and let $S \subseteq [a,b]$ be a union of $N$ 
  disjoint sub-intervals 
  $S=\bigcup_{i=1}^{N} [a_i,b_i]$.  Set $I^{-1}=[\frac{1}{b},\frac{1}{a}]$ and $S^{-1} = \bigcup_{i=1}^{N} [\frac{1}{b_i},\frac{1}{a_i}]$.
  Then 
  $$\frac{\nu(S)}{\nu(I)} \le \frac{b}{a} \frac{\nu(S^{-1})}{\nu(I^{-1})} .$$ 
\end{lemma}
\begin{proposition}\label{prop.small.c}
	There exists constants $K_{11}, K_{12}$ depending only on $d$ such that for $\frac{K_{11}}{\eta}\le \Omega \le
	\frac{K_{12}}{h}$ the following holds. For each $3h < |c| \le \frac{\eta}{6}$, there exists an interval $I \subset
	\Lambda(\xv)$ of length $|I| = (2 d^2 \eta)^{-1}$ such that for all $\lambda \in I $ and for all $k \in \Z$ 
	\begin{equation}\label{eq.small.c}
		\left| c-\frac{k}{\lambda}\right| > 3 h.
	\end{equation}
\end{proposition}

We now complete the proof of Proposition \ref{prop.elimiate.period} using
the claims above, and provide their proofs thereafter.

\bigskip

{\bf Step 1:}

\bigskip

First it is shown, 
using Lemma \ref{lem.real.intervals.1} and Lemma \ref{lem.real.intervals.2}, 
that there exists $\gs \in I_1$ such that for all pair of distinct nodes $i,j$ 
with not both $x_i, x_j$ in $\xv^c$, it holds that
\begin{equation}\label{eq.nodes.mod.lambda}
  \left|x_i-x_j+\frac{n}{\gs}\right| > (32d^4)^{-1}\frac{1}{\lambda_1}, \tab \mbox{for all } n\in \Z.
\end{equation}

Put 
$$ I^{-1}_1=\left[\frac{1}{\lambda_1+(d^2 2 \eta)^{-1}},\frac{1}{\lambda_1}\right],\tab \tilde{I}^{-1}_1= 
\lambda_1 I^{-1}_1=\left[\frac{\lambda_1}{\lambda_1+(d^2 2 \eta)^{-1}},1\right].$$

Fix any distinct indices $i,j$ such that not both $x_i,x_j$ are in $\xv^c$. Put $c_{i,j} = x_i-x_j$ and observe that under the cluster assumption
\begin{equation}\label{eq.cij}
  |c_{i,j}| \ge \eta.
\end{equation}
Put $I=\tilde{I}^{-1}_1$, $c = c_{i,j} \lambda_1$, $\epsilon = (32d^4)^{-1}$ and $\alpha = \frac{1}{d^2}$.
We now validate that under appropriate assumptions on the size of $\Omega$ 
we have that $I,c,\epsilon,\alpha$ satisfy the conditions
of Lemma \ref{lem.real.intervals.1}.
Put $a$ as the left end point of the interval $I$ then 
with $\Omega \ge \frac{2}{\eta d}$ we have that $a\ge \frac{1}{2}$. 
With $d\ge 2$ by assumption we have that $\epsilon = \frac{1}{32 d^4} < \frac{1}{100d^2}$.
With $\Omega \ge \frac{2}{\eta d}$ we have that
\begin{equation}\label{eq.length.I.1}
  |I|=|\tilde{I}^{-1}_1| \ge (4d^2 \eta \lambda_1)^{-1}.			
\end{equation}
Now with \eqref{eq.cij} and \eqref{eq.length.I.1} we have that $|c| \ge 8 \frac{\epsilon}{\alpha |I|}$.
Having validated the conditions of Lemma \ref{lem.real.intervals.1} hold for $I,c,\epsilon,\alpha$ we 
now invoke it and get that 
$$
  \nu \big( \left\{ t \in \tilde{I}^{-1}_1 : 
    \exists k\in \Z \mbox{ such that } \left| kt - c_{i,j}\lambda_1 \right|\le (32d^4)^{-1} \right\}\big) < \frac{1}{d^2} |\tilde{I}^{-1}_1|.
$$
Then
$$
  \nu \big( \left\{ t \in I^{-1}_1 : 
    \exists k\in \Z \mbox{ such that } \left| kt - c_{i,j}\right|\le (32d^4)^{-1}\frac{1}{\lambda_1} \right\}\big) < \frac{1}{d^2} |I^{-1}_1|.
$$
Now we apply Lemma \ref{lem.real.intervals.2} and conclude from the above that 
\begin{equation}\label{eq.appen.upper.3}
  \nu \big( \left\{ \lambda \in I_1 : 
    \exists k\in \Z \mbox{ such that } \left| \frac{k}{\lambda} - c_{i,j}\right|\le (32d^4)^{-1}\frac{1}{\lambda_1} \right\}\big) < \frac{2}{d^2} |I_1|.
\end{equation}
Define the set 
$$E=\bigcup_{\substack{1\le i<j\le d\\ \neg (x_i\in \xv^c \wedge x_j \in \xv^{c})}}\left\{ \lambda \in I_1 : 
  \exists k\in \Z \mbox{ such that } \left| \frac{k}{\lambda} - c_{i,j}\right|\le (32d^4)^{-1}\frac{1}{\lambda_1} \right\}.$$ 
Then using \eqref{eq.appen.upper.3} and the union bound 
\begin{equation}\label{eq.g.s}
  \nu \big(E\big) < \binom{d}{2}\frac{2}{d^2}|I_1|<|I_1|.
\end{equation}
We conclude from \eqref{eq.g.s} that there exists $\lambda^{*} \in I_1$ which satisfies \eqref{eq.nodes.mod.lambda}.

\bigskip

{\bf Step 2:}

\bigskip

Now we show that in fact $\lambda^{*}$ satisfies \eqref{eq.what.we.intend.to.show}, i.e. it satisfies the condition 
of Proposition \ref{prop.elimiate.period}. 

\smallskip

Let
$(\tilde{\pi},\tilde{\lv}) \in \Pi_d \times
(\Z^d\setminus\{\vec{0}\})$.  We will show that there exists
$\lambda_{\tilde{\pi},\tilde{\lv}} \in \Lambda(\xv)$ such that for all
$\pi \in \Pi_d$ and for all $\lv \in \Z^d$
\begin{equation}\label{eq.empty.intersection}
  \left(A^{\tilde{\pi}}_{R,\lambda^*}(F) + \frac{1}{\lambda^*} \tilde{\lv}\right) \cap
  \left(A^{\pi}_{R,\lambda_{\tilde{\pi},\tilde{\lv}}}(F) + \frac{1}{\lambda_{\tilde{\pi},\tilde{\lv}}} \lv\right) = \emptyset.
\end{equation}

Proposition \ref{prop.elimiate.period} will then follow
by Proposition \ref{prop.global.geo}.

\smallskip

We can assume without loss of generality that $\tilde{\pi}=id$. Accordingly
we put $A^{\tilde{\pi}}_{R,\lambda^*}(F) = A_{R,\lambda^*}(F)$ and we will
prove that there exists $\lambda_{\tilde{\lv}} \in \Lambda(\xv)$ such that for all $\pi \in \Pi_d$ and
for all $\lv \in \Z^d$
\begin{equation}\label{eq.empty.intersection.1}
  \left(A_{R,\lambda^*}(F) + \frac{1}{\lambda^*} \tilde{\lv}\right) \cap
  \left(A^{\pi}_{R,\lambda_{\tilde{\lv}}}(F) + \frac{1}{\lambda_{\tilde{\lv}}} \lv\right) = \emptyset.
\end{equation}

Fix $i$ such that $\tilde{\lv}_i \ne 0$ and set $n=\tilde{\lv}_i$. Assume that $x_i \in \xv^{c}$, 
and one can verify that the case where $x_i \in \xv\setminus\xv^c$ is proved using a similar argument
to the one that is given below.

In the cases considered below we will use the following fact about 
the ``radius'' of the set $A_{R,\lambda}(F)$ for each $\lambda \in \Lambda(\xv)$, 
established in Proposition \ref{prop.lipshitz}. 
For each $F'=(\av',\xv')\in A_{R,\lambda}(F)$ with $\xv'=(x'_1,\ldots,x'_d)$,
\begin{align}\label{eq.A.R.radious}
  \left|x_j' - x_j\right| & \le \tilde{C}_{1} \frac{1}{\Omega}(\Omega \tau h)^{-2p+2}R
                            \le h,& j=1,\ldots,d.
\end{align}

We consider the following mutually exclusive and collectively exhaustive cases: 

\bigskip 

{\bf Case 1:} $\frac{n}{\lambda^*}\le \frac{\eta}{6}$.

\bigskip

Put $c=\frac{n}{\lambda^*}$. Then under the assumption of this case and 
with $\Omega \ge \frac{d}{3h}$ we have
that $3h <|c| \le \frac{\eta}{6}$. We can therefore apply Proposition \ref{prop.small.c} for $c$
and (under appropriate further assumptions on $\Omega$) get that 
there exists an interval $I_2 \subset \Lambda(\xv)$ of
length $|I_2| = (2 d^2 \eta)^{-1}$, 
such that for all $\lambda \in I_2 $ and for all $k \in \Z$ it holds that 
\begin{equation}\label{eq.appen.case.1.cluster.property}
	\left|c-\frac{k}{\lambda}\right| = \left|\frac{n}{\lambda^*}-\frac{k}{\lambda}\right| > 3 h.
\end{equation}
Put
$$ I_2=[\lambda_2, \lambda_2 +(d^2 2 \eta)^{-1}], \tab I^{-1}_2=\left[\frac{1}{\lambda_2+(d^2 2
\eta)^{-1}},\frac{1}{\lambda_2}\right],\tab
\tilde{I}^{-1}_2= \lambda_2 I^{-1}_2.$$

Let $1 \le j \le d$ be any index such that $x_j \in \xv\setminus \xv^c$.
Put $c_j=(x_i + \frac{n}{\lambda^*} -x_j)$.
Then 
\begin{equation}\label{eq.appen.case.1.non.cluster}
	|c_j|=|x_i + \frac{n}{\lambda^*} -x_j| \ge |x_i-x_j| - \frac{n}{\lambda^*} \ge \eta -\frac{n}{\lambda^*}  
	\ge \eta -\frac{\eta}{6} \ge  \frac{5}{6}\eta,
\end{equation} 
where in the second inequality we used the fact that $x_j$ is a non-cluster node and in the third
inequality we used the assumption of case 1. 

Put $I=I^{-1}_2$, $c=c_j \lambda_2$, 
$\epsilon=2h\lambda_2$ and $\alpha=\frac{1}{2d}$. 
By \eqref{eq.appen.case.1.non.cluster} we have that $|c|\ge \frac{5}{6}\eta \lambda_2$.
Using the former, one can validate that there exists positive constants $C'(d),C''(d)$ such that if 
$\frac{C'(d)}{\eta}\le \Omega \le \frac{C''(d)}{h}$, then $I,c,\epsilon,\alpha$ meet the conditions of
Lemma \ref{lem.real.intervals.1}. We then invoke Lemma \ref{lem.real.intervals.1} and get that
$$
  \nu \big( \left\{ t \in \tilde{I}^{-1}_2 : 
    \exists k\in \Z \mbox{ such that } \left| kt - c_j\lambda_2 \right|\le 2h\lambda_2   \right\}\big)
    < \frac{1}{2d} |\tilde{I}^{-1}_2|.
$$ 
Then
$$
  \nu \big( \left\{ t \in I^{-1}_2 : 
    \exists k\in \Z \mbox{ such that } \left| kt - c_j \right|\le 2 h  \right\}\big)
    < \frac{1}{2d} |I^{-1}_2|.
$$ 
By the above and using Lemma \eqref{lem.real.intervals.2}
\begin{equation}\label{eq.appen.case.1.measure}
	\nu \big( \left\{ \lambda \in I_2 : 
    \exists k\in \Z \mbox{ such that } \left|\frac{k}{\lambda}-c_j\right|\le 2 h
    \right\}\big) < \frac{1}{d} |I_2|.
\end{equation}
Define the set
$$
  E=\bigcup_{\substack{1\le j\le d,\\ x_j\notin \xv^c}}\left\{ \lambda \in I_2: 
    \exists k\in \Z \mbox{ such that } \left| \frac{k}{\lambda} - c_j\right|\le 2h \right\}.
$$
Using the union bound and \eqref{eq.appen.case.1.measure}
\begin{equation}
  \nu(E) < |I_2|. 
\end{equation}
We conclude from the above that there exists $\lambda \in I_2$ such that for any non-cluster
node $x_j$ and for any $k \in \Z$
$$
	\left|x_i + \frac{n}{\lambda^*} -x_j-\frac{k}{\lambda}\right|> 2h.
$$ 
On the other hand we have that for all $k \in \Z$ (see \eqref{eq.appen.case.1.cluster.property})
$$\left|\frac{n}{\lambda^*}-\frac{k}{\lambda}\right| > 3 h.$$
  
Fix $\lambda_{\tilde{\lv}} = \lambda$. Then using the above, for any $\pi \in \Pi_d$ and any $k \in \Z$,
if $x_{\pi(i)}$ is a cluster node then
\begin{equation}\label{eq.small.n.cluster.to.cluster}
	\left|x_i+\frac{n}{\lambda^*} -
	x_{\pi(i)}-\frac{k}{\lambda_{\tilde{\lv}}}\right|\ge
	\left|\frac{n}{\lambda^*}-\frac{k}{\lambda_{\tilde{\lv}}}\right|-\left|x_i- x_{\pi(i)}\right| >3h-h=2h,
\end{equation}
and if $x_{\pi(i)}$ is a non-cluster node then
\begin{equation}\label{eq.small.n.cluster.to.non.cluster}
	\left|x_i+\frac{n}{\lambda^*} -
	x_{\pi(i)}-\frac{k}{\lambda_{\tilde{\lv}}}\right|> 2h.
\end{equation}

Now by combing \eqref{eq.A.R.radious}, \eqref{eq.small.n.cluster.to.cluster}
and \eqref{eq.small.n.cluster.to.non.cluster}, we get that $\lambda_{\tilde{\lv}}$ satisfies \eqref{eq.empty.intersection.1} .
This completes the proof of case 1.

\bigskip 

{\bf Case 2:} $\frac{n}{\lambda^*}> \frac{\eta}{6}$ and $\forall y \in \xv\setminus\xv^c  
: |x_i+\frac{n}{\lambda^*} -y| > \frac{\eta}{6}.$

\bigskip

We show that in this case there exists $\lambda \in I_1$ such that $\lambda_{\tilde{\lv}} = \lambda$ satisfies
\eqref{eq.empty.intersection.1}.

Put (as above)
$$ I^{-1}_1=\left[\frac{1}{\lambda_1+(d^2 2 \eta)^{-1}},\frac{1}{\lambda_1}\right],\tab \tilde{I}^{-1}_1= 
\lambda_1 I^{-1}_1=\left[\frac{\lambda_1}{\lambda_1+(d^2 2 \eta)^{-1}},1\right].$$
Put $I = \tilde{I}^{-1}_1$, $c=\frac{n}{\lambda^*}\lambda_1$, $\epsilon=3 h \lambda_1$ and $\alpha = \frac{1}{4}$.
By the assumptions of this case we have $\frac{n}{\lambda^*}> \frac{\eta}{6}$, then $c=\frac{n}{\lambda^*}\lambda_1 > \frac{\eta }{6}
\lambda_1$. Using the former, one can validate that there exist positive constants $C'(d),C''(d)$ such that if 
$\frac{C'(d)}{\eta}\le \Omega \le \frac{C''(d)}{h}$, then $I,c,\epsilon,\alpha$ meet the conditions of
Lemma \ref{lem.real.intervals.1}. We then invoke Lemma \ref{lem.real.intervals.1} and get that
$$
  \nu \big( \left\{ t \in \tilde{I}^{-1}_1 : 
    \exists k\in \Z \mbox{ such that } \left| kt - \frac{n}{\lambda^*}\lambda_1  \right|\le 3 h \lambda_1 \right\}\big)
    < \frac{1}{4} |\tilde{I}^{-1}_1|.
$$  
Then
$$
  \nu \big( \left\{ t \in I^{-1}_1 : 
    \exists k\in \Z \mbox{ such that } \left| kt - \frac{n}{\lambda^*} \right|\le 3 h  \right\}\big)
    < \frac{1}{4} |I^{-1}_1|.
$$ 
By the above and using Lemma \eqref{lem.real.intervals.2}
\begin{equation}\label{eq.apen.case.2.cluster.measure}
	\nu \big( \left\{ \lambda \in I_1 : 
    \exists k\in \Z \mbox{ such that } \left| \frac{k}{\lambda} - \frac{n}{\lambda^*} \right|\le 3 h \right\}\big) 
  	< \frac{1}{2} |I_1|.
\end{equation}

Now for any index $j$ such that $x_j$ is a non-cluster node put $c_j = x_i +\frac{n}{\lambda^*} - x_j$.
Put $I = \tilde{I}^{-1}_1$, $c=c_j\lambda_1$, $\epsilon=2 h \lambda_1$ and $\alpha = \frac{1}{4d}$.
Then by the assumptions of this case $|c|>\frac{\eta}{6}\lambda_1$ and with this 
one can validate that there exist positive constants $C'(d),C''(d)$ such that if 
$\frac{C'(d)}{\eta}\le \Omega \le \frac{C''(d)}{h}$, then $I,c,\epsilon,\alpha$ meet the conditions of
Lemma \ref{lem.real.intervals.1}. Invoking it 
and using Lemma \eqref{lem.real.intervals.2} we have that
\begin{equation}\label{eq.apen.case.2.non.cluster.measure}
	\nu \big( \left\{ \lambda \in I_1 : 
    \exists k\in \Z \mbox{ such that } \left| \frac{k}{\lambda} - c_j \right|\le 2 h \right\}\big) 
  < \frac{1}{2d} |I_1|.
\end{equation}
Define the set
$$
  E=\bigcup_{\substack{1\le j\le d,\\ x_j\notin \xv^c}}\left\{ \lambda \in I_1: 
    \exists k\in \Z \mbox{ such that } \left| \frac{k}{\lambda} - c_j\right|\le 2h \right\}.
$$
Using the union bound and \eqref{eq.apen.case.2.non.cluster.measure}
\begin{equation}\label{eq.appen.non.cluster.nodes.tool}
  \nu(E) < \frac{1}{2} |I_1|. 
\end{equation}

Now combing \eqref{eq.apen.case.2.cluster.measure} and \eqref{eq.appen.non.cluster.nodes.tool}
we get that there exists $\lambda \in I_1$ such that for all $k \in \Z$
\begin{align*}
	 \left| \frac{k}{\lambda} - \frac{n}{\lambda^*} \right|&> 3 h, &\\
	  \left|x_i +\frac{n}{\lambda^*} -x_j - \frac{k}{\lambda}\right|&>2h, & \forall x_j \in \xv\setminus \xv^c.
\end{align*} 
Finally setting $\lambda_{\tilde{\lv}}=\lambda$ we get from the above and \eqref{eq.A.R.radious}
that $\lambda_{\tilde{\lv}}$ satisfies \eqref{eq.empty.intersection.1}.    
 
\bigskip 

{\bf Case 3:} $\frac{n}{\lambda^*}> \frac{\eta}{6}$ and $\exists y \in \xv\setminus \xv^c  
: |x_i+\frac{n}{\lambda^*}-y| \le \frac{\eta}{6}.$

\bigskip

First we note that 
since the non-cluster nodes are each separated from any other node 
by at least $\eta$, there can be at most one node 
$y \in \xv\setminus \xv^c$ such that $|x_i+\frac{n}{\lambda^*} - y| \le \frac{\eta}{6}$.    
Therefore let $j$ be the index of the non-cluster node for which we have $|x_i+\frac{n}{\lambda^*} - x_j| \le \frac{\eta}{6}$. 
By the choice of $\lambda^*$ we also have that $|x_i+\frac{n}{\lambda^*} - x_j| > (32d^4)^{-1}\frac{1}{\lambda_1}$ (see \eqref{eq.nodes.mod.lambda}). 
We conclude that 
$$(32d^4)^{-1}\frac{1}{\lambda_1} \le |x_i+\frac{n}{\lambda^*} - x_j|\le  \frac{\eta}{6},$$
and for $\Omega \le \frac{1}{96d^3 h}$ we then have that 
$$3h < |x_i+\frac{n}{\lambda^*} - x_j|\le  \frac{\eta}{6}.$$

We now invoke Proposition \ref{prop.small.c} and get that there exists an interval $I_3 \in \Lambda(\xv)$ 
of length $|I_3| = (2 d^2 \eta)^{-1}$ such that for all $\lambda \in I_3$ and for all $k \in \Z$
\begin{equation}\label{eq.I.2.property}
	\left|x_i+\frac{n}{\lambda^*} - x_j -\frac{k}{\lambda}\right| > 3 h.
\end{equation}	

Put 
$$ I_3=[\lambda_3,\lambda_3+(2 d^2 \eta)^{-1}],\tab I^{-1}_3=\left[\frac{1}{\lambda_3+(d^2 2 \eta)^{-1}},\frac{1}{\lambda_3}\right],\tab \tilde{I}^{-1}_3= 
\lambda_3 I^{-1}_3.$$

For each index $1 \le \ell \le d, \ell \ne j$  
put $c_{\ell} = x_i+\frac{n}{\lambda^*} - x_{\ell}$ 
and note that 
$$|c_{\ell}| = |x_i+\frac{n}{\lambda^*} -x_j +x_j - x_{\ell}|\ge |x_j-x_{\ell}| - |x_i+\frac{n}{\lambda^*} - x_j| \ge  \frac{5}{6}\eta.$$ 
Put
$I=\tilde{I}_3^{-1}$, $c=c_{\ell} \lambda_3$, $\epsilon = 2h\lambda_3$ and $\alpha = \frac{1}{2d}$.
Then with the above $|c|\ge  \frac{5}{6}\eta \lambda_3$
and then following similar computations as in the previous cases (see cases 1,2), one can validate that
$I,c,\epsilon,\alpha$ meet the conditions of Lemma \ref{lem.real.intervals.1} for 
$\frac{C'}{\eta} \le \Omega \le \frac{C''}{h}$ where $C',C''$ are constants depending only on $d$.
Invoking Lemma \ref{lem.real.intervals.1}   
with $I,c,\epsilon,\alpha$ we get that
$$
  \nu \big( \left\{ t \in \tilde{I}^{-1}_3 : 
    \exists k\in \Z \mbox{ such that } \left| kt - c_{\ell}\lambda_3 \right|\le 2 h \lambda_3 \right\}\big)
    < \frac{1}{2d} |\tilde{I}^{-1}_3|.
$$  
Then
$$
  \nu \big( \left\{ t \in I^{-1}_3 : 
    \exists k\in \Z \mbox{ such that } \left| kt - c_{\ell} \right|\le 2 h  \right\}\big)
    < \frac{1}{2d} |I^{-1}_3|.
$$ 
By the above and using Lemma \eqref{lem.real.intervals.2}
\begin{equation}\label{eq.apen.case.3.all.nodes.measure}
	\nu \big( \left\{ \lambda \in I_3 : 
    \exists k\in \Z \mbox{ such that } \left|\frac{k}{\lambda}-c_{\ell}\right|\le 2 h
    \right\}\big) < \frac{1}{d} |I_3|.
\end{equation}

Define the set
$$
  E=\bigcup_{\substack{1\le \ell \le d,\ \ell \ne j}}\left\{ \lambda \in I_3: 
    \exists k\in \Z \mbox{ such that } \left| \frac{k}{\lambda} - c_{\ell}\right|\le 2h \right\}.
$$
Using the union bound and \eqref{eq.apen.case.3.all.nodes.measure}
$$
  \nu(E) < |I_3|. 
$$
We conclude from the above that there exists $\lambda \in I_3$ such that
for all $k \in \Z$ and for any index $1 \le \ell \le d, \ell\ne j$,
\begin{equation}\label{eq.appen.case.3.all.but.one}
	\left|x_i+\frac{n}{\lambda^*}-x_{\ell} -\frac{k}{\lambda}\right|>2h.
\end{equation}
Put $\lambda_{\tilde{\lv}}=\lambda$. 
Recall that $I_3$ satisfies \eqref{eq.I.2.property}. 
Then with \eqref{eq.I.2.property} and \eqref{eq.appen.case.3.all.but.one} 
$\lambda_{\tilde{\lv}}$ satisfies that for
all $k \in \Z$ and for any index $1 \le \ell \le d$ $$
	\left|x_i+\frac{n}{\lambda^*}-x_{\ell} -\frac{k}{\lambda_{\tilde{\lv}}}\right|>2h.
$$    
Using the above and \eqref{eq.A.R.radious} we get that 
that $\lambda_{\tilde{\lv}}$ satisfies
\eqref{eq.empty.intersection.1}. \qed

\bigskip

We now prove the intermediate claims: Lemma \ref{lem.real.intervals.1}, Lemma \ref{lem.real.intervals.2} and
Proposition \ref{prop.small.c}.  	

\smallskip

\begin{proof}[Proof of Lemma \ref{lem.real.intervals.1}]
  Let $a,\epsilon,\alpha,c$ and $I=[a,1]$ as specified in Lemma \ref{lem.real.intervals.1}.
  Without loss of generality we assume that $c>0$, consequently it is sufficient
  to prove that 
  $$					
  \nu \big( \left\{x\in I : 
    \exists k\in \mathbb{N} \mbox{ such that } \left| kx - c\right|\le \epsilon \right\}\big) 
  <\alpha |I|.
  $$
  If $0 < c < 2 $ then one can verify that 
  $$
  \nu \big( \left\{x\in I :\exists k\in \mathbb{N} \mbox{ such that } \left| kx - c\right|\le \epsilon \right\}\big)
  \le 2\epsilon.
  $$ 
  Then under this condition and with the assumption that $c \ge 8 \frac{\epsilon}{\alpha |I|}$,
  we have that $2\epsilon < \alpha |I|$, therefore
  $$
  \nu \big( \left\{x\in I :\exists k\in \mathbb{N} \mbox{ such that } \left| kx - c\right|\le \epsilon \right\}\big)
  \le 2\epsilon < \alpha |I|.
  $$
  We now prove the case $c\ge 2$.
  
  \smallskip 
  
  Let $N \in \mathbb{N}$ be the unique integer such that
  \begin{equation}\label{eq.lemma.intervals.1.1}
    \frac{c}{\lfloor c \rfloor + N}\le a < \frac{c}{\lfloor c \rfloor + N-1}.
  \end{equation}
  Then 
  \begin{align}\label{eq.lemma.intervals.1.2}
    \nu \big( \left\{x\in I : 
    \exists k\in \Z \mbox{ such that } \left| kx - c\right|\le \epsilon \right\}\big) \le 
    \sum_{k=0}^{N} \frac{2\epsilon}{\lfloor c \rfloor +k} =
    2\epsilon \sum_{k=0}^{N} \frac{1}{\lfloor c \rfloor +k}.  				
  \end{align}
  If $N\le 2$ then with $c \ge 8 \frac{\epsilon}{\alpha |I|}$ 	
  \begin{align*}
    2\epsilon \sum_{k=0}^{N} \frac{1}{\lfloor c \rfloor +k} \le 2\epsilon \sum_{k=0}^{2} \frac{1}{\lfloor c \rfloor +k}
    < 8\frac{\epsilon}{c}\le \alpha |I|.  				
  \end{align*}
  Combining \eqref{eq.lemma.intervals.1.2} with the above proves the claim for this case.
  
  \smallskip
  
  We are left to prove the case $N \ge 3, c\ge 2$.
  
  \smallskip
  
  For $H_n$ the $n^{th}$ partial sum of the Harmonic series we have that
  $$ \log(n) + \gamma < H_n < \log(n+1) + \gamma,$$
  where $\log$ is the base $2$ logarithm.
  Then 
  \begin{align}\label{eq.lemma.intervals.1.3}
    \begin{split}
      2\epsilon \sum_{k=0}^{N} \frac{1}{\lfloor c \rfloor +k} & \le 
      2\epsilon\left( \log(\lfloor c \rfloor +N + 1) - \log(\lfloor c \rfloor-1)\right) \\
      & = 2\epsilon \log\left(\frac{\lfloor c \rfloor +N + 1}{\lfloor c \rfloor-1}\right)\\
      & = 2\epsilon \log\left(1+\frac{N+2}{\lfloor c \rfloor-1}\right).		
    \end{split}
  \end{align}
  Using \eqref{eq.lemma.intervals.1.1} and since by assumption $a\ge \frac{1}{2}$ we
  have that 
  \begin{equation}\label{eq.lemma.intervals.1.4}
    N \le \lfloor c \rfloor + 2.				
  \end{equation}
  Then by \eqref{eq.lemma.intervals.1.1} and \eqref{eq.lemma.intervals.1.4} (and assuming $N\ge 3$, $c\ge 2$)
  \begin{equation}\label{eq.lemma.intervals.1.5}
    |I| = 1-a \ge \frac{N-2}{\lfloor c \rfloor + N-1}\ge \frac{N-2}{2\lfloor c \rfloor+1} \ge 
    \frac{1}{5}\frac{(N+2)}{2\lfloor c \rfloor+1} \ge \frac{1}{25}\frac{(N+2)}{\lfloor c \rfloor-1}.				
  \end{equation}
  Inserting \eqref{eq.lemma.intervals.1.5} into \eqref{eq.lemma.intervals.1.3} 	
  and using the assumption that $100\epsilon \le \alpha$
  \begin{align}\label{eq.lemma.intervals.1.6}
    \begin{split}
      2\epsilon \log\left(1+\frac{N+2}{\lfloor c \rfloor-1}\right) & \le
      2\epsilon \log\left(1+25|I|\right) \\ 					
      & = 2\epsilon \log(e)\ln\left(1+25|I|\right)\\
      & < 100\epsilon|I|\\
      & \le \alpha |I|,
    \end{split}
  \end{align} 
  which then proves the claim using \eqref{eq.lemma.intervals.1.2} and \eqref{eq.lemma.intervals.1.3}.
  
  This completes the proof of Lemma \ref{lem.real.intervals.1}. 
\end{proof}

\begin{proof}[Proof of Lemma \ref{lem.real.intervals.2}]
  For any sub-interval $[c,d] \subseteq I$ we have that
  \begin{equation}
    \frac{\nu\left([c,d]\right)}{\nu(I)}=\frac{d-c}{b-a}=\frac{cd}{ab}\frac{\frac{1}{c}-\frac{1}{d}}{\frac{1}{a}-\frac{1}{b}}\le \frac{b}{a}
    \frac{\nu(\left[\frac{1}{d},\frac{1}{c}\right])}{\nu(I^{-1})}.
  \end{equation}
  Using the above
  $$\frac{\nu(S)}{\nu(I)} = \sum_{i}\frac{\nu([a_i,b_i])}{\nu(I)} \le 
  \frac{b}{a} \sum_{i}\frac{\nu([\frac{1}{b_i},\frac{1}{a_i}])}{\nu(I^{-1})}
  =\frac{b}{a}\nu(S^{-1}).$$

  This completes the proof of Lemma \ref{lem.real.intervals.2}.
\end{proof}	

\begin{proof}[Proof of Proposition \ref{prop.small.c}]
	Without loss of generality assume that $c>0$ and put $T=c\lambda_1$.
	
	We will use the following inequality repeatably below. 
	For each $k \ge 0$ and $0\le \alpha \le \lambda_1$ we have 
	\begin{equation}\label{eq.alpha.slow.down}
		\frac{k\alpha}{2\lambda_1^2}\le k\left(\frac{1}{\lambda_1}-\frac{1}{\lambda_1+\alpha}\right)\le \frac{k\alpha}{\lambda_1^2}.
	\end{equation}
	
	Put $\beta = T-\lfloor T \rfloor$ and consider the following cases: 
	
	\smallskip
	
	{\bfseries Case 1}: $\frac{1}{8} \le \beta \le \frac{7}{8}$.
	
	We show that in this case $I=I_1 \subset \Lambda(\xv)$
        satisfies \eqref{eq.small.c} provided that
        $\Omega h < \frac{d}{96}$ and $\Omega \ge \frac{4}{d\eta}$. To
        see this recall that
        $I_1 = [\lambda_1,\lambda_1+(2d^2\eta)^{-1}]$.  Put
        $\lambda(\alpha)=\lambda_1 + \alpha$,
        $0 \le \alpha \le (2d^2\eta)^{-1}$.
        We have that for each integer $k \le \lfloor T \rfloor$
	$$\left|c- \frac{k}{\lambda(\alpha)} \right|=\frac{T}{\lambda_1}- \frac{k}{\lambda(\alpha)}\ge \frac{\beta}{\lambda_1}
	\ge \frac{1}{8\lambda_1}.$$ 
	On the other hand, for each integer $k \ge \lceil T \rceil$ 
	\begin{align}\label{eq.large.T}
		\begin{split}
			\left|c- \frac{k}{\lambda(\alpha)} \right| &\ge \frac{k-T}{\lambda_1}-k\left(\frac{1}{\lambda_1}-
			\frac{1}{\lambda(\alpha)}\right)\\ &\ge  \frac{k-T}{\lambda_1} - \frac{k\alpha}{\lambda_1^2}=
			(k-T)\left(\frac{1}{\lambda_1}-\frac{\alpha}{\lambda_1^2}\right)-\frac{T\alpha}{\lambda_1^2}\\
			&\ge (1-\beta) \left(\frac{1}{\lambda_1}-\frac{\alpha}{\lambda_1^2}\right) -\frac{T\alpha}{\lambda_1^2} \\
			&\ge\frac{1}{8}\left(\frac{1}{\lambda_1}-\frac{\alpha}{\lambda_1^2}\right) -\frac{T\alpha}{\lambda_1^2},
		\end{split}
	\end{align}
	where in the second inequality we used \eqref{eq.alpha.slow.down}. Using 
	$\Omega \ge \frac{4}{d\eta} \Rightarrow \frac{\alpha}{\lambda_1} \le \frac{1}{2}$,
	$\frac{T}{\lambda_1} \le \frac{\eta}{6}$ and $\Omega h < \frac{d}{96}$ we have that
	$$
		\frac{1}{8}\left(\frac{1}{\lambda_1}-\frac{\alpha}{\lambda_1^2}\right) -\frac{T\alpha}{\lambda_1^2}
		\ge \frac{1}{16 \lambda_1}-\frac{1}{32 \lambda_1}=\frac{1}{32 \lambda_1} > 3h.
	$$
	We conclude from the above that for $\frac{1}{8} \le \beta \le
	\frac{7}{8}$ (and under the assumptions on $\Omega$ and $\Omega h$) $I=I_1 \subset \Lambda(\xv)$ satisfies
	\eqref{eq.small.c}.
	
	\smallskip
	
	{\bfseries Case 2}: $\beta \le \frac{1}{8}$.
	
	First if $\lfloor T \rfloor = 0 $ we show that $I=I_1 \subset \Lambda(\xv)$ satisfies \eqref{eq.small.c}
	for $\Omega h \le \frac{d}{8}$.
	For $k=0$
	$$\left|c-\frac{k}{\lambda}\right|=c > 3h.$$
	For $k>0$ and $\lambda \in I_1$   
	$$\left|c-\frac{k}{\lambda}\right| = \left|\frac{\beta}{\lambda_1}-\frac{k}{\lambda}\right|\ge 
	\frac{1}{\lambda} - \frac{\beta}{\lambda_1} \ge \frac{1}{2\lambda_1} -
	\frac{1}{8\lambda_1}=\frac{3}{8\lambda_1}>3h,$$
	where in the last inequality we used the assumption that $\Omega h \le \frac{d}{8}$. 
	
	Now assume that $\lfloor T \rfloor > 0 $ and consider the next inequalities
	\begin{align}
		T \left(\frac{1}{\lambda_1}-\frac{1}{\lambda(\alpha)}\right)&> 3h\label{eq.T.alpha.ineq.1}, \\
		\lfloor T\rfloor \left(\frac{1}{\lambda_1}-\frac{1}{\lambda(\alpha)}\right)& < \frac{1}{4
		\lambda_1}\label{eq.T.alpha.ineq.2}.
	\end{align}
	We show that if for $0\le \alpha \le \lambda_1$, $\lambda(\alpha)$ satisfies both \eqref{eq.T.alpha.ineq.1} and \eqref{eq.T.alpha.ineq.2}
	then  $\lambda(\alpha)$ satisfies \eqref{eq.small.c}, provided that $\Omega h \le \frac{d}{24}$.
	
	For any integer $k \le \lfloor T \rfloor$ we have using \eqref{eq.T.alpha.ineq.1} that
	$$ \frac{T}{\lambda_1}-\frac{k}{\lambda(\alpha)} \ge T \left(\frac{1}{\lambda_1}-\frac{1}{\lambda(\alpha)}\right) > 3h.$$
	For any integer $k > \lfloor T \rfloor$  
	\begin{align*}
		\frac{k}{\lambda(\alpha)}-\frac{T}{\lambda_1}&\ge  \frac{\lfloor T
		\rfloor}{\lambda(\alpha)}-\frac{T}{\lambda_1} + \frac{1}{\lambda(\alpha)}\\
			&\ge -\lfloor T\rfloor\left(\frac{1}{\lambda_1}-\frac{1}{\lambda(\alpha)}\right)
				-\frac{\beta}{\lambda_1}+\frac{1}{\lambda(\alpha)}\\
			&> -\frac{1}{4\lambda_1} -\frac{\beta}{\lambda_1} +\frac{1}{\lambda(\alpha)}\\
			&\ge -\frac{3}{8\lambda_1} +\frac{1}{2\lambda_1}\\
			&\ge \frac{1}{8\lambda_1}\\
			&\ge 3h,
	\end{align*}
	where in the $3^{rd}$ inequality we used \eqref{eq.T.alpha.ineq.2}, in the $4^{th}$ inequality
	we used both $\beta \le \frac{1}{8}$ and $0\le \alpha \le \lambda_1$, and 
	in last inequality we used $\Omega h \le \frac{d}{24}$.
	
	We then conclude that when $\Omega h$ is small enough, each
        $\lambda(\alpha)$ with $0\le \alpha \le \lambda_1$ which
        satisfies both \eqref{eq.T.alpha.ineq.1} and
        \eqref{eq.T.alpha.ineq.2} satisfies \eqref{eq.small.c}.  We
        now solve \eqref{eq.T.alpha.ineq.1} and
        \eqref{eq.T.alpha.ineq.2} for $\alpha$.  By
        \eqref{eq.alpha.slow.down}
        $\frac{T\alpha}{2\lambda_1^2} > 3h \Rightarrow T
        \left(\frac{1}{\lambda_1}-\frac{1}{\lambda(\alpha)}\right)> 3h
        $, then each $0\le \alpha \le \lambda_1$ such that
	$$\alpha > \frac{6\lambda_1^2 h}{T}$$
	satisfies \eqref{eq.T.alpha.ineq.1}.
	By \eqref{eq.alpha.slow.down} $\frac{\lfloor T \rfloor \alpha}{\lambda_1^2} < \frac{1}{4\lambda_1} \Rightarrow
	\lfloor T \rfloor \left(\frac{1}{\lambda_1}-\frac{1}{\lambda(\alpha)}\right) < \frac{1}{4\lambda_1}$,
	then each $0\le \alpha \le \lambda_1$ such that
	$$ \alpha < \frac{\lambda_1}{4 \lfloor T \rfloor},$$
	satisfies \eqref{eq.T.alpha.ineq.2}.	
	
	We conclude from the above that for 
	$$\alpha \in \left(\frac{6\lambda_1^2 h}{T},\frac{\lambda_1}{4 \lfloor T \rfloor}\right)= I_3,$$
	$\lambda(\alpha)$ satisfies \eqref{eq.small.c}.
					
	Now we recall that by Proposition \ref{prop.good.blowup}, every interval $I' \subset
	\left[\frac{1}{2}\frac{\Omega}{2d-1},\frac{\Omega}{2d-1}\right]$ of size $\frac{1}{\eta}$ contains a sub-interval $I$ 
	of size $(2d^2\eta)^{-1}$ such that $I \subset \Lambda(\xv)$.
	Put $I_4 = \lambda_1 + I_3$ and $I_5 = I_4 \cap \left[\frac{1}{2}\frac{\Omega}{2d-1},\frac{\Omega}{2d-1}\right]$. We
	will now validate that $|I_5| > \frac{1}{\eta}$ for $\Omega h < \frac{d}{72}$. 
	To prove that we show that 
	$$\lambda_1 + \frac{6\lambda_1^2 h}{T} + \frac{1}{\eta} < \min\left(\lambda_1+\frac{\lambda_1}{4 \lfloor T
	\rfloor},\frac{\Omega}{2d-1}\right).$$ 
	First we show that $\lambda_1 + \frac{6\lambda_1^2 h}{T} + \frac{1}{\eta} < \lambda_1+\frac{\lambda_1}{4 \lfloor T
	\rfloor}$:
	$$ \frac{\lambda_1}{4 \lfloor T \rfloor} - \frac{6\lambda_1^2
	h}{T}\ge \frac{\lambda_1}{T}\left(\frac{1}{4}-6\lambda_1h\right)\ge 
			\frac{6}{\eta} \left(\frac{1}{4}-6\lambda_1h\right)>\frac{1}{\eta},$$
	where in the penultimate inequality we used the proposition assumption that $\frac{\eta}{6}\ge
	c=\frac{T}{\lambda_1}$ and in the last inequality we used $\Omega h < \frac{d}{72}$.
	Next we show that $\lambda_1 + \frac{6\lambda_1^2 h}{T} + \frac{1}{\eta} < \frac{\Omega}{2d-1}$
	for $\Omega > \frac{5(2d-1)}{\eta}$ and $\Omega h < \frac{d}{72}$:
	\begin{equation*}
		\begin{split}
		\lambda_1 + \frac{6\lambda_1^2 h}{T} + \frac{1}{\eta}\le
		\lambda_1\left(1+6\lambda_1 h\right)+ \frac{1}{\eta}\le \frac{13}{12}\lambda_1+\frac{1}{\eta}\\
		\le \frac{13}{12}\left(\frac{\Omega}{2(2d-1)}+\frac{1}{\eta}\right)+\frac{1}{\eta}<\frac{\Omega}{2d-1}.
		\end{split}
	\end{equation*}
	We conclude that $|I_5| > \frac{1}{\eta}$ and $I_5 \subset
	\left[\frac{1}{2}\frac{\Omega}{2d-1},\frac{\Omega}{2d-1}\right]$ then by Proposition 
	\ref{prop.good.blowup} $I_5$ contains a sub-interval $I$ of size $(2d^2\eta)^{-1}$ such that $I \subset \Lambda(\xv)$.
	Since by construction $I_5$ satisfies \eqref{eq.small.c} this completes the proof of the
	case $\beta \le \frac{1}{8}$ of Proposition \eqref{prop.small.c}.
	
	\smallskip
	
	We are left to prove the case $\frac{7}{8} \le \beta$. This case is proved similarly to the 
	case $\beta \le \frac{1}{8}$. We therefore omit the proof of this case.
\end{proof}

\end{document}